\documentclass[12pt]{amsart}

\usepackage[latin1]{inputenc}
\usepackage[english]{babel}
\usepackage{color}
\usepackage{indentfirst}
\usepackage{amssymb}
\usepackage{amsthm}
\usepackage{hyperref}

\newcommand{\con}{\mathfrak c}
\newcommand{\vf}{\varphi}

\newcommand{\sub}{\subseteq}

\def\cA{{{\mathcal A}}}
\def\cB{{{\mathcal B}}}
\def\cC{{{\mathcal C}}}

\def\cG{{{\mathcal G}}}
\def\cJ{{{\mathcal J}}}

\def\cK{{{\mathcal K}}}

\def\cR{{{\mathcal R}}}

\def\AK{{\mathcal{AK}}}
\def\natnums{\mathbb N}
\def\reals{\mathbb R}
\def\R{\reals}

\def\N{\natnums}

\newcommand\eps{\ensuremath{\varepsilon}}

\newcommand{\impli}{\Rightarrow}

\newcommand{\qu}{\mathbb{Q}}
\newcommand{\erre}{\mathbb{R}}

\newcommand{\BB}{\protect{\mathcal B}}

\newcommand{\er}{\mathbb R}

\newcommand{\cf}{{\rm cf} }
\newcommand{\add}{{\rm add} }

\def\epsilon{\varepsilon}
\newcommand{\axiom}{$\mathbf{\Sigma^1_1}$\textbf{D}}

\newtheorem{theo}{Theorem}[section]
\newtheorem{lem}[theo]{Lemma}%[section]
\newtheorem{pro}[theo]{Proposition}%[section]
\newtheorem{cor}[theo]{Corollary}%[theo]
\newtheorem{defi}[theo]{Definition}%[section]
\newtheorem{rem}[theo]{Remark}
\newtheorem{exa}[theo]{Example}
\newtheorem{thm}[theo]{Theorem}
\newtheorem{prop}[theo]{Proposition}%[section]
\newtheorem{defn}[theo]{Definition}%[section]

\newtheorem{proposition}[theo]{Proposition}
\newtheorem{problem}[theo]{Problem}

\theoremstyle{definition}

\theoremstyle{remark}

\numberwithin{equation}{section}

\title[Tukey classification of some ideals on $\omega$\dots]{Tukey classification
of some ideals on $\omega$ and the lattices of weakly compact sets in Banach spaces}

\author{ A. Avil\'{e}s}
\address{Departamento de Matem\'{a}ticas\\
Facultad de Matem\'{a}ticas\\ Universidad de Murcia\\ 30100 Espinardo, Murcia\\
Spain} \email{avileslo@um.es}

\author{G. Plebanek}
\address{Instytut Matematyczny\\ Uniwersytet Wroc\l awski\\ Pl.\ Grunwaldzki 2/4\\
50-384 Wroc\-\l aw\\ Poland} \email{grzes@math.uni.wroc.pl}

\author{J. Rodr\'{i}guez}
\address{Departamento de Matem\'{a}tica Aplicada\\
Facultad de Inform\'{a}tica\\ Universidad de Murcia\\ 30100 Espinardo, Murcia\\
Spain} \email{joserr@um.es}

\date{\today}



\thanks{A. Avil\'{e}s and J. Rodr\'{i}guez were partially supported by
the research projects MTM2011-25377 and MTM2014-54182-P funded by {\em Ministerio de Econom\'{i}a y Competitividad - FEDER}
and the research project 19275/PI/14 funded by {\em Fundaci\'{o}n S\'{e}neca - Agencia de Ciencia y Tecnolog\'{i}a
de la Regi\'{o}n de Murcia} within the framework of {\em PCTIRM 2011-2014}. G. Plebanek was partially supported by
NCN grant 2013/11/B/ST1/03596 (2014-2017)}

\begin{document}

\begin{abstract}
We study the lattice structure of the family of weakly compact subsets of the unit ball~$B_X$ of a separable Banach space~$X$,
equipped with the inclusion relation (this structure is denoted by $\cK(B_X)$) and also with the parametrized family
of ``almost inclusion'' relations $K \subseteq L+\epsilon B_X$, where $\epsilon>0$ (this structure is denoted by $\mathcal{AK}(B_X)$).
Tukey equivalence between partially ordered sets and a suitable extension to deal with $\mathcal{AK}(B_X)$ are used.
Assuming the axiom of analytic determinacy, we prove that separable Banach spaces fall into four categories, namely:
$\cK(B_X)$ is equivalent either to a singleton, or to $\omega^\omega$, or to
the family $\cK(\mathbb{Q})$ of compact subsets of the rational numbers, or
to the family $[\mathfrak{c}]^{<\omega}$ of all finite subsets of the continuum.
Also under the axiom of analytic determinacy, a similar classification of~$\mathcal{AK}(B_X)$ is obtained.
For separable Banach spaces not containing~$\ell^1$, we prove
in ZFC that $\cK(B_X) \sim \mathcal{AK}(B_X)$ are equivalent
to either $\{0\}$, $\omega^\omega$, $\mathcal{K}(\mathbb{Q})$ or $[\mathfrak{c}]^{<\omega}$.
The lattice structure of the family of all weakly null subsequences of an
unconditional basis is also studied.
\end{abstract}

\maketitle

\section{Introduction}\label{section:intro}

The purpose of this paper is to establish a classification of separable Banach spaces
according to how complicated the lattice of weakly compact subsets is. Let $\mathcal{K}(B_X)$ denote the
family of all weakly compact subsets of the unit ball~$B_X$ of a Banach space~$X$,
that we view as a partially ordered set endowed with inclusion. The way in which we measure
the complexity of $\mathcal{K}(B_X)$ is through Tukey reduction. This has become a standard way
to compare partially ordered sets, proven useful to isolate some essential features of the ordered
structure \cite{sol-tod}. Let us recall that two upwards-directed partially ordered sets
are Tukey equivalent if and only if they are order isomorphic to cofinal subsets
of some third upwards-directed partially ordered set. Our first main result is the following:

\renewcommand{\thetheo}{\Alph{theo}}

\begin{theo}[\axiom]\label{theoremA}
If $X$ is a separable Banach space, then $\mathcal{K}(B_X)$ is Tukey equivalent to one of the following
partially ordered sets:
\begin{enumerate}
\item[(i)] either to a singleton,
\item[(ii)] or to $\omega^\omega$ (ordered pointwise),
\item[(iii)] or to the family $\cK(\mathbb{Q})$ of compact subsets of the rational numbers (ordered by inclusion),
\item[(iv)] or to the family $[\mathfrak{c}]^{<\omega}$ of all finite subsets of the continuum (ordered by inclusion).
\end{enumerate}
\end{theo}

The symbol (\axiom) in this and later results means that the statement holds under the axiom of analytic determinacy
(which is consistent with ZFC if one believes in large cardinals). A reader unfamiliar with determinacy axioms
can think that, in practical terms, Theorem~\ref{theoremA} holds for any \emph{reasonable} Banach space,
not arising from any set-theoretic oddity. The case (i) corresponds to reflexivity, so the result can be interpreted
as saying that non-reflexive separable Banach spaces split into three categories, depending on three canonical patterns of
disposition in the lattice of weakly compact sets.
When $X^*$ (the dual of~$X$) is separable (for the norm topology), Theorem~\ref{theoremA} holds in ZFC without
any determinacy axiom required, and (iv) never happens. This particular case is a corollary to a result of Fremlin~\cite{fre12}, who
established the Tukey classification of the lattices of compact subsets of coanalytic metric spaces.
Our main contribution is therefore the case of non-separable dual.
In the case of separable Banach spaces not containing~$\ell^1$, the classification of
Theorem~\ref{theoremA} corresponds to the following well-studied classes of spaces:
\begin{enumerate}
\item[(i)] reflexive spaces,
\item[(ii)] non-reflexive spaces with separable dual and the PCP (point of continuity property),
\item[(iii)] spaces with separable dual that fail the PCP
\item[(iv)] spaces with non-separable dual that do not contain copies of~$\ell^1$,
\end{enumerate}
see Theorem~\ref{separabledual}. When $\ell^1$ is present, the classification of Theorem~\ref{theoremA} does not match,
to the best of our knowledge, previously studied classes. Let us stress that analytic determinacy and case~(iv) rise from
the use of a \emph{Lusin gap} dichotomy~\cite{tod-J-2} for gaps which are more complex than analytic,
which is in turn related to the validity of the open graph theorem for projective sets considered in~\cite{fen}.
The following purely combinatorial result is behind our approach to Theorem~\ref{theoremA}:

\begin{theo}[\axiom]\label{theoremNuevo}
Let $\mathcal{I}$ be an analytic family of subsets of $\omega$. Then
$\mathcal{I}^\perp$ (ordered by inclusion)
is Tukey equivalent to either $\{0\}$, $\omega$, $\omega^\omega$, $\mathcal{K}(\mathbb{Q})$ or $[\mathfrak{c}]^{<\omega}$.
\end{theo}

The information provided by $\mathcal{K}(B_X)$ can be refined by taking into consideration not only the inclusion relation,
but also the metric structure of~$X$. For this purpose we introduce the object $\mathcal{AK}(B_X)$ which consists again of the
family of weakly compact subsets of~$B_X$, but now endowed with the family of binary relations
$K\subset L + \varepsilon B_X$, parametrized by $\varepsilon>0$. We introduce a suitable notion of Tukey reduction
that allows to compare such structures among them, and to compare them with ordinary partially ordered sets.
For instance, the condition $\mathcal{AK}(B_X) \sim \omega$ is equivalent to saying that $X$ is non-reflexive and strongly
weakly compactly generated in the sense of~\cite{sch-whe} (see Theorem~\ref{swcg:2}).
We obtain the following classification result:

\begin{theo}[\axiom]\label{theoremB}
If $X$ is a separable Banach space, then:
\begin{itemize}
\item[(i)] either $\mathcal{AK}(B_X)\sim \{0\}$,
\item[(iia)] {or $\mathcal{AK}(B_X) \sim \omega$,}
\item[(iib)] {or $\mathcal{AK}(B_X) \sim \omega^\omega$,}
\item[(iii)] or $\mathcal{AK}(B_X)\sim \mathcal{K}(\mathbb{Q})$,
\item[(iv)] or $\mathcal{AK}(B_X)\sim [\mathfrak{c}]^{<\omega}$.
\end{itemize}
\end{theo}

\renewcommand{\thetheo}{\thesection.\arabic{theo}}

{We enumerate the cases in this way, because (i), (iii) and~(iv) of Theorem~\ref{theoremB}
correspond exactly to cases (i), (iii) and~(iv) of Theorem~\ref{theoremA}, while
(ii) of Theorem~\ref{theoremA} splits into cases (iia) and (iib) of Theorem~\ref{theoremB}.
This result} provides finer information on the structure of the
family of weakly compact sets, and its proof requires substantial extra effort with respect to
Theorem~\ref{theoremA}. Without assuming any determinacy axiom, we prove that for Banach spaces not containing~$\ell^1$
the structures $\mathcal{K}(B_X)$ and $\mathcal{AK}(B_X)$ are equivalent (Theorem~\ref{separabledual}).

However, we do not know if, consistently,  there is a non-reflexive separable Banach space $X$ (necessarily with non-separable dual) such that
$\cK(B_X)$ is neither Tukey equivalent to $\omega^\omega$ nor to $\cK(\qu)$ nor to $[\con]^{<\omega}$
(and the same question for $\mathcal{AK}(B_X)$). A possible example in the absence of analytic determinacy is
provided in Theorem~\ref{omega1sequence} by the construction of a peculiar unconditional basis from a coanalytic set of cardinality~$\aleph_1$.

The structure of the paper is as follows. In Section~\ref{section:KBX} we review some basic
facts on $\mathcal{K}(B_X)$. In Section~\ref{section:AKBX} we introduce the main object of our study, the structure
$\mathcal{AK}(B_X)$, and analyse  its basic properties, including the connection with strongly weakly compactly generated spaces;
some illustrating examples are given as well.
Section~\ref{section:KBXclassification} is devoted to the proofs of Theorems~\ref{theoremA} and~\ref{theoremNuevo}, while
Section~\ref{section:atclassification} deals with the proof of Theorem~\ref{theoremB}.
Once these general classification results are established, in the next two sections we try to identify Banach spaces~$X$
for which  $\mathcal{AK}(B_X)$ is equivalent to some of our five basic posets
($\{0\}$, $\omega$, $\omega^\omega$, $\mathcal{K}(\mathbb{Q})$ and $[\mathfrak{c}]^{<\omega}$)
in the following cases: spaces without copies of $\ell^1$ in Section~\ref{section:nol1} and spaces
with unconditional basis in Section~\ref{section:unconditional}. Finally, Section~\ref{section:Problems}
contains some open problems.

\subsection*{Terminology}
Throughout this paper $X$ is a (real) Banach space. The weak topology on~$X$ is denoted by~$w$ and the weak$^*$ topology
on~$X^*$ is denoted by~$w^*$. The norm on~$X$ is denoted by $\|\cdot\|$ or $\|\cdot\|_X$ if needed explicitly. By a subspace
of $X$ we mean a closed linear subspace. Given a family $\{X_i\}_{i\in I}$ of Banach spaces, the symbols $(\bigoplus_{i\in I}X_i)_{c_0}$ and
$(\bigoplus_{i\in I}X_i)_{\ell^p}$ stand for their $c_0$-sum and $\ell^p$-sum ($1\leq p \leq \infty$), respectively.

The set of all natural numbers (identified as the first infinite ordinal) is denoted by~$\omega=\{0,1,\dots\}$,
while we write $\N=\{1,2,\dots\}$. Given a set~$S$, we denote by~$[S]^{<\omega}$ the set of all finite subsets of~$S$,
the cardinality of~$S$ is denoted by~$|S|$ and we write $\mathcal{P}(S)$ for the power set of~$S$.

The Cantor set (the set of all infinite sequences of~$0$'s and $1$'s) is denoted by~$2^\omega$ and
we write $2^{<\omega}$ for the {\em dyadic tree}, that is, the set of all finite sequences of $0$'s and $1$'s
(the empty sequence is included here).
Given $t,s\in 2^{<\omega} \cup 2^\omega$, we write $t\sqsubseteq s$ if $s$ extends~$t$.
Given $\sigma\in 2^\omega$ and $m<\omega$, we write
$\sigma|_m$ to denote the unique element of~$2^{<\omega}$ such that ${\rm length}(\sigma|_m)=m$
and $\sigma|_m \sqsubseteq \sigma$. The concatenation of
$t,s\in 2^{<\omega}$ is denoted by $t\smallfrown s$. For every $t=(t_0,\dots,t_n)\in 2^{<\omega}$ we define $t\smallfrown {\bf 0}\in 2^{\omega}$
by $(t_0,\dots,t_n,0,0,\dots)$.

\section{The lattice $\cK(B_X)$}\label{section:KBX}

We first recall the concept of Tukey reduction between arbitrary binary relations. A suitable reference
for the basic Tukey ordering theory is~\cite{fre13}.

\begin{defn}\label{posets:1}
Let $(U,R)$ and $(V,S)$ be two sets equipped with binary relations (so $R\sub U\times U$ and $S\sub V\times V)$.
\begin{itemize}
\item A function $f:U\to V$ is said to be a {\em Tukey function} if for every $v_0\in V$ there is $u_0\in U$ such that,
for every $u\in U$, the following implication holds:
$$
	(f(u),v_0)\in S \ \Longrightarrow \ (u,u_0)\in R.
$$
\item $(U,R)$ is said to be {\em Tukey reducible to} $(V,S)$ if there is a Tukey function $f:U\to V$. In this case,
we write $(U,R)\preceq (V,S)$ or simply $U \preceq V$.
\item $(U,R)$ and $(V,S)$ are said to be {\em Tukey equivalent} if both $(U,R) \preceq (V,S)$ and $(V,S) \preceq (U,R)$.
In this case, we write $(U,R)\sim (V,S)$ or simply $U \sim V$.
\item We write $(U,R)\prec (V,S)$, or simply $U \prec V$, whenever $(U,R)\preceq (V,S)$ but $(U,R)$ and $(V,S)$ are not Tukey equivalent.
\end{itemize}
\end{defn}

Typically $R$ and $S$ are partial orders on the corresponding sets; note that in such a case a function
$f:U\to V$ is Tukey if and only if the preimage of every bounded above subset of~$V$ is bounded above in~$U$.

If $P$ is a partially ordered set ({\em poset} for short), then $\cf(P)$ denotes its cofinality (i.e. the least cardinality of a cofinal subset of~$P$) and
$\add_\omega(P)$ is the least cardinality of a set $A\sub P$ which is not $\sigma$-bounded (i.e. $A$ cannot
be written as the union of countably many bounded above sets); we use the convention $\add_\omega(P)=\infty$ whenever $P$ is $\sigma$-bounded.
If $Q$ is another partially ordered set and $P\preceq Q$, then
$\add_\omega(P)\ge \add_\omega(Q)$ and $\cf(P)\le\cf(Q)$, see \cite[Theorem~1J]{fre13}.

Given any topological space $E$, we write $\cK(E)$ for the family of all compact subsets of~$E$. In the sequel we
always assume that $\cK(E)$ is equipped with the relation of inclusion, i.e. we consider the partially ordered set $(\cK(E),\sub)$.

It is time to present the five canonical partial orders that will show up once and again along this paper:
\begin{itemize}
\item $\{0\}$, a singleton.
\item $\omega = \{0,1,2,\ldots\}$, endowed with its natural order.
\item $\omega^\omega$, the set of all sequences of natural numbers $\{p_n\}_{n<\omega}$, endowed with the pointwise order (i.e. $\{p_n\}_{n<\omega}\leq \{q_n\}_{n<\omega}$ if and only if $p_n\leq q_n$ for all $n<\omega$).
\item $\mathcal{K}(\mathbb{Q})$, endowed with the inclusion order.
\item $[\mathfrak{c}]^{<\omega}$, endowed with the inclusion order.
\end{itemize}
The five posets are enumerated in increasing Tukey-complexity, that is
$$
	\{0\} \prec \omega \prec \omega^\omega \prec \mathcal{K}(\mathbb{Q}) \prec [\mathfrak{c}]^{<\omega}
$$
(see \cite{fre12}). The metric space $\qu$ (the space of rational numbers) is
homeomorphic to the subset of~$2^\omega$ made up of all eventually zero sequences, and from
now on we identify both spaces. Note that $\qu$ is Borel (hence coanalytic) in~$2^\omega$, but $\qu$
is not a Polish space. The coefficients $\cf(\cdot)$ and $\add_\omega(\cdot)$ of the partial orders above are as follows:
\[
	\add_\omega([\con]^{<\omega})=\omega_1,\qquad \add_\omega(\omega^\omega)=\add_\omega(\cK(\qu))={\mathfrak b},
\]
\[
	\cf(\omega^\omega)=\cf(\cK(\qu))={\mathfrak d},\qquad  \cf([\con]^{<\omega})=\con,
\]
see \cite{vanDouwen} and \cite[Theorem~16(c)]{fre12} for the case of~$\cK(\qu)$.

Our starting point is Fremlin's classification of $\mathcal{K}(E)$ when $E$ is a separable metric space
which is coanalytic in some Polish space, see~\cite[Theorem~15]{fre12}.

\begin{thm}[Fremlin]\label{Fr91classification}
Let $E$ be a separable metric space which is coanalytic in some Polish space. Then:
\begin{enumerate}
\item[(i)] $\mathcal{K}(E)\sim\{0\}$ if $E$ is compact.
\item[(ii)] $\mathcal{K}(E)\sim\omega$ if $E$ is locally compact, not compact.
\item[(iii)] $\mathcal{K}(E)\sim\omega^\omega$ if $E$ is Polish, not locally compact.
\item[(iv)] $\mathcal{K}(E)\sim\mathcal{K}(\mathbb{Q})$ if $E$ is not Polish.
\end{enumerate}
\end{thm}

The Cartesian product of any family of partially ordered sets is endowed with
the coordinatewise order unless otherwise specified.
We include a short proof of the following known fact since we did not find a suitable reference for it.

\begin{lem}\label{lem:TukeyKQ}
$\cK(\qu)^\omega \sim \cK(\qu)$.
\end{lem}
\begin{proof}
$\qu^\omega$ is a separable metrizable space which is Borel (hence coanalytic)
in~$(2^\omega)^\omega$, but $\qu^\omega$ is
not Polish. Hence Theorem~\ref{Fr91classification} yields $\cK(\qu^\omega) \sim \cK(\qu)$.
On the other hand, we clearly have $ \cK(\qu)\preceq \cK(\qu)^\omega$. Finally, it is easy to check
that the mapping
\[
	f: \cK(\qu)^\omega \to \cK(\qu^\omega), \quad
	f((K_n)):=\prod_{n<\omega}K_n,
\]
is Tukey, so $\cK(\qu)^\omega\preceq \cK(\qu^\omega)$. It follows that $\cK(\qu)^\omega \sim \cK(\qu)$.
\end{proof}

\begin{rem}\label{rem:CardinalTukey}
If $P$ is any upwards directed partially ordered set, then
the mapping $f: P \to [P]^{<\omega}$ given by $f(p):=\{p\}$ is Tukey and so $P \preceq [P]^{<\omega}$.
\end{rem}

Let us turn to the Banach space setting. We consider the Banach space~$X$ equipped with its weak topology. Thus,
for any $A\sub X$, the symbol $\cK(A)$ denotes the family of all weakly compact subsets of~$A$ ordered by inclusion.

Proposition~\ref{KXKBX} below explains that there is essentially no difference between
considering~$\cK(X)$ and $\cK(B_X)$.
It also establishes that the lowest possible Tukey equivalence class for $\mathcal{K}(B_X)$, after $\{0\}$, is
that of $\omega^\omega$. In particular, this excludes $\omega$.
The proof of the Tukey reduction $\omega^\omega \preceq \cK(B_X)$ in~(iii)
is a simple adaptation of \cite[Lemma~11]{fre12}.

\begin{proposition}\label{KXKBX}
\mbox{ }
\begin{enumerate}
\item[(i)] $\cK(X)\sim \cK(B_X)\times\omega$.
\item[(ii)] If $X$ is reflexive, then $\cK(B_X)\sim \{0\}$ and $\cK(X) \sim \omega$.
\item[(iii)] If $X$ is not reflexive, then $\omega^\omega \preceq \cK(B_X) \sim \cK(X)$.
\end{enumerate}
\end{proposition}
\begin{proof}
(i). Every $K\in\cK(X)$ is bounded; let $n_K$ denote the least $n\in \N$ such that $K\sub nB_X$.
The mapping $\tau:\cK(X)\to \cK(B_X)\times\omega$, where
\[
	\tau(K):=\left(\frac{1}{n_K}K, n_K\right),
\]
is Tukey. Indeed, take any $(L,m)\in \cK(B_X)\times\omega$ and
define $\widehat{L}:=m \cdot \overline{{\rm aco}}(L) \in \cK(X)$ (note that
the closed absolutely convex hull $\overline{{\rm aco}}(L)$ of~$L$ is weakly
compact by Krein's theorem). If $K\in \cK(X)$ satisfies
$\tau(K)\leq (L,m)$, then we have $n_K\le m$ and $(1/n_K)K\sub L$, hence
$K\sub n_K L\sub \widehat{L}$.

On the other hand, to see that $\cK(B_X)\times\omega \preceq \cK(X)$ we can take the mapping
$\tau': \cK(B_X)\times\omega \to \cK(X)$ given by
$$
	\tau'(K,n):=K\cup\{nx_0\},
$$
where $x_0$ is a fixed norm one vector in~$X$. We claim that $\tau'$ is Tukey. Indeed, given
any $L\in \cK(X)$, we choose $n_0<\omega$ large enough such that $\sup_{x\in L}\|x\|\leq n_0$ and we consider
$(L\cap B_X,n_0)\in \cK(B_X)\times \omega$. Clearly, if $(K,n)\in \cK(B_X)\times\omega$ satisfies
$\tau'(K,n) = K\cup\{nx_0\} \sub L$, then $(K,n) \leq (L\cap B_X,n_0)$.

(ii). If $X$ is reflexive, then $B_X$ is weakly compact and so $\mathcal{K}(B_X)\sim \{0\}$. By~(i) we also have $\cK(X) \sim \omega$.

(iii).
If $X$ is not reflexive, then $B_X$ is not weakly compact, so
there is a sequence $(x_n)$ in~$B_X$ without weakly convergent subsequences.
Define $f:\cK(B_X)\times\omega \to \cK(B_X)$ by $f(K,n):=K\cup\{x_0,\dots,x_n\}$. Then $f$ is a Tukey map. Indeed, take any
$L\in\cK(B_X)$ and choose $n_0< \omega$ such that $x_n\not\in L$ for all $n\ge n_0$.
Clearly, if $(K,n) \in \cK(B_X) \times \omega$ satisfies $f(K,n)\sub L$, then $(K,n) \leq (L,n_0)$.
Hence $\cK(B_X)\times\omega \preceq \cK(B_X)$. Bearing in mind~(i), we conclude that
$\cK(X) \sim \cK(B_X)$.

For each $n<\omega$ the ball $\frac{1}{n+1}B_X$ is not weakly sequentially compact, so we may choose
a sequence $(x_{ni})_i$ in $\frac{1}{n+1}B_X$ without weakly convergent subsequences.
We claim that for every $\vf\in \omega^\omega$ the set
$$
	\tau(\vf):=\{x_{ni}:\, i\le\vf(n)\} \cup \{0\}
$$
is norm compact. Indeed, let $(y_k)$ be a sequence in~$\tau(\vf)$. If $(y_k)$ is not eventually~$0$, we
can find a subsequence, not relabeled, of the form $y_k=x_{n_k i_k}$ where
$i_k\le\vf(n_k)$. Now there are two possibilities:
\begin{itemize}
\item There is a further subsequence $(y_{k_j})$ such that $n_{k_j} < n_{k_{j+1}}$ for every~$j<\omega$. Then
$\|y_{n_{k_j}}\|=\|x_{n_{k_j} i_{k_j}}\|\leq \frac{1}{n_{k_j}+1} \to 0$ as $j\to\infty$.
\item There exist a further subsequence $(y_{k_j})$ and $n<\omega$ such that $n_{k_j} = n$
for every~$j<\omega$. Since $i_{k_j}\le\vf(n)$ for every $j<\omega$, the sequence $(y_{k_j})=(x_{n i_{k_j}})$ admits
a constant subsequence.
\end{itemize}
This proves that $\tau(\vf)$ is norm compact.

We now check that the map $\tau:\omega^\omega\to \cK(B_X)$ is Tukey. To this end,
fix $L\in \cK(B_X)$ and define $\vf_L\in \omega^\omega$ by $\vf_L(n):=\max\{i: x_{ni}\in L\}$.
Take any $\vf\in \omega^\omega$ such that $\tau(\vf)\sub L$. Clearly, for every $n<\omega$ we have
$\vf(n)\le\vf_L(n)$, hence $\vf\le\vf_L$.
\end{proof}

\begin{rem}\label{rem:Ksubspaces}
If $Y \sub X$ is a subspace, then $\cK(B_Y) \preceq \cK(B_X)$.
\end{rem}
\begin{proof}
The mapping $\tau: \cK(B_Y) \to \cK(B_X)$ given by $\tau(K):=K$ is Tukey, because
$L\cap B_Y \in \cK(B_Y)$ for every $L\in \cK(B_X)$.
\end{proof}

The space~$X$ is said to have the {\em Point of Continuity Property} ({\em PCP} for short) if,
for every weakly closed bounded set $A\sub X$, the identity mapping on~$A$ has at least one
point of weak-to-norm continuity. For instance, every Banach space with the Radon-Nikod\'{y}m property
has the PCP (see e.g. \cite[Corollary~3.14]{edg-whe}). In \cite[Theorem~A]{edg-whe} it was proved
that $(B_X,w)$ is a Polish space if and only if $X$ has the PCP and $X^*$ is separable.
This equivalence and Theorem~\ref{Fr91classification} yield the following result.

\begin{pro}\label{posets:3}
Suppose $X^*$ is separable. Then:
\begin{enumerate}
\item[(i)] $\mathcal{K}(B_X) \sim \{0\}$ if $X$ is reflexive.
\item[(ii)] $\mathcal{K}(B_X)\sim \omega^\omega$ if $X$ is not reflexive and has the PCP.
\item[(iii)] $\mathcal{K}(B_X)\sim \mathcal{K}(\mathbb{Q})$ if $X$ does not have the PCP.
\end{enumerate}
\end{pro}
\begin{proof}
(i) is clear and does not require the separability of~$X^*$.
On the other hand, since $X^*$ is separable, $(B_X,w)$ is separable metrizable
and $(B_{X^{**}}, w^*)$ is a compact metrizable space (hence a Polish space).
Let $\{x_k:k<\omega\}$ be a norm dense subset of~$B_{X}$. Then
\[
	B_X=\bigcap_{n<\omega}\bigcup_{k<\omega}
	\left(x_k+\frac{1}{n+1} B_{X^{**}}\right)\cap B_{X^{**}},
\]
so $B_X$ is an $\mathcal{F}_{\sigma\delta}$ subset of $(B_{X^{**}},w^*)$. In particular,
$(B_X,w)$ is coanalytic in $(B_{X^{**}}, w^*)$.

By Theorem~\ref{Fr91classification}, we have $\cK(B_X)\sim\omega^\omega$ if and only if $(B_X,w)$ is Polish but not locally compact,
while $\cK(B_X)\sim\cK(\mathbb{Q})$ if and only if $(B_X,w)$ is not Polish.
The possibility that $\mathcal{K}(B_X)\sim \omega$ is excluded by Proposition~\ref{KXKBX}(iii).
The conclusion now follows from the aforementioned \cite[Theorem~A]{edg-whe}.
\end{proof}

In view of Proposition~\ref{posets:3}, we have $\cK(B_{c_0}) \sim \cK(\mathbb{Q})$
(as the space $c_0$ fails the PCP, see e.g. \cite[Example~3.3]{edg-whe}).
Bearing in mind Remark~\ref{rem:Ksubspaces}, we get the following:

\begin{cor}\label{cor:Kc0}
If $X$ contains an isomorphic copy of~$c_0$, then $\cK(\mathbb{Q}) \preceq \cK(B_X)$.
\end{cor}

The converse is not valid in general: there exist Banach spaces with separable dual, not containing $c_0$,
and failing the PCP (see \cite[Section~IV]{gho-alt-2}).

The classification of Proposition~\ref{posets:3}
no longer holds without the separability assumption on~$X^*$, as we show in Example~\ref{posets:4}
below, which turns out to be a particular case of a result in Section~\ref{section:AKBX} (Corollary~\ref{as:4-cor}(iv)).
The space $L^1[0,1]$ fails the PCP (see e.g. \cite[Section~6]{edg-whe}) and its reflexive subspaces
are precisely those having separable dual or, equivalently, those containing no copy~$\ell^1$
(see e.g. \cite[p.~94]{die-J}).

\begin{exa}\label{posets:4}
Let $X$ be a non-reflexive subspace of~$L^1[0,1]$. Then $\cK(B_X)\sim\omega^\omega$.
\end{exa}
\begin{proof} The reduction $\omega^\omega\preceq \cK(B_X)$ follows from Proposition~\ref{KXKBX}(iii).
We next prove the reduction $\cK(B_X) \preceq \omega^\omega$.

We denote by~$\lambda$ the Lebesgue measure on the Borel $\sigma$-algebra $\mathcal{B}$ of~$[0,1]$.
For a given $f\in L^1[0,1]$ and $n<\omega$ we denote by $o(f,n)$ the least $k<\omega$ such that, for every $B\in \mathcal{B}$,
the following implication holds:
\[
	\mbox{if } \lambda(B)\le \frac{1}{k+1} \ \mbox{ then } \left| \int_{B} f\, d\lambda\right|\le \frac{1}{n+1}.
\]
The classical Dunford-Pettis criterion (see e.g. \cite[Theorem~5.2.9]{alb-kal}) states that a bounded set
$F\sub L^1[0,1]$ is relatively weakly compact if and only if it is uniformly integrable, that is,
$\{o(f,\cdot):f\in F\}$ is bounded above in~$\omega^\omega$
(note that $\int_{A} |f|\, d\lambda\leq 2\sup\{\left| \int_{B} f\, d\lambda\right|: \, B \sub A, \, B\in\mathcal{B}\}$
for every $f\in L^1[0,1]$ and $A\in \mathcal{B}$). For every $K\in \cK(B_X)$ we fix
$\vf_K\in \omega^\omega$ such that $o(f,\cdot) \leq \vf_K$ for all $f\in K$. We claim that the mapping $\cK(B_X) \to \omega^\omega$
given by $K \mapsto \vf_K$ is Tukey. Indeed, fix $\psi\in \omega^\omega$ and define
\begin{multline*}
	K:=\{f\in B_X: \, o(f,\cdot) \leq \psi\}= \\ =
	\bigcap_{n<\omega} \Bigl\{f\in B_X: \, \Bigl|\int_B f \, d\lambda\Bigr| \leq \frac{1}{n+1} \
	\mbox{ whenever }\lambda(B) \leq \frac{1}{\psi(n)+1} \Bigr\}.
\end{multline*}
Since $K$ is weakly closed and $o(f,\cdot) \leq \psi$ for every~$f\in K$, it follows that $K\in \cK(B_X)$.
Now, if $L\in \cK(B_X)$ satisfies $\vf_L\leq\psi$, then $L \sub K$.
This shows that $\cK(B_X)\preceq \omega^\omega$.
\end{proof}

\section{The asymptotic structure $\mathcal{AK}(B_X)$}\label{section:AKBX}

Our inspiration for introducing the notion of asymptotic structure
(Definition~\ref{at:1} below) is the following class of Banach spaces.

\begin{defn}
A Banach space $X$ is called {\em strongly weakly compactly generated} ({\em SWCG} for short) if there exists
a weakly compact set $K\sub X$ such that for every $\varepsilon>0$ and every
weakly compact set $L\sub X$ there is $n<\omega$ such that $L\sub nK + \varepsilon B_X$.
In this case, we say that $K$ {\em strongly generates}~$X$.
\end{defn}

This is a well studied class of Banach spaces that includes reflexive spaces, separable spaces with the Schur property
and the space $L^1(\mu)$ for any probability measure~$\mu$, but excludes $c_0$ and $C[0,1]$ among others.
For more information on SWCG spaces, we refer the reader to \cite[Section~6.4]{fab-alt-JJ}
and \cite{fab-mon-ziz,kam-mer,kam-mer2,mer-sta-2,sch-whe}.

Observe that the notion of SWCG space cannot be defined in terms of the partially ordered set $\mathcal{K}(X)$,
as it involves the relations of ``almost inclusion'' $L\subset K + \varepsilon B_X$ between weakly compact sets,
not just the inclusion relation $L\subset K$. The following definitions are intended to
develop a Tukey theory that takes into account these ``almost inclusion'' relations.

\begin{defn}\label{at:1}
An {\em asymptotic structure} is a set $P$ endowed with a
family of binary relations $\{\leq_t\}_{t>0}$ satisfying:
\begin{itemize}
\item[(i)] $p \leq_t p$ for every $p\in P$ and every $t>0$;
\item[(ii)] if $t<s$ and $p_1\leq_t p_2$, then $p_1 \leq_s p_2$;
\item[(iii)] the binary relation $\bigcap_{t>0}\leq_t$ is a partial order on $P$.
\end{itemize}
\end{defn}

Note that every poset $(P,\leq)$ can be viewed naturally
as an asymptotic structure by declaring $\leq_t := \leq$ for all $t>0$.

\begin{defn}\label{at:3}
Let $P$ and $Q$ be asymptotic structures.
\begin{enumerate}
\item[(i)] A Tukey reduction $f:P\Rightarrow Q$ is a family of functions $\{f_\varepsilon : P \to Q \}_{\varepsilon>0}$
such that for every $\varepsilon>0$ there is $\delta>0$ such that
\[
	f_\eps: (P,\leq_\eps)\to (Q,\leq_\delta) \quad\mbox{is Tukey},
\]
i.e. for every $q_0\in Q$ there is $p_0\in P$ such that $p\leq_\varepsilon p_0$ whenever $f_\varepsilon(p)\leq_\delta q_0$.
\item[(ii)] We say that $P$ is {\em Tukey reducible} to~$Q$ (and we write $P \preceq Q$) if there is a Tukey reduction $P\Rightarrow Q$.
\item[(iii)] We say that $P$ and $Q$ are {\em Tukey equivalent} (and we write $P \sim Q$) if both $P \preceq Q$ and $Q \preceq P$.
\end{enumerate}
\end{defn}

\begin{rem}\label{at:4}
Tukey reduction between asymptotic structures is transitive.
\end{rem}
\begin{proof}
Let $P$, $Q$ and $R$ be asymptotic structures for which there exist Tukey reductions
$f:P\Rightarrow Q$ and $g:Q\Rightarrow R$. For every $\eps>0$ we choose $\delta(\eps)>0$ such that
$f_\eps:(P,\leq_\eps)\to (Q,\leq_{\delta(\eps)})$ is a Tukey map. Then the family of functions
$\{g_{\delta(\eps)}\circ f_\eps\}_{\eps>0}$ is a Tukey reduction $P\Rightarrow R$.
\end{proof}

\begin{rem}\label{at:5}
Let $P$ be an asymptotic structure and $Q$ an ordinary poset (that we view
as an asymptotic structure in the natural way). Then:
\begin{itemize}
\item[(i)] $P\preceq Q$ if and only if there is a family of functions $\{f_\eps:P\to Q\}_{\eps>0}$
such that
\[
	f_\eps:(P,\leq_\eps)\to Q \quad \mbox{is Tukey}
\]
for every $\eps>0$.
\item[(ii)] $Q\preceq P$ if and only if there is a Tukey function $g:Q \to (P,\leq_\delta)$
for some $\delta>0$.
\end{itemize}
\end{rem}

Let $P$ be an asymptotic structure equipped with the family of binary relations $\{\leq_t\}_{t>0}$.
A set $D\sub P$ is said to be {\em cofinal} if for every $t>0$ and every $p\in P$ there is
$d\in D$ such that $p\leq_t d$. The {\em cofinality} of~$P$ is defined by
$$
	\cf(P):=\max\{\aleph_0,\mbox{ the least cardinality of a cofinal subset of }P\}.
$$
A set $A\subset P$ is said to be {\em $\sigma$-bounded} if for every $t>0$ there is a sequence $(p_n)$ in~$P$ such that
for every $a\in A$ there is $n<\omega$ with $a\leq_t p_n$. We define $\add_\omega(P)$
as the least cardinality of a subset of~$P$ which is not $\sigma$-bounded (with
the convention $\add_\omega(P)=\infty$ if $P$ is $\sigma$-bounded).

\begin{lem}\label{SG:1}
Let $P$ and $Q$ be asymptotic structures such that $P\preceq Q$. Then:
\begin{enumerate}
\item[(i)] $\cf(P) \leq \cf(Q)$;
\item[(ii)] $\add_\omega(P)\ge \add_\omega(Q)$.
\end{enumerate}
Consequently, $\add_\omega(P)= \add_\omega(Q)$ and $\cf(P) = \cf(Q)$ whenever $P\sim Q$.
\end{lem}
\begin{proof} Let $\{f_\varepsilon : P \to Q \}_{\varepsilon>0}$
be a family of functions giving a Tukey reduction $P \Rightarrow Q$.

(i). Let $D \sub Q$ be a cofinal set. For each $\varepsilon>0$ and each $d\in D$, we fix
\begin{itemize}
\item $\delta_\varepsilon>0$ such that $f_\eps: (P,\leq_\eps) \to (Q,\delta_\eps)$ is Tukey;
\item $p_{d,\varepsilon}\in P$ such that $p\leq_\eps p_{d,\varepsilon}$ whenever $f_\varepsilon(p) \leq_{\delta_\eps} d$.
\end{itemize}
We claim that the set $C:=\{p_{d,\varepsilon} : d\in D, \, \varepsilon\in \mathbb{Q}^+\}$ is cofinal in~$P$.
Indeed, take $\eps>0$ and $p\in P$. Pick a rational $0<\eps'<\eps$ and use the cofinality of~$D$
to find $d\in D$ such that $f_{\varepsilon'}(p)\leq_{\delta_{\eps'}} d$.
Then $p\leq_{\varepsilon'} p_{d,\varepsilon'}$ and so
$p\leq_{\varepsilon} p_{d,\varepsilon'}$ as well. This proves that
$C$ is cofinal in~$P$ and, bearing in mind that $|C|\leq \max\{\aleph_0,|D|\}$, we conclude
that $\cf(P)\leq \max\{\aleph_0,|D|\}$. As $D$ is an arbitrary cofinal subset of~$Q$, we get
$\cf(P)\leq \cf(Q)$.

(ii). Fix $\kappa< \add_\omega(Q)$ and consider any set $A\sub P$ with $|A|\leq \kappa$. We shall check that $A$ is
$\sigma$-bounded. Take any $\eps>0$ and choose $\delta>0$ such that $f_\eps:(P,\leq_\eps) \to (Q,\leq_\delta)$
is Tukey. Since $|f_\eps(A)|\le |A|\le \kappa<\add_\omega(Q)$, there is
a sequence $(q_n)$ in~$Q$ such that for every $p\in A$ we have $f_\eps(p)\leq_\delta q_n$ for some $n<\omega$.
Now, for each $n<\omega$ we take $p_n\in P$ such that $p\leq_\eps p_n$ whenever $f_\eps(p)\leq_\delta q_n$.
Since for every $p\in A$ there is $n<\omega$ such that $p\leq_\eps p_n$,
it follows that $A$ is $\sigma$-bounded and this shows $\add_\omega(P)\ge \add_\omega(Q)$.
The proof is over.
\end{proof}

Let us present an example of asymptotic structure in the Banach space setting.

\begin{defn}\label{at:2}
Given a Banach space~$X$ and $E \sub X$, we define an asymptotic structure $\mathcal{AK}(E)$ as follows:
\begin{enumerate}
\item[(i)] the underlying set is the family of all weakly compact subsets of~$E$;
\item[(ii)] for every $t>0$ the binary relation $\leq_t$ is defined by
$$
	K \leq_t L \quad \Longleftrightarrow \quad
	K\subset L + tB_X.
$$
\end{enumerate}
\end{defn}

In Proposition~\ref{at:7} we explore the general relations between
$\mathcal{K}(B_X)$, $\AK(B_X)$ and $\AK(X)$. We first need some lemmata.
The first one is sometimes called {\em Grothendieck's test}
of weak compactness (see e.g. \cite[p.~227, Lemma~2]{die-J}).
This lemma is closely related to the so-called De Blasi measure of weak noncompactness~\cite{bla},
which is implicitly used later in the paper.

\begin{lem}\label{lem:Grothendieck}
A bounded set $C \sub X$ is relatively weakly compact if and only if
for every $\eps>0$ there is $K\in \cK(X)$ such that $C \sub K+\eps B_X$.
\end{lem}

\begin{lem}\label{lem:famous}
Let $C \sub X$ be a bounded set. The following statements are equivalent:
\begin{enumerate}
\item[(i)] $C$ is relatively weakly compact.
\item[(ii)] For every sequence $(x_n)$ in~$C$ and every $\delta>0$ there is $K\in \cK(X)$ such that
$x_n\in K+\delta B_X$ for infinitely many $n$'s.
\end{enumerate}
\end{lem}
\begin{proof}
The implication (i)$\impli$(ii) is obvious.
Conversely, assume that (ii) holds.
It suffices check that every sequence $(x_n)$ in~$C$
admits a relatively weakly compact subsequence.
By~(ii) (taking $\delta=1$) we can choose infinite $N_0 \sub \omega$ and $K_0 \in \cK(X)$ such that
$x_n \in K_0+B_X$ for every $n\in N_0$. Condition~(ii) applied to the subsequence $(x_n)_{n\in N_0}$ and $\delta=\frac{1}{2}$
ensures the existence of an infinite set $N_1 \sub N_0$ and $K_1\in \cK(X)$ such that
$x_n \in K_1+\frac{1}{2} B_X$ for every $n\in N_1$. Continuing in this manner, we can find a decreasing sequence
$(N_m)$ of infinite subsets of~$\omega$ and a sequence $(K_m)$ in~$\cK(X)$ such that
\begin{equation}\label{equation:oh}
	x_n \in K_m+\frac{1}{m+1} B_X \quad \mbox{for every }n\in N_m \mbox{ and }m<\omega.
\end{equation}
Let $(n_m)$ be a strictly increasing sequence such that $n_m\in N_m$ for all $m< \omega$. We claim that
$A:=\{x_{n_m}:m< \omega\}$ is relatively weakly compact. Indeed, since $C$ is bounded, so is~$A$. Fix $\epsilon>0$. Choose $m<\omega$
large enough such that $\frac{1}{m+1}\leq \epsilon$. Given any $k\geq m$, we have
$n_k\in N_k \sub N_m$ and~\eqref{equation:oh} yields
$x_{n_k} \in K_m + \frac{1}{m+1} B_X \sub K_m +\epsilon B_X$.
Hence
$$
	A \sub \Big(\{x_{n_1},\dots,x_{n_m}\}\cup K_m \Big)+ \epsilon B_X,
$$
where $\{x_{n_1},\dots,x_{n_m}\}\cup K_m \in \cK(X)$. An appeal to Lemma~\ref{lem:Grothendieck}
ensures that $A$ is relatively weakly compact.
\end{proof}

\begin{lem}\label{SuperLemma}
Let $A \sub B_X$, $C \sub X$ and $\eps>0$. If $A \sub C + \epsilon B_X$, then
$$
	A \sub \Bigl(\frac{1}{1+\epsilon}C\Bigr)\cap B_X + \frac{2\epsilon}{1+\epsilon}B_X.
$$
\end{lem}
\begin{proof}
Note that
\begin{equation}\label{eqn:SL}
	A \sub C\cap (1+\epsilon)B_X + \epsilon B_X.
\end{equation}
Indeed, given any $x\in A$ we can write $x=y+z$, where $y\in C$ and $z\in \epsilon B_X$, and so
$\|y\|\leq \|x\|+\|z\|\leq 1+\epsilon$. Therefore
\begin{multline*}
	A
	\stackrel{A \sub B_X}{\sub}
	\frac{1}{1+\epsilon}A + \frac{\epsilon}{1+\epsilon}B_X
	\stackrel{\eqref{eqn:SL}}{\sub}
	\frac{1}{1+\epsilon}\Bigl(C\cap (1+\epsilon)B_X + \epsilon B_X\Bigr)
	+\frac{\epsilon}{1+\epsilon}B_X = \\
	= \Bigl(\frac{1}{1+\epsilon}C\Bigr)\cap B_X + \frac{2\epsilon}{1+\epsilon}B_X,
\end{multline*}
as required.
\end{proof}

\begin{prop}\label{at:7}
\mbox{ }
\begin{enumerate}
\item[(i)] $\mathcal{AK}(E) \preceq \cK(E)$ for every $E\sub X$.
\item[(ii)] $\mathcal{AK}(B_X) \preceq \mathcal{AK}(X)$.
\item[(iii)] $\mathcal{AK}(B_X)\times \omega \preceq \mathcal{AK}(B_X)$ whenever $X$ is non-reflexive.
\item[(iv)] If $Y$ is a Banach space which is isomorphic to a complemented subspace of~$X$, then $\AK(B_Y) \preceq \AK(B_X)$.
\item[(v)] If $X$ and $Y$ are isomorphic Banach spaces, then $\mathcal{AK}(B_X) \sim \mathcal{AK}(B_Y)$.
\end{enumerate}
\end{prop}
\begin{proof}
(i). A Tukey reduction $f:\mathcal{AK}(E) \Rightarrow \cK(E)$ is defined by
taking $f_\eps(K):=K$ for every $\eps>0$ and every weakly compact set~$K \sub E$.

(ii). We define $f_\eps: \mathcal{AK}(B_X) \to \mathcal{AK}(X)$ by $f_\eps (K):=K$ for every $\eps>0$. Let us check that $\{f_\eps\}_{\eps>0}$
is a Tukey reduction. Fix $\varepsilon>0$,
choose $0<\epsilon'<\min\{\epsilon,2\}$ and take $\delta:=\frac{\epsilon'}{2-\epsilon'}$. We shall prove that
$f_\eps: (\mathcal{AK}(B_X),\leq_\eps) \to (\mathcal{AK}(X),\leq_\delta)$ is Tukey. Indeed, given $L_0\in\mathcal{AK}(X)$, we take $K_0 :=(\frac{1}{1+\delta}L_0)\cap B_X\in \mathcal{AK}(B_X)$. Then Lemma~\ref{SuperLemma} implies that
$K\sub K_0 + \varepsilon B_X$ for every $K \in \AK(B_X)$ satisfying $K\sub L_0 + \delta B_X$.

(iii). Since $X$ is not reflexive, $B_X$ is not weakly compact and Lemma~\ref{lem:famous}
ensures the existence of a sequence $(x_n)$ in~$B_X$ and a constant $\delta_0>0$ such that for every $K\in \cK(X)$
the set $\{n<\omega:x_n\in K+\delta_0 B_X\}$ is finite.
For every $\eps>0$ we define
$$
	\tau_\eps: \AK(B_X)\times\omega\to\AK(B_X),
	\quad \tau_\eps(K,n):=K\cup \{x_0,\dots,x_n\}.
$$
We claim that the family $\{\tau_\eps\}_{\eps>0}$ defines a Tukey reduction
$\AK(B_X)\times\omega\Rightarrow \AK(B_X)$.
Indeed, fix $\eps>0$ and take $\delta:=\min\{\delta_0,\eps\}$.
Given $L_0\in \AK(B_X)$, there is $n_0<\omega$ such that $x_n\not\in L_0+\delta_0 B_X$ for all $n> n_0$.
Therefore, if $(K,n)\in \AK(B_X)\times\omega$ satisfies $\tau_\epsilon(K,n) \sub L_0+\delta B_X$,
then $K \sub L_0+\epsilon B_X$ and $n\le n_0$.

(iv). Let $Z$ be a complemented subspace of~$X$ which is isomorphic to~$Y$. Let $T:Y \to Z$ be an isomorphism and let $P:X \to X$ be a projection
onto~$Z$. For every $\eps>0$ we define the map
$$
	f_\eps:\AK(B_Y) \to \AK(B_X), \quad
	f_\eps(K):=\frac{1}{\|T\|}T(K).
$$
Let us check that $\{f_\eps\}_{\eps>0}$ defines a Tukey reduction $\AK(B_Y) \Rightarrow \AK(B_X)$.
Fix $\eps>0$ and take $\delta>0$ small enough such that
\begin{equation}\label{eqn:TT}
	\frac{2\delta\|T\|\|T^{-1}\|\|P\|}{1+\delta\|T\|\|T^{-1}\|\|P\|} \leq \epsilon.
\end{equation}
Write $\eps':=\delta\|T\|\|T^{-1}\|\|P\|$. Given $L_0\in \AK(B_X)$, set
$$
	K_0:=\Bigl(\frac{\|T\|}{1+\eps'}T^{-1}(P(L_0))\Bigr)\cap B_Y \in \AK(B_Y).
$$
Now if $K \in \AK(B_Y)$ satisfies $f_\eps(K) \sub L_0+\delta B_X$, then
$$
	K \sub
	\|T\| T^{-1}((L_0+\delta B_X)\cap Z) \sub
	\|T\| T^{-1}(P(L_0+\delta B_X)) \sub
	\|T\| T^{-1}(P(L_0)) + \eps' B_Y
$$
and so Lemma~\ref{SuperLemma} implies that
$$
	K \sub K_0 + \frac{2\eps'}{1+\eps'}B_Y \stackrel{\eqref{eqn:TT}}{\sub} K_0+\eps B_Y.
$$
This finishes the proof of~(iv).
Finally, (v) follows at once from~(iv).
\end{proof}

Theorem~\ref{mer_sta} in Section~\ref{section:unconditional}
will make clear that the assertion of Proposition~\ref{at:7}(iv)
is no longer true if $Y$ is just an uncomplemented subspace of~$X$.

We presented SWCG Banach spaces as a motivation for introducing the asymptotic
structure $\mathcal{AK}(B_X)$. To see how this class of spaces fits in the theory,
we need the following elementary characterization (the equivalence (i)$\Leftrightarrow$(ii) was already pointed
out in \cite[Theorem~2.1]{sch-whe}).

\begin{lem}\label{charSWCG}
The following statements are equivalent:
\begin{enumerate}
\item[(i)] $X$ is SWCG;
\item[(ii)] there exist countably many weakly compact sets $\{K_n : n<\omega\}$ in~$X$ such that for every
$\varepsilon>0$ and every weakly compact set $L\subset X$ there is $n<\omega$ such that $L\sub K_n + \varepsilon B_X$;
\item[(iii)] there exist countably many weakly compact sets $\{S_n : n<\omega\}$ in~$B_X$ such that for every
$\varepsilon>0$ and every weakly compact set $L\subset B_X$ there is $n<\omega$ such that $L\sub S_n + \varepsilon B_X$.
\end{enumerate}
\end{lem}
\begin{proof}
(i)$\impli$(ii) is trivial.

(ii)$\impli$(iii). For each $n<\omega$ and each $p \in \qu^+$, we consider the weakly compact set
$S_{n,p} := p K_n \cap B_X$. Fix $\epsilon>0$ and a weakly compact set $L \sub B_X$. Choose $\eps' \in \qu^+$ such that
$\frac{2\eps'}{1+\eps'}\leq \eps$. By the assumption, there is
$n<\omega$ such that $L \sub K_n+\eps' B_X$. An appeal to Lemma~\ref{SuperLemma}
yields
$$
	L \sub \Bigl(\frac{1}{1+\eps'}K_n\Bigr)\cap B_X + \frac{2\eps'}{1+\eps'}B_X \sub
	\Bigl(\frac{1}{1+\eps'}K_n\Bigr)\cap B_X + \eps B_X.
$$
So, the family $\{S_{n,p}:n<\omega,p\in \qu^+\}$ satisfies the required property.

(iii)$\impli$(ii). The family of weakly compact sets
$\{p S_n:n<\omega,p \in \qu^+\}$ fulfills the required property, as can be easily checked.

(iii)$\impli$(i). Define $S'_n=\overline{{\rm aco}}(S_n)$ for every $n<\omega$.
Since each $S'_n$ is weakly compact, the set
\[
	K:=\left\{\sum_{n<\omega}\frac{1}{2^{n}} \, x_n: \, x_n\in S'_n \mbox{ for all }n<\omega\right\}
\]
is weakly compact as well. Indeed, note that the mapping
$$
	\prod_{n<\omega}S'_n \to X,
	\quad
	(x_n) \mapsto \sum_{n<\omega}\frac{1}{2^{n}} \, x_n,
$$
is continuous when $\prod_{n<\omega}S'_n$ is equipped with the product of the weak topology
and $X$ is equipped with the weak topology. Bearing in mind that $S_n \sub 2^{n} K$ for every $n<\omega$, it
is easy to check that $K$ strongly generates~$X$.
\end{proof}

\begin{thm}\label{swcg:2}
\mbox{ }
\begin{enumerate}
\item[(i)] $X$ is reflexive if and only if $\mathcal{AK}(B_X)\sim \{0\}$.
\item[(ii)] $X$ is SWCG and non-reflexive if and only if $\mathcal{AK}(B_X)\sim \omega$.
\end{enumerate}
\end{thm}
\begin{proof}
(i). Since $\AK(B_X) \preceq K(B_X)$ (Proposition~\ref{at:7}(i)),
we have $\mathcal{AK}(B_X)\sim \{0\}$ when $X$ is reflexive. Conversely, if $X$ is non-reflexive,
then Proposition~\ref{at:7}(iii) yields $\mathcal{AK}(B_X)\times \omega \preceq \AK(B_X)$
and so $\omega \preceq \AK(B_X)$.

(ii). By the proof of~(i), it only remains to prove that
\begin{center}
	$X$ is SWCG if and only if $\mathcal{AK}(B_X)\preceq \omega$.
\end{center}
Suppose first that $X$ is SWCG and let $\{S_n:n<\omega\}$ be a family of weakly compact subsets of~$B_X$ as in Lemma~\ref{charSWCG}(iii).
We can assume that $S_n \sub S_{n+1}$ for every $n<\omega$.
Given $\varepsilon>0$, we define $f_\eps: \mathcal{AK}(B_X) \to \omega$ as follows: for every $K\in \AK(B_X)$, we choose
$f_\eps(K) <\omega$ satisfying $K\sub S_{f_\eps(K)} + \varepsilon B_X$. It is clear
that $\{f_\eps\}_{\eps>0}$ defines a Tukey reduction $\mathcal{AK}(B_X)\Rightarrow \omega$.
Hence $\mathcal{AK}(B_X)\preceq \omega$.

Conversely, suppose there is a Tukey reduction $f:\mathcal{AK}(B_X)\Rightarrow \omega$.
For every $\eps\in \qu^+$ and every $n<\omega$, we take $S_{\eps,n}\in \AK(B_X)$ such that
$L \sub S_{\eps,n}+\eps B_X$ whenever $L \in \AK(B_X)$ satisfies $f_\eps(L)\leq n$.
Of course, the collection $\{S_{\varepsilon,n}:\eps\in \qu^+,n<\omega\}$ fulfills condition (iii) in Lemma~\ref{charSWCG},
and so $X$ is SWCG.
\end{proof}

The previous results say that $X$ is SWCG if and only if $\cf(\mathcal{AK}(B_X))=\aleph_0$. For an arbitrary space,
$\cf(\AK(B_X))$ is the least cardinality of a family $\mathcal{G}$ of weakly compact subsets of~$X$ which {\em strongly generates}~$X$, in
the sense that for every weakly compact set $K \sub X$ and every $\epsilon>0$ there is $G \in \mathcal{G}$ such that
$K \sub G + \epsilon B_X$.
On the other hand, the cardinal invariant $\add_\omega(\AK(B_X))$ is the least cardinality of
a family $\mathcal{H} \sub \cK(X)$ for which there exists $\eps>0$ such that for every sequence $(K_n)$ in~$\cK(X)$
there is $K\in\mathcal{H}$ such that $K\not\subset K_n+\eps B_X$ for all $n<\omega$.
The following remark collects some basic properties of these coefficients.

\begin{rem}\label{SG:2}
\mbox{ }

\begin{enumerate}
\item[(i)] $\add_\omega(\AK(B_X))\le \cf(\AK(B_X))$ whenever $X$ is not SWCG.
\item[(ii)] $\cf(\AK(B_X))\leq \mathfrak{c}$ whenever $B_{X^*}$ is $w^*$-separable.
\item[(iii)] If $Y$ is a complemented subspace of~$X$, then $\add_\omega(\AK(B_Y))\ge\add_\omega(\AK(B_X))$ and $\cf(\AK(B_Y)) \leq \cf(\AK(B_X))$.
\end{enumerate}
\end{rem}
\begin{proof}
Part~(i) is immediate. (ii) is a consequence of the fact that $X$ embeds isomorphically into~$\ell^\infty$
and so it has $\mathfrak{c}$ many weakly compact subsets. (iii) follows
directly from Proposition \ref{at:7}(iv) and Lemma~\ref{SG:1}.
\end{proof}

\begin{thm}\label{as:4}
If $\AK(B_X)\preceq P$ for a partially ordered set~$P$, then $\cK(B_X)\preceq P^\omega$.
\end{thm}
\begin{proof}
Let $\{f_\eps:\AK(B_X)\to P\}_{\eps>0}$ be a Tukey reduction.
Define
\[
	F:
	\cK(B_X)\to P^\omega,
	\quad F(L):=(f_1(L), f_{\frac{1}{2}}(L), f_{\frac{1}{3}}(L),\dots).
\]
We claim that $F$ is a Tukey map. Indeed, fix any $(p_n)\in P^\omega$.
For every $n<\omega$ there is $K_n \in \AK(B_X)$ such that $L \sub K_n + \frac{1}{n+1}B_X$
whenever $L \in \AK(B_X)$ fulfills $f_{\frac{1}{n+1}}(L) \leq p_n$. Note that
\[
	K:=\bigcap_{n<\omega} \Bigl(K_n+\frac{1}{n+1}B_X\Bigr)
\]
is a weakly compact subset of~$B_X$ (apply Lemma~\ref{lem:Grothendieck}).
Clearly, if $L \in \cK(B_X)$ satisfies
$F(L) \leq (p_n)$, then $L \sub K$. It follows that $F$ is Tukey, as claimed.
\end{proof}

{Our proof of the next lemma imitates the argument of \cite[Proposition 1]{lou-vel}.
For the tree $\omega^{<\omega}$ of all finite sequences of natural numbers
we use notations similar to those used for the dyadic tree~$2^{<\omega}$. For instance, given $s=(s_0,\dots,s_k)\in \omega^{<\omega}$
and $n<\omega$, we write $s\smallfrown n=(s_0,\dots,s_k,n)$.}

\begin{lem}\label{lem:LV}
{If $X$ is separable and $\mathcal{AK}(B_X) \preceq \omega^\omega$, then either $X$ is SWCG
or $\mathcal{AK}(B_X)\sim\omega^\omega$.}
\end{lem}
\begin{proof} If $X$ is not SWCG, then there exists $\varepsilon>0$ such that,
for every countable family $\{K_n:n<\omega\}$ of weakly compact subsets of~$B_X$,
there is $K\in \mathcal{AK}(B_X)$ such that $K\not\subseteq K_n + \varepsilon B_X$ for any $n<\omega$
(apply Lemma~\ref{charSWCG}). Let us say that a set $\mathcal{W}\sub \mathcal{AK}(B_X)$ is \emph{large} if for every countable family $\{K_n:n<\omega\}$ of
weakly compact subsets of~$B_X$, there is $K\in\mathcal{W}$ such that $K\not\subseteq K_n + \varepsilon B_X$ for any $n<\omega$.

Set $\eps':=\frac{\eps}{2}$. Fix a Tukey function $\phi: (\mathcal{AK}(B_X),\leq_{\eps'}) \to \omega^\omega$.
For each $s\in\omega^{<\omega}$, let
$$
	I_s := \{K\in\mathcal{AK}(B_X) : \, s\sqsubseteq \phi(K)\}.
$$
We know that $I_\emptyset = \mathcal{AK}(B_X)$ is large. It is clear that if $I_s$ is
large, then there is $n<\omega$ such that $I_{s\smallfrown n}$ is large
(because $I_s=\bigcup_{m<\omega}I_{s\smallfrown m}$). Hence, we can find
$\sigma\in\omega^\omega$ such that $I_{\sigma|_n}$ is large for every $n<\omega$.

Since $X$ is separable, for every $n<\omega$ we can choose $\{L_{n,k} : k<\omega\}\sub I_{\sigma|_n}$ such that
\begin{equation}\label{eqn:denseLV}
	\overline{\bigcup_{k<\omega}L_{n,k}} \supseteq \bigcup \{K: \, K\in I_{\sigma|_n}\}.
\end{equation}

{\em Claim.} For every $f\in\omega^\omega$ the family $\{\phi(L_{n,k}): n<\omega, k\leq f(n)\}$ is bounded in~$\omega^\omega$.
Indeed, define $g\in \omega^\omega$ by declaring
$$
	g(m):=\max\big\{\sigma(m), \, \max\{\phi(L_{n,k})(m): \, n\leq m, \, k\leq f(n)\}\big\}, \quad m<\omega.
$$
Given any $n<\omega$ and $k\leq f(n)$, we have $\phi(L_{n,k}) \leq g$ in~$\omega^\omega$, because
$\sigma|_n \sqsubseteq \phi(L_{n,k})$ implies that $\sigma(m) = \phi(L_{n,k})(m)$ for all $m<n$.

By the previous claim and the fact that $\phi: (\mathcal{AK}(B_X),\leq_{\eps'}) \to \omega^\omega$ is Tukey,
it follows that for every $f\in \omega^\omega$ there is $\psi(f)\in\mathcal{AK}(B_X)$ such that
\begin{equation}\label{eqn:LVTukey}
	L_{n,k}\sub \psi(f) + \varepsilon' B_X
	\quad\mbox{whenever }k\leq f(n).
\end{equation}
In order to finish the proof it suffices to check that $\psi:\omega^\omega\to (\mathcal{AK}(B_X),\leq_{\eps'})$ is
Tukey. Fix $L\in \mathcal{AK}(B_X)$. Suppose, for contradiction, that
$F:=\{f\in \omega^\omega:\psi(f) \sub L + \eps' B_X\}$ is not bounded. Then there is $n<\omega$ such that
$\{f(n) : f\in F\}$ is unbounded in~$\omega$. Hence, for every $k<\omega$ there is $f\in F$ such that $k \leq f(n)$
and so
$$
	L_{n,k}\stackrel{\eqref{eqn:LVTukey}}{\sub} \psi(f) + \eps' B_X \sub L + \varepsilon B_X.
$$
Bearing in mind~\eqref{eqn:denseLV}, we get $\bigcup \{K: K\in I_{\sigma|_n} \}\sub L + \varepsilon B_X$,
which contradicts that $I_{\sigma|_n}$ is large. The proof is over.
\end{proof}

Statement~(iv) of the following corollary generalizes Example~\ref{posets:4}.

\begin{cor}\label{as:4-cor} \mbox{ }
\begin{itemize}
\item[(i)] If $\AK(B_X)\sim P$ for a partially ordered set~$P$, then $P\preceq \cK(B_X)\preceq P^\omega$.
\item[(ii)] $\mathcal{K}(B_X) \sim \omega^\omega$ if and only if $\omega\preceq \mathcal{AK}(B_X)\preceq \omega^\omega$.
\item[(iii)] If $\AK(B_X)\sim \cK(\qu)$, then $\cK(B_X)\sim \cK(\qu)$.
\item[(iv)] If $X$ is SWCG, then $\cK(B_Y)\sim \omega^\omega$ for every non-reflexive subspace $Y \sub X$.
\item[(v)] {If $X$ is SWCG, then $\mathcal{AK}(B_Y)\sim \omega^\omega$ for every separable and
non-SWCG subspace $Y \sub X$.}
\end{itemize}
\end{cor}
\begin{proof}
(i) is a consequence of Theorem~\ref{as:4} and Proposition~\ref{at:7}(i).

(ii). Suppose first that $\mathcal{K}(B_X)\sim \omega^\omega$. On one hand $\mathcal{AK}(B_X)\preceq \mathcal{K}(B_X)\sim \omega^\omega$
(Proposition~\ref{at:7}(i)). On the other hand, $X$ is not reflexive, so by Proposition~\ref{at:7}(iii) we get $\omega\preceq \mathcal{AK}(B_X)$.
Conversely, if $\omega\preceq \mathcal{AK}(B_X)\preceq \omega^\omega$, then Theorem~\ref{as:4}
yields $\mathcal{K}(B_X) \preceq \omega^\omega$
(note that $\omega^\omega$ and $(\omega^\omega)^\omega$ are Tukey equivalent).
Since $X$ is not reflexive (because $\mathcal{AK}(B_X)\not\sim\{0\}$), Proposition~\ref{KXKBX}(iii)
allows us to deduce that $\mathcal{K}(B_X) \sim \omega^\omega$.

(iii) follows from (i), the Tukey equivalence $\cK(\qu)^\omega\sim \cK(\qu)$ (which was proved in Lemma~\ref{lem:TukeyKQ})
and Proposition~\ref{at:7}(i).

(iv). If $Y\sub X$ is a non-reflexive subspace, then $\omega^\omega\preceq \cK(B_Y)$ (by Proposition~\ref{KXKBX}(iii)).
Since $X$ is SWCG, we have $\AK(B_X) \preceq \omega$ (by Theorem~\ref{swcg:2}), and so Theorem \ref{as:4}
yields $\cK(B_X) \preceq \omega^\omega$. Since $\cK(B_Y) \preceq \cK(B_X)$ (Remark~\ref{rem:Ksubspaces}), we get
$\cK(B_Y) \sim \omega^\omega$.

{Finally, (v) is an immediate consequence of (ii), (iv) and Lemma~\ref{lem:LV}.}
\end{proof}

The structure $\mathcal{K}(B_X)$ cannot identify the class of SWCG spaces. Indeed,
there exist non-reflexive Banach spaces with separable dual having the PCP
(like the predual of the James tree space, see \cite[Section~6]{edg-whe}). Such a space~$X$ satisfies
$\mathcal{K}(B_X)\sim \omega^\omega$ (by Proposition~\ref{posets:3}) but it is not SWCG: note
that any SWCG space is weakly sequentially complete \cite[Theorem~2.5]{sch-whe}, and so it is reflexive if
and only if it does not contain isomorphic copies of~$\ell^1$ (thanks to Rosenthal's $\ell^1$-theorem, \cite[Theorem~5.37]{fab-ultimo}).

To illustrate the concept of Tukey reduction of asymptotic structures of weakly compact sets,
we now exhibit some examples where $\AK(B_X)$ can be identified.
The first one says, in particular, that for $X=c_0$
we have $\AK(B_X)\sim \cK(B_X)\sim \cK(\qu)$
(this is also a particular case of Theorem~\ref{separabledual} in Section \ref{section:nol1}) and so
\[
	\add_\omega(\AK(B_{c_0}))= \add_\omega(\cK(\qu))={\mathfrak b} \quad\mbox{and}\quad  \cf(\AK(B_{c_0}))=\cf (\cK(\qu))={\mathfrak d}.
\]

\begin{exa}\label{exa:c0l1}
Let $(X_n)$ be a sequence of separable Banach spaces having the Schur property.
Then $X=(\bigoplus_{n<\omega} X_n)_{c_0}$ satisfies $\AK(B_X)\sim \cK(B_X)\sim \cK(\qu)$.
\end{exa}
\begin{proof} For each $x\in X$
we write $|x|:=(\|\pi_n(x)\|_{X_n}) \in \mathbb{R}^\omega$, where
$\pi_n:X \to X_n$ denotes the $n$-th coordinate projection. For each $n<\omega$, we write
$\cK_0(B_{X_n})$ to denote the collection of all weakly compact subsets of~$B_{X_n}$ containing~$0$, which is
partially ordered by the inclusion relation. Since $X_n$ has the Schur property, its weakly compact subsets are norm compact.
Recall that $\mathbb{Q}$ is identified with the subspace of~$2^\omega$ made up of all eventually
zero sequences. Using the previous identification, we write each $a\in \mathbb{Q}$ as a sequence $a=(a(n))$
of~$0$'s and~$1$'s which is eventually zero.

{\em Step~1.} For every $A\in \cK(\qu)$ and every $(L_n)\in \prod_{n<\omega}\cK_0(B_{X_n})$, the set
\[
	K_{A,(L_n)}:=\Bigl\{x\in \bigcap_{n<\omega}\pi_n^{-1}(L_n): \, |x|\le a \mbox{ for some } a\in A\Bigr\} \sub B_X
\]
is weakly compact. Indeed, let $(x^k)$ be a sequence in~$K_{A,(L_n)}$. Since for each $n<\omega$ the
sequence $(\pi_n(x^k))$ is contained in the norm compact set~$L_n$, we can pass to a subsequence of~$(x^k)$, not relabeled,
such that for every $n<\omega$ the sequence $(\pi_n(x^k))$ is norm convergent to some~$\varphi_n\in L_n$.
For each $k<\omega$ we fix $a^k\in A$ such that $|x^k|\le a^k$. By passing to a further subsequence, we
can assume that $a^k\to a\in A$. Pick $n_0<\omega$ such that $a(n)=0$ for every $n> n_0$ and define
$x\in X$ by declaring $\pi_n(x):=\varphi_n$ for all $n\leq n_0$ and $\pi_n(x):=0$ for all $n>n_0$,
so that $x\in \bigcap_{n<\omega}\pi_n^{-1}(L_n)$. Observe that $|x|\leq a$, because
for every $n<\omega$ we have
\begin{equation}\label{eqn:Xn}
	\|\varphi_n\|_{X_n}=\lim_{k\to \infty}\|\pi_n(x^k)\|_{X_n} \leq \lim_{k\to \infty}a^k(n)=a(n).
\end{equation}
Hence $x\in K_{A,(L_n)}$. Note that~\eqref{eqn:Xn} also implies that $\varphi_n=0$ for all $n>n_0$. Finally,
observe that $(x^k)$ converges weakly to~$x$, because $(x^k)$ is bounded
and for every $n<\omega$ we have $\pi_n(x^k) \to \varphi_n=\pi_n(x)$ in~$X_n$.

{\em Step~2.} Let $\eps>0$ and define $s_\eps: X \to 2^\omega$ by
$$
	s_\eps(x)(n):=\begin{cases}
	1 & \text{if $\|\pi_n(x)\|_{X_n}\ge \eps$}, \\
	0 & \text{otherwise}.
	\end{cases}
$$
If $K\in \cK(B_{X})$, then $\tau_\eps(K):=\overline{\{s_\eps(x):x\in K\}} \sub \qu$, i.e.
$\tau_\eps(K)\in\cK(\qu)$. Indeed, fix $a\in \tau_\eps(K)$ and let $(x^k)$ be a sequence in~$K$ such that $s_\eps(x^k) \to a$ in~$2^\omega$.
We can additionally assume that $(x^k)$ converges weakly to some $x\in K$, so that for every $n<\omega$
we have $\pi_n(x^k) \to \pi_n(x)$ in~$X_n$. Therefore, for each $n<\omega$ we have
$$
	\|\pi_n(x)\|_{X_n}\ge\eps \quad \mbox{whenever }a(n)=\lim_{k\to \infty}s_\eps(x^k)(n)=1.
$$
Since the set $\{n<\omega: \|\pi_n(x)\|_{X_n}\geq \eps\}$ is finite, it follows that $a\in \qu$.

{\em Step~3.} For every $\eps>0$ we define the map
$$
	F_\eps: \AK(B_X) \to \cK(\qu) \times \prod_{n<\omega}\cK(B_{X_n}), \quad
	F_\eps(K):=\bigl(\tau_\eps(K),(\pi_n(K))\bigr).
$$
Then $\{F_\eps\}_{\eps>0}$ defines a Tukey reduction $\AK(B_X) \Rightarrow
\cK(\qu) \times \prod_{n<\omega}\cK(B_{X_n})$. Indeed, fix
$A\in \cK(\qu)$ and $(L_n)\in \prod_{n<\omega}\cK(B_{X_n})$. Write $L'_n:=L_n\cup \{0\}\in \cK_0(B_{X_n})$ for all $n<\omega$.
Let $K_{A,(L'_n)} \in \AK(B_X)$ be as in Step~1. We next check that if $K\in \AK(B_X)$ satisfies
$F_\eps(K) \leq (A,(L_n))$, then $K\sub K_{A,(L'_n)}+\eps B_X$. To this end,
take $x\in K$. Then $a:=s_\eps(x)\in \tau_\eps(K) \sub A$. Define $y\in B_X$ by declaring
$\pi_n(y):=\pi_n(x)$ whenever $\|\pi_n(x)\|_{X_n}\geq \eps$ and
$\pi_n(y):=0$ otherwise. Clearly, we have $\|x-y\|_X \leq \eps$. On the other hand,
since $\pi_n(x) \in \pi_n(K) \sub L'_n$ and $0\in L'_n$ for every $n<\omega$, we also have
$y\in \bigcap_{n<\omega}\pi_n^{-1}(L'_n)$; since $a(n)=1$ whenever $\pi_n(y)\neq 0$, we conclude that
$y\in K_{A,(L'_n)}$. This proves that $K\sub K_{A,(L'_n)}+\eps B_X$.

{\em Step~4.} Since each $X_n$ is SWCG, we have either $\cK(B_{X_n}) \sim \{0\}$ or $\cK(B_{X_n}) \sim \omega^\omega$
(apply Corollary~\ref{as:4-cor}(iv)), hence either $\prod_{n<\omega}\cK(B_{X_n}) \sim \{0\}$
or $\prod_{n<\omega}\cK(B_{X_n}) \sim \omega^\omega$ and therefore
$$
	\AK(B_X) \preceq \cK(\qu) \times \prod_{n<\omega}\cK(B_{X_n}) \preceq \cK(\qu)\times \cK(\qu) \sim \cK(\qu).
$$

{\em Step~5.} If $K \sub B_X$ is weakly compact, then the set
$$
	A_K:=\{a\in 2^\omega: \, a \le 2|x|\mbox{ for some } x\in K\}
$$
belongs to~$\cK(\mathbb{Q})$. Indeed, since for every $x\in X$ the set $\{n<\omega: \|\pi_n(x)\|_{X_n} \geq \frac{1}{2}\}$
is finite, $A_K \sub \qu$. Therefore, it suffices to check that $A_K$ is closed in~$2^\omega$.
To this end, let $(a^k)$ be a sequence in~$A_K$ which converges to~$a\in 2^\omega$. For each $k<\omega$ we pick
$x^k\in K$ such that $a^k\le 2|x^k|$. Since $K$ is weakly compact, by passing to a subsequence
we may assume that $(x^k)$ is weakly convergent to some~$x\in K$. There is $n_0<\omega$ such that for every
$n\geq n_0$ we have $\|\pi_n(x)\|_{X_n}<\frac{1}{2}$. For each $n<\omega$ with $n\geq n_0$
the sequence $(\pi_n(x^k))$ is norm convergent to~$\pi_n(x)$, hence
$$
	a^k(n)\le 2\|\pi_n(x^k)\|_{X_n}< 1 \quad
	\mbox{for }k \mbox{ large enough,}
$$
and so $a(n)=0$. It follows that $a\in \qu$.

{\em Step~6.}
Fix a norm one vector $x_n\in X_n$ for every $n<\omega$.
We define a mapping $G: \cK(\qu) \to \AK(B_X)$ by $G(A):=K_{A,(L_n)}$ where $L_n:=\{0,x_n\}$ for all $n<\omega$.
In order to check that $G$ gives a Tukey reduction $\cK(\qu) \Rightarrow \AK(B_X)$, we only have to
show that $G:\cK(\qu) \to (\AK(B_X),\leq_\delta)$ is Tukey for $\delta=\frac{1}{2}$. Fix $K \in \AK(B_X)$ and take $A_{K}\in \cK(\qu)$
as in Step~5. We shall check that $A \sub A_K$ for every $A \in \cK(\qu)$ satisfying $K_{A,(L_n)} \sub K+\frac{1}{2}B_X$.
Indeed, fix $a\in A$ and define $x \in K_{A,(L_n)}$ by declaring $\pi_n(x):=a(n) x_n$ for all $n<\omega$.
There is $y\in K$ such that $\|x-y\|_X \le \frac{1}{2}$. Clearly, we have $a\le 2|y|$ and so $a\in A_K$.
This shows that
$$
	\cK(\qu) \preceq \AK(B_X).
$$

Finally, note that by Steps~4 and~6 we have $\AK(B_X) \sim \cK(\qu)$. An
appeal to Corollary~\ref{as:4-cor}(iii) allows us to conclude
that $\mathcal{K}(B_X)\sim \cK(\qu)$ as well. The proof is finished.
\end{proof}

The following result is an improvement of Corollary~\ref{cor:Kc0}
within the class of Banach spaces having the {\em Separable Complementation Property} ({\em SCP} for short),
which includes all weakly compactly generated spaces. Recall
that $X$ is said to have the SCP (see e.g. \cite[Section~13.2]{fab-ultimo}) if every separable subspace of~$X$ is contained in a
complemented separable subspace of~$X$. By Sobczyk's theorem (see e.g. \cite[Theorem 2.5.8]{alb-kal}), any copy of~$c_0$ in
a Banach space having the SCP is complemented.

\begin{cor}
If $X$ contains a complemented copy~$c_0$, then $\cK(\qu)\preceq \AK(B_X)$.
\end{cor}
\begin{proof}
This follows from Example~\ref{exa:c0l1} and Proposition~\ref{at:7}(iv).
\end{proof}

Our next example requires an auxiliary lemma.

\begin{lem}\label{lem:cofinal-l1sum}
Let $\{X_i\}_{i\in I}$ be a family of Banach spaces and let $X=(\bigoplus_{i\in I} X_i)_{\ell^1}$.
Then
$$
	\AK(B_X) \preceq \prod_{i\in I}\cK(B_{X_i}) \times [I]^{<\omega}.
$$
\end{lem}
\begin{proof}
For each $i\in I$, let $\pi_i:X \to X_i$ be the $i$-th coordinate projection.
Given any finite set $J \sub I$ and weakly compact sets $K_j \sub X_j$ for $j\in J$, the set
$$
	\bigoplus_{j\in J} K_j :=\Bigl(\bigcap_{j\in J}\pi_j^{-1}(K_j)\Bigr)\cap
	\Bigl(\bigcap_{i\in I\setminus J} \pi_i^{-1}(\{0\})\Bigr) \sub X
$$
is weakly compact. The following fact is exactly the statement of \cite[Lemma~7.2(ii)]{kac-alt}.

{\em Fact.} For every weakly compact set $K \sub X$ and every $\epsilon>0$ there is a finite set $J(K,\eps) \sub I$ such that
$$
	\sum_{i\in I \setminus J(K,\eps)}\|\pi_i(x)\|_{X_i} \leq \epsilon \quad
	\mbox{for all }x\in K
$$
and so
\begin{equation}\label{eqn:JKeps}
	K \sub \bigoplus_{j\in J(K,\eps)} \pi_j(K) + \epsilon B_X.
\end{equation}

Write $P=\prod_{i\in I}\cK(B_{X_i})$.
For every $\eps>0$ we fix $r(\eps)>0$ small enough
such that $\frac{2r(\eps)}{1+r(\eps)}\leq \eps$ and define a mapping $f_\eps: \AK(B_X) \to P \times [I]^{<\omega}$ by
$$
	f_\eps(K):=\Bigl((\pi_i(K))_{i\in I},J(K,r(\eps))\Bigr).
$$
We claim that the family $\{f_\eps\}_{\eps>0}$ defines a Tukey reduction $\AK(B_X) \Rightarrow P \times [I]^{<\omega}$.
Indeed, fix $\eps>0$, take $(L_i)_{i\in I}\in P$ and $J \sub I$ finite.  Consider the weakly compact subset of~$B_X$ defined by
$$
	K_0:=\Bigl(\frac{1}{1+r(\eps)} \bigoplus_{j\in J} L_j \Bigr)\cap B_X.
$$
Let $K \in \AK(B_X)$ such that $f_\eps(K) \leq ((L_i)_{i\in I},J)$. Then $J(K,r(\eps)) \sub J$ and
$$
	K  \stackrel{\eqref{eqn:JKeps}}{\sub}
	\bigoplus_{j\in J(K,r(\eps))} \pi_j(K) + r(\epsilon) B_X
	\sub
	\bigoplus_{j\in J} L_j + r(\epsilon) B_X,
$$
which implies (via Lemma~\ref{SuperLemma}) that $K \sub K_0 + \eps B_X$. The proof is over.
\end{proof}

\begin{exa}\label{exa:l1c0}
The space $X=\ell^1(c_0)$ satisfies $\AK(B_X)\sim \cK(B_X)\sim \cK(\qu)$.
\end{exa}
\begin{proof}
We already know that $\cK(B_{c_0})\sim \cK(\qu)$, hence
$\cK(B_{c_0})^\omega \sim \cK(\qu)^\omega \sim \cK(\qu)$
(see Lemma~\ref{lem:TukeyKQ}). Thus, Lemma~\ref{lem:cofinal-l1sum} yields
$$
	\AK(B_X) \preceq \cK(\qu) \times [\omega]^{<\omega} \sim \cK(\qu) \times \omega \sim \cK(\qu).
$$
On the other hand, since $X$ contains complemented subspaces isomorphic to~$c_0$ and
$\AK(B_{c_0}) \sim \cK(\qu)$ (Example~\ref{exa:c0l1}), we have
$\cK(\qu) \preceq \AK(B_X)$ (Proposition~\ref{at:7}(iv)).
Thus $\AK(B_X) \sim \cK(\qu)$ and so Corollary~\ref{as:4-cor}(iii)
applies to get $\cK(B_X) \sim \cK(\qu)$.
\end{proof}

The $\ell^1$-sum of countably many SWCG spaces is SWCG
(see \cite[Proposition~2.9]{sch-whe}). Bearing in mind this fact, Theorem~\ref{swcg:2}
and Corollary~\ref{as:4-cor}(iv) yield:

\begin{exa}\label{exa:l1l2}
The space $X=\ell^1(\ell^2)$ satisfies $\AK(B_X)\sim \omega$ and $\cK(B_X)\sim \omega^\omega$.
\end{exa}

The case of $\ell^p$-sums, $1<p<\infty$, is different. For instance, $\ell^p(\ell^1)$ is not isomorphic to a subspace
of a SWCG space (see \cite[Corollary~2.29]{kam-mer2}). In Example~\ref{exa:l2l1}
we shall check that this space satisfies $\AK(B_X)\sim \cK(B_X)\sim \omega^\omega$.

\section{Classification of $\cK(B_X)$ under analytic determinacy}\label{section:KBXclassification}

This section is devoted to the proof of the following theorem,
but the machinery developed here will be used in further sections.

\begin{thm}[\axiom]\label{PDKclassification}
If $X$ is separable, then $\mathcal{K}(B_X)$ is Tukey equivalent to either
$\{0\}$, $\omega^\omega$, $\mathcal{K}(\mathbb{Q})$ or $[\mathfrak{c}]^{<\omega}$.
\end{thm}

The first step is to state the consequences of the axiom of analytic determinacy (denoted by \axiom) that we shall need: Theorems~\ref{projectivestronggap}
and~\ref{projectiveideal}. We shall not state here the axiom and we refer the reader to~\cite{jec} for basic knowledge on
games, winning strategies and determinacy. (This is only needed to understand our proof of Theorem~\ref{projectivestronggap}.)

We say that a family $\mathcal{A}$ of subsets of~$\omega$ is {\em hereditary} if for every $A\in \mathcal{A}$ and every
$A'\subset A$ we have $A'\in\mathcal{A}$. If $\mathcal{A}$ and $\mathcal{B}$ are two
families of subsets of~$\omega$, we say that
$\mathcal{A}$ and $\mathcal{B}$ are {\em countably separated} if there exists a countable family $\mathcal{C}$ of subsets of~$\omega$ such that
for every $a\in \mathcal{A}$ and every $b\in \mathcal{B}$ with $a\cap b = \emptyset$ there is $c\in\mathcal{C}$ such that $a \subset c$ and
$b\subset \omega\setminus c$. By identifying each subset of~$\omega$ with its characteristic function,
any family $\mathcal{A}$ of subsets of~$\omega$ can be viewed as a subset of~$2^\omega$, and in this sense
we can say that $\mathcal{A}$ is analytic, coanalytic, etc.

Given $i\in\{0,1\}$, an infinite set $S\sub 2^{<\omega}$ is said to be an {\em $i$-chain} if it can be
enumerated as $S=\{s^p:p<\omega\}$ where $s^p \smallfrown i \sqsubseteq s^{p+1}$ for all $p<\omega$.
The set of all $i$-chains of~$2^{<\omega}$ is denoted by~$\mathcal{C}_i$.

\begin{thm}[\axiom]\label{projectivestronggap}
Let $\mathcal{A}_0$ and $\mathcal{A}_1$ be hereditary families of subsets of~$\omega$ such that
$\mathcal{A}_0$ is analytic and $\mathcal{A}_1$ is coanalytic. Then:
\begin{enumerate}
\item[(i)] either $\mathcal{A}_0$ and $\mathcal{A}_1$ are countably separated,
\item[(ii)] or there exists an injective function $u:2^{<\omega}\to \omega$ such that
$u(S)\in \mathcal{A}_i$ whenever $S \sub 2^{<\omega}$ is an $i$-chain, $i\in\{0,1\}$.
\end{enumerate}
\end{thm}

\begin{rem}\label{rem:AvilesTodorcevic}
\rm If both $\mathcal{A}_0$ and $\mathcal{A}_1$ are analytic, then the statement of Theorem~\ref{projectivestronggap}
holds in ZFC, as shown by Todorcevic~\cite{tod-J-2} in a slightly different language.
His proof relies on the open graph theorem of Feng~\cite{fen}, who
also proves that an open graph theorem holds for projective sets under projective determinacy.
So it follows from combining the argument of both papers that Theorem~\ref{projectivestronggap}
holds under projective determinacy when the families are projective.
Here we present a direct proof for the case that we are going to use. Similar results for more than
two families can be found in~\cite{avi-tod,avi-tod-2}.
\end{rem}

We shall need the following elementary lemma.

\begin{lem}\label{lem:injectivization}
Let $W$ be a set and let $u':2^{<\omega}\to W$ be a function which is injective when restricted to each
$i$-chain, $i\in \{0,1\}$. Then there is a function $v:2^{<\omega}\to 2^{<\omega}$ such that:
\begin{enumerate}
\item[(i)] $u'\circ v$ is injective;
\item[(ii)] $v(S)$ is an $i$-chain whenever~$S \sub 2^{<\omega}$ is an $i$-chain, $i\in \{0,1\}$.
\end{enumerate}
\end{lem}
\begin{proof}
Consider the order $\leq_{lex}$ on~$2^{<\omega}$ defined by
declaring that $t\leq_{lex} s$ if either ${\rm length}(t) < {\rm length}(s)$ or
${\rm length}(t) = {\rm length}(s)$ and $t \leq s$ with respect to the usual lexicographical order on~$2^{{\rm length}(t)}$.
Since initial segments with respect to $\leq_{lex}$ are finite we can
construct the required  function $v$ inductively as follows.

Set $v(\emptyset):=\emptyset$. Let $s=(s_0,\dots,s_p)\in 2^{<\omega}$ and suppose that
$v(t)$ is already defined for every $t\leq_{lex} s$ with $t\neq s$.
Since $u'$ is injective on every $s_p$-chain, there is $n_s \in \N$ such that the element of~$2^{<\omega}$ given by
$$
	v(s):=v(s_0,\dots,s_{p-1})\smallfrown \underbrace{(s_p,\dots,s_p)}_{\text{of length }n_s}
$$
satisfies that $u'(v(s)) \neq u'(v(t))$ for every $t\leq_{lex} s$ with $t\neq s$. By the very construction, $u'\circ v$
is injective, and so is~$v$. Clearly, we also have
\begin{equation}\label{eqn:monotonicV}
	v(t) \sqsubseteq v(s) \quad
	\mbox{whenever }t \sqsubseteq s.
\end{equation}
Let $S=\{s^p:p<\omega\}$ be an $i$-chain, $i\in \{0,1\}$, enumerated
so that $s^p \smallfrown i \sqsubseteq s^{p+1}$ for all $p<\omega$.
Then
$$
	v(s^p)\smallfrown i \sqsubseteq
	v(s^p \smallfrown i) \stackrel{\eqref{eqn:monotonicV}}{\sqsubseteq} v(s^{p+1})
	\quad
	\mbox{ for every }p<\omega,
$$
hence $v(S)$ is an $i$-chain as well.
\end{proof}

\begin{proof}[Proof of Theorem~\ref{projectivestronggap}]
Since $\mathcal{A}_0$ is analytic, there is a continuous function
$f:\omega^\omega \to 2^\omega$ such that $f(\omega^\omega) = \mathcal{A}_0$.
Consider an infinite game with two players where, at each step~$k$, Player I plays $(m_k,\bar{m}_k) \in \omega\times \omega$,
and Player II plays $i_k \in \{0,1\}$. At the end of the game, Player I wins if and only if the following three conditions hold:
\begin{enumerate}
\item $m_k\neq m_{k'}$ for every $k\neq k'$,
\item $\{m_k : i_k=0\}\subset f(\bar{m}_0,\bar{m}_1,\dots)$,
\item $\{m_k : i_k=1\}\in \mathcal{A}_1$.
\end{enumerate}
It is easy to check that the set of all infinite rounds $(m_0,\bar{m}_0,i_0,m_1,\bar{m}_1,i_1,\dots)$
where Player~I wins is coanalytic in the Polish space~$R:=(\omega \times \omega \times \{0,1\})^\omega$.
A brief indication follows. On one hand, the set determined by condition~(1)
$$
	\{(m_0,\bar{m}_0,i_0,m_1,\bar{m}_1,i_1,\dots)\in R: \, m_k\neq m_{k'} \mbox{ for every }k\neq k'\}
$$
is closed in~$R$. On the other hand, the set determined by condition~(2) is Borel in~$R$, since it can be written as
$$
	\bigcap_{k<\omega}\{(m_0,\bar{m}_0,i_0,m_1,\bar{m}_1,i_1,\dots)\in R: \, i_k=1\} \cup U_k
$$
where
$$
	U_k=\bigcup_{l<\omega}\big\{(m_0,\bar{m}_0,i_0,m_1,\bar{m}_1,i_1,\dots)\in R: \, i_k=0, \, m_k=l, f(\bar{m}_0,\bar{m}_1,\dots)[l]=1\big\}.
$$
Finally, note that for every $l<\omega$ the mapping
$\Phi_l: R \to \{0,1\}$ defined by
$$
	\Phi_l((m_0,\bar{m}_0,i_0,m_1,\bar{m}_1,i_1,\dots)):=
	\begin{cases}
	1 & \text{if there is $k<\omega$ such that $i_k=1$ and $m_k=l$}\\
	0 & \text{otherwise}
	\end{cases}
$$
is Borel, and so does its product $\Phi:=(\Phi_l)_{l<\omega}: R \to 2^\omega$.
The subset of~$R$ determined by condition~(3) coincides with $\Phi^{-1}(\mathcal{A}_1)$, which
is coanalytic in~$R$.

So, under the axiom of analytic determinacy, this game is determined and we have two cases.

{\em Case~1. Player~I has a winning strategy.} We shall check that condition~(ii) of the theorem holds.
For each $s=(s_0,\dots,s_p)\in 2^{<\omega}$, let $(m_{p+1},\bar{m}_{p+1})$ be the move of Player~I under his strategy
after Player~II has played $s_0,\dots,s_p$ and Player~I followed his strategy. Define the function
$u':2^{<\omega}\to \omega$ by $u'(s) = m_{p+1}$. Now, let $S \sub 2^{<\omega}$ be any $i$-chain, $i\in \{0,1\}$.
Write $S = \{s^p:p<\omega\}$, where $s^p=(s_0,\dots,s_{n_p})$, the sequence
$(n_p)$ is strictly increasing and $s_{n_{p}+1} = i$ for all $p<\omega$. Consider the
infinite round $(m_0,\bar{m}_0,s_0,m_1,\bar{m}_1,s_1,\dots)$ played according to the winning strategy
of Player~I. Then $u'(S)=\{m_{n_p+1}:p<\omega\} \sub \{m_k: s_k=i\}$. Since Player~I is the winner of this round
we conclude that $u'(S) \in \cA_i$ (bear in mind that~$\cA_i$ is hereditary) and
that the restriction $u'|_S$ is injective. An appeal to Lemma~\ref{lem:injectivization}
ensures the existence of a function $u:2^{<\omega}\to \omega$ as required in~(ii).

{\em Case~2. Player~II has a winning strategy.} We shall show that $\cA_0$ and~$\cA_1$ are countably separated.
For every $i\in\{0,1\}$, every finite round of the game
$$
	\xi = (m_0,\bar{m}_0,i_0,\ldots,m_k,\bar{m}_k,i_k)
$$
and every $\bar{m}\in \omega$, we define $c^i[\xi,\bar{m}]$
to be the set of all $m<\omega$ such that the strategy of Player~II chooses $1-i$ after
$$
	(m_0,\bar{m}_0,i_0,\dots,m_k,\bar{m}_k,i_k,m,\bar{m})
$$
is played. Notice that $c^0[\xi,\bar{m}]\cap c^1[\xi,\bar{m}] = \emptyset$. The countable family
of subsets of~$\omega$ that will witness the countable separation of~$\cA_0$ and~$\cA_1$ is:
$$
	\mathcal{C} := \bigl\{c\sub\omega : \,  c  \triangle c^0[\xi,\bar{m}]
	\, \mbox{ is finite for some }\bar{m}<\omega \mbox{ and some } \xi\bigr\}.
$$
Indeed, fix $a_0\in\mathcal{A}_0$ and $a_1\in\mathcal{A}_1$ such that $a_0\cap a_1 = \emptyset$.
Take $(\bar{n}_k)\in \omega^\omega$ such that $a_0=f((\bar{n}_k))$.
We say that a finite round of the game
$$
	\xi= (m_0,\bar{m}_0,i_0,\dots,m_k,\bar{m}_k,i_k)
$$
is {\em acceptable} if it is played according to the strategy of Player~II and the following conditions hold:
\begin{enumerate}
\item[(a)] $\bar{m}_j = \bar{n}_j$ for all $j$;
\item[(b)] $m_j \neq m_{j'}$ whenever $j\neq j'$;
\item[(c)] $\{m_j : i_j=0\}\subset a_0$;
\item[(d)] $\{m_j : i_j=1\}\subset a_1$.
\end{enumerate}
Let $\Xi$ be the family of all acceptable $\xi$'s.
Given $\xi= (m_0,\bar{n}_0,i_0,\dots,m_k,\bar{n}_k,i_k)\in \Xi$, we say that $\xi$ is {\em extensible}
if there exist $m\in \omega \setminus \{m_0,\dots,m_k\}$ and $i\in\{0,1\}$ such that
$$
	(m_0,\bar{n}_0,i_0,\dots,m_k,\bar{n}_k,i_k,m,\bar{n}_{k+1},i)\in \Xi.
$$
Then there is some $\xi\in \Xi$ which is not extensible, because otherwise there would exist
an infinite round of the game, played according to the strategy of Player~II, in which Player I wins.
Fix a non-extensible $\xi= (m_0,\bar{n}_0,i_0,\ldots,m_k,\bar{n}_k,i_k)\in \Xi$ and
define
$$
	c := \bigl(c^0[\xi,\bar{n}_{k+1}]\setminus\{m_0,\dots,m_k\} \bigr) \cup \{m_j\in a_0 : \, j\leq k\}\in\mathcal{C}.
$$
We claim that $a_0 \sub c$ and $a_1 \cap c=\emptyset$. Indeed, this follows at
once from the fact that $c^0[\xi,\bar{n}_{k+1}]\cap c^1[\xi,\bar{n}_{k+1}] = \emptyset$ and the inclusion
\begin{equation}\label{eqn:CXI}
	a_i\sub c^i[\xi,\bar{n}_{k+1}]\cup\{m_0,\dots,m_k\} \quad \mbox{for }i\in\{0,1\}.
\end{equation}
To check~\eqref{eqn:CXI}, fix $i\in \{0,1\}$ and note that for every $m\in a_i \setminus \{m_0,\dots,m_k\}$ the non-extendability of~$\xi$ ensures that
$$
	(m_0,\bar{n}_0,i_0,\dots,m_k,\bar{n}_k,i_k,m,\bar{n}_{k+1},i)\not\in \Xi
$$
and so $m\in c^{i}[\xi,\bar{n}_{k+1}]$. The proof of the theorem is over.
\end{proof}

\begin{lem}\label{relativecompact}
Let $D$ be a dense subset of a metric space~$E$ and let $\mathcal{K}_E(D)$ be the family of all subsets
of~$D$ which are relatively compact in~$E$, ordered by inclusion. Then $\mathcal{K}_E(D) \sim \mathcal{K}(E)$.
\end{lem}
\begin{proof}
The function $G: \mathcal{K}_E(D)\to \mathcal{K}(E)$ given by $G(A):=\overline{A}$ is Tukey, because
for each $K\in \cK(E)$ the set $K\cap D\in \cK_E(D)$ contains
every $A\in \cK_E(D)$ for which $G(A)=\overline{A}\subset K$.

In order to check the Tukey reduction $\mathcal{K}(E) \preceq \mathcal{K}_E(D)$,
fix $L\in \mathcal{K}(E)$. Given any $n<\omega$, let $\mathcal{U}_n^L$ be a finite cover of~$L$ by open balls of radius~$\frac{1}{n+1}$
such that $L \cap B\neq \emptyset$ for every $B\in \mathcal{U}_n^L$. Since $D$ is dense in~$E$, we can select
a finite set $F_n(L) \sub D \cap (\bigcup \mathcal{U}_n^L)$ such that $F_n(L) \cap B$ is a singleton for every $B\in \mathcal{U}_n^L$.
Define
$$
	F(L) := \bigcup_{n<\omega } F_n(L) \sub D.
$$
Clearly, $L \sub \overline{F(L)}$. We claim that $F(L)\in \cK_E(D)$. Indeed, let
$(x_m)$ be a sequence in~$F(L)$ with infinitely many distinct terms. Then we can find
two strictly increasing subsequences $(m_k)$ and $(n_k)$ such that $x_{m_k}\in F_{n_k}(L)$ for all $k<\omega$.
Pick a sequence $(y_k)$ in~$L$ such that the distance between $y_k$ and $x_{m_k}$ is less than
or equal to~$\frac{2}{n_k+1}$ for all~$k<\omega$. Since $L$ is compact, $(y_k)$
admits a convergent subsequence, and so does~$(x_{m_k})$. This proves that $\overline{F(L)}$ is compact.

Finally, we check that $F:\mathcal{K}(E) \to \mathcal{K}_E(D)$ is Tukey.
Indeed, if $R\in \mathcal{K}_E(D)$, then $\overline{R}\in \cK(E)$ contains
every $L \in \cK(E)$ satisfying $F(L)\subset R$ (since $L \sub \overline{F(L)}$).
\end{proof}

The following lemma will also be needed in Section~\ref{section:atclassification}.

\begin{lem}\label{lem:GeneralLemma}
Let $W$ be a set, $f:2^{<\omega} \to W$ a function, $\mathcal{A} \sub \mathcal{P}(W)$
and ``$\leq$'' a binary relation on~$\mathcal{P}(W)$ satisfying the following properties:
\begin{enumerate}
\item[(i)] $\cA$ is closed under finite unions;
\item[(ii)] for every $A,B \sub W$ with $B \sub \bigcup\{C: C \leq A\}$ we have $B \leq A$;
\item[(iii)] $f(S) \in \cA$ for every $1$-chain $S \sub 2^{<\omega}$;
\item[(iv)]  $f(S) \not\leq A$ for every $0$-chain $S \sub 2^{<\omega}$ and every $A\in \cA$.
\end{enumerate}
Then $[\mathfrak{c}]^{<\omega} \preceq (\cA,\leq)$.
\end{lem}
\begin{proof}
Let $G$ be the set of all $\sigma=(\sigma_k)\in 2^\omega$ having infinitely many $0$'s and infinitely many~$1$'s.
For such a~$\sigma$, we consider the following
$1$-chains contained in~$2^{<\omega}$:
\begin{multline*}
	S_1(\sigma):=\big\{(\sigma_0,\dots,\sigma_k): \, \sigma_{k+1}=1\big\}
	\quad\mbox{and}\quad \\
	S_2(\sigma):=\big\{(1-\sigma_0,\dots,1-\sigma_k): \, \sigma_{k+1}=0\big\},
\end{multline*}
so that $F(\sigma):=f(S_1(\sigma))\cup f(S_2(\sigma)) \in \cA$ (by~(i) and~(iii)).

{\sc Claim.} {\em For every $A\in \cA$ the set $\Omega_{A}:=\big\{\sigma\in G: \, F(\sigma) \leq A\big\}$ is finite.}
Our proof is by contradiction. Suppose $\Omega_{A}$ is infinite
and let $(\sigma^n)$ be a sequence in~$\Omega_{A}$ of pairwise distinct elements.
Write $\sigma^n=(\sigma^n_k)$ for each $n<\omega$.
By passing to a subsequence, we can assume that $(\sigma^n)$ converges to some~$\sigma = (\sigma_k)\in 2^\omega$
such that $\sigma\neq \sigma^n$ for all $n<\omega$. Now, we
can find recursively two strictly increasing subsequences $(n_j)$ and $(p_j)$ of~$\omega$
such that, for every $j<\omega$, we have $\sigma^{n_j}_k=\sigma_k$ for all $k\leq p_j$ and $\sigma^{n_j}_{p_j+1}\neq \sigma_{p_j+1}$.
Let $(j_m)$ be a strictly increasing subsequence of~$\omega$ for which
the sequence $(\sigma_{p_{j_m}+1})$ is constant. There are two cases:
\begin{itemize}

\item If $\sigma_{p_{j_m}+1}=0$ for every $m<\omega$, then the set
$\{(\sigma_0,\dots,\sigma_{p_{j_m}}) : m<\omega\} \sub 2^{<\omega}$
is a $0$-chain, hence $B:=\{f(\sigma_0,\dots,\sigma_{p_{j_m}}) : m<\omega\} \not\leq A$ by property~(iv).
On the other hand, we have
$$
	f(\sigma_0,\dots,\sigma_{p_{j_m}})=
	f(\sigma_0^{n_{j_m}},\dots,\sigma_{p_{j_m}}^{n_{j_m}})
	\in f(S_1({\sigma^{n_{j_m}}})) \sub F(\sigma^{n_{j_m}}) \leq A
$$
for every $m<\omega$, hence property~(ii) yields~$B\leq A$, which is a contradiction.

\item If $\sigma_{p_{j_m}+1}=1$ for every $m<\omega$, the argument is similar. Indeed, the set
$\{(1-\sigma_0,\dots,1-\sigma_{p_{j_m}}) : m<\omega\} \sub 2^{<\omega}$
is a $0$-chain and so (iv) yields
$$
	B:=\{f(1-\sigma_0,\dots,1-\sigma_{p_{j_m}}) : m<\omega\} \not\leq A.
$$
On the other hand,
$$
	f(1-\sigma_0,\dots,1-\sigma_{p_{j_m}})=
	f(1-\sigma_0^{n_{j_m}},\dots,1-\sigma_{p_{j_m}}^{n_{j_m}})
	\in f(S_2({\sigma^{n_{j_m}}})) \sub F(\sigma^{n_{j_m}}) \leq A
$$
for every $m<\omega$. It follows from~(ii) that $B\leq A$, again a contradiction.
\end{itemize}
This proves that $\Omega_{A}$ is finite, as claimed.

Finally, the mapping
$$
	\bar{F}:[G]^{<\omega} \to (\cA,\leq),
	\quad
	\bar{F}(B) := \bigcup_{\sigma\in B} F(\sigma),
$$
is Tukey. Indeed, take any $A\in \mathcal{A}$. If $B \sub G$ is finite and $\bar{F}(B) \leq A$,
then for every $\sigma\in B$ we have $F(\sigma) \sub \bar{F}(B) \leq A$ and so
$F(\sigma)\leq A$ (by~(ii)); therefore, $\sigma\in \Omega_A$. It follows
that $B \sub \Omega_A$. This shows that $\bar{F}$ is Tukey.
Since $G$ has cardinality~$\mathfrak{c}$, we get $[\mathfrak{c}]^{<\omega} \preceq (\cA,\leq)$.
\end{proof}

Given a family $\mathcal{I}$ of subsets of~$\omega$, its {\em orthogonal} is defined by
$$
	\mathcal{I}^\perp=\{a\sub \omega: \, a\cap b \mbox{ is finite for every }b\in \mathcal{I}\}.
$$

\begin{thm}[\axiom]\label{projectiveideal}
Let $\mathcal{I}$ be an analytic family of subsets of $\omega$. Then
$\mathcal{I}^\perp$ (ordered by inclusion)
is Tukey equivalent to either $\{0\}$, $\omega$, $\omega^\omega$, $\mathcal{K}(\mathbb{Q})$ or $[\mathfrak{c}]^{<\omega}$.
\end{thm}

\begin{proof}
Notice that the family $\{a\sub\omega : \, a\subset b \mbox{ for some }b \in\mathcal{I}\}$
is also analytic, so we can assume that $\mathcal{I}$ is hereditary.
Since $\mathcal{I}^\perp$ is coanalytic and hereditary,
Theorem~\ref{projectivestronggap} can be applied to the families~$\mathcal{I}$ and~$\mathcal{I}^\perp$.
Two cases arise.

{\em Case~1. $\mathcal{I}$ and~$\mathcal{I}^\perp$ are countably separated.}
Recall that $\mathbb{Q}$ is identified with the subspace of~$2^\omega$ made up of all eventually
zero sequences.
By \cite[Theorem~10 and Proposition~11]{avi-tod}, there exist a bijection $f:\omega \to \qu$
and an analytic set $F\subset 2^\omega$ such that
${\rm acc}(f(a)) \sub F$ (resp. ${\rm acc}(f(a)) \sub 2^\omega \setminus F$) for every $a\in \mathcal{I}$
(resp. $a\in \mathcal{I}^\perp$). Here we write ${\rm acc}(C)$ to denote the set of all accumulation points of a set~$C\sub 2^\omega$.
Let $K$ be the disjoint union $\omega \cup 2^\omega$ equipped with the compact metrizable topology defined by:
\begin{itemize}
\item each $n<\omega$ is isolated in~$K$;
\item for each $x\in 2^\omega$, the collection
$$
	\bigl\{f^{-1}(O\setminus\{x\})\cup O: \, O \mbox{ is an open neighborhood of }x \mbox{ in }2^\omega \Bigr\}
$$
is a basis of neighborhoods of~$x$ in~$K$.
\end{itemize}
This is just a variant of the Alexandroff duplicate of $2^\omega$ in which we are duplicating only rational numbers. In this way, the topology inherited by~$2^\omega$ as a subspace of~$K$ coincides
with its usual topology. Define $E_0:=\omega \cup (2^\omega \setminus F) \sub K$,
so that $E_0$ is coanalytic in~$K$ and $\omega$ is dense in~$E_0$. Let $g:\omega \to K$ be the identity mapping.

{\sc Claim.} {\em A set $a\sub \omega$ belongs to $\mathcal{I}^\perp$
if and only if $g(a)\in \cK_{E_0}(\omega)$.} Indeed, note first that
the condition $g(a)\in \cK_{E_0}(\omega)$ is equivalent to saying that $\overline{g(a)} \sub E_0$
(because $K$ is compact metrizable). Suppose that $a \in \mathcal{I}^\perp$.
Then ${\rm acc}(f(a)) \sub 2^\omega \setminus F \sub E_0$.
Take any $x\in \overline{g(a)}$. If $x\in \omega$, then $x\in E_0$. If $x\not\in \omega$, then
$x\in \overline{g(a)}\setminus g(a)\sub {\rm acc}(g(a))$ and so $x\in {\rm acc}(f(a)) \sub E_0$.
This proves that $\overline{g(a)} \sub E_0$ whenever $a\in \mathcal{I}^\perp$.
Conversely, let $a \sub \omega$ be a set satisfying~$\overline{g(a)} \sub E_0$ and take any $b\in \mathcal{I}$. Then
${\rm acc}(f(b)) \sub F$ and so
$$
	{\rm acc}(f(a\cap b)) \sub
	{\rm acc}(f(a)) \cap {\rm acc}(f(b)) \sub \overline{g(a)} \cap F \sub E_0 \cap F = \emptyset,
$$
hence $a\cap b$ is finite. This shows that $a\in \mathcal{I}^\perp$ and finishes the proof of the claim.

Now, the mapping $\mathcal{I}^\perp \to \cK_{E_0}(\omega)$ given by $a\mapsto g(a)$
is an isomorphism of partially ordered sets, hence
$\mathcal{I}^\perp \sim \cK_{E_0}(\omega)$. By Lemma~\ref{relativecompact}, we have
$\cK_{E_0}(\omega) \sim \cK(E_0)$ and so $\mathcal{I}^\perp \sim \cK(E_0)$. Finally,
an appeal to Fremlin's Theorem~\ref{Fr91classification}
allows us to conclude that $\mathcal{I}^\perp$ is Tukey equivalent to either $\{0\}$, $\omega$, $\omega^\omega$ or $\mathcal{K}(\mathbb{Q})$.

\emph{Case~2. There is an injective function $u:2^{<\omega}\to \omega$ such that
$u(S)\in \mathcal{I}$ (resp. $u(S)\in \mathcal{I}^\perp$)
whenever $S \sub 2^{<\omega}$ is a $0$-chain (resp. $1$-chain).}
In this case we have $\mathcal{I}^\perp\sim [\mathfrak{c}]^{<\omega}$. Indeed,
since $\mathcal{I}^\perp$
is upwards directed and has cardinality~$\leq \mathfrak{c}$, we have $\mathcal{I}^\perp \preceq [\mathfrak{c}]^{<\omega}$
(Remark~\ref{rem:CardinalTukey}). On the other hand, the Tukey
reduction $[\mathfrak{c}]^{<\omega} \preceq \mathcal{I}^\perp$ follows
from Lemma~\ref{lem:GeneralLemma} applied to the function $f:=u$ and the family $\mathcal{A}:=\mathcal{I}^\perp$
equipped with the inclusion relation. The proof of the theorem is finished.
\end{proof}

The following result is similar to~Lemma~\ref{relativecompact}, now dealing
with the weak topology of the Banach space~$X$.

\begin{lem}\label{relativeBanach}
Suppose $X$ is separable. Let $D \sub B_X$ be a norm dense set and let $\mathcal{RK}(D)$ denote the family of all subsets
of~$D$ which are relatively weakly compact, ordered by inclusion. Then $\mathcal{K}(B_X)
\sim \mathcal{RK}(D)$.
\end{lem}

\begin{proof}
It is clear that $F:\mathcal{RK}(D) \to \mathcal{K}(B_X)$, $F(A) := \overline{A}^w$, is a Tukey function.
Let us check that $\mathcal{K}(B_X)\preceq \mathcal{RK}(D)$. Fix $L\in \mathcal{K}(B_X)$. Let
$(x^L_n)$ be a norm dense sequence in~$L$ such that every element is repeated infinitely many times.
For each $n<\omega$, choose $y^L_n \in D$ such that $\|x^L_n - y^L_n\|\leq \frac{1}{n+1}$.
Since $L$ is weakly compact and $y^L_n \in L+\frac{1}{n+1}B_X$ for every $n<\omega$,
an appeal to Lemma~\ref{lem:Grothendieck} ensures that
the set $\{y^L_n: \, n<\omega\}$ is relatively weakly compact. Define
$$
	G:\mathcal{K}(B_X)\to \mathcal{RK}(D), \quad
	G(L) := \{y^L_n: \, n<\omega\}.
$$
To prove that $G$ is a Tukey function, fix $C_0 \in \mathcal{RK}(D)$ and define $K_0:=\overline{C_0}^w \in \cK(B_X)$. Take any
$L\in \mathcal{K}(B_X)$ satisfying $G(L) \sub C_0$. Given $x\in L$ and $\eps>0$, we can choose
$m<\omega$ large enough such that $\frac{1}{m+1}\leq \epsilon$ and
$\|x^L_{m}-x\|\leq \epsilon$, hence $\|y^L_{m}-x\|\leq 2\eps$. Therefore,
$L \sub \overline{G(L)} \sub \overline{C_0} \sub K_0$. This shows that $G$ is Tukey.
\end{proof}

The proof of the next elementary result is left to the reader.

\begin{lem}\label{lem:Description}
Let $(x_n)$ be a sequence in~$B_X$. Let $\Omega(x_n) \sub 2^\omega \times (B_{X^*})^\omega$ be the
set of all pairs $(a,(y_m^*))$ for which
the iterated limits
$$
		\lim_{n\in a}\lim_{m\to \infty} y_m^*(x_n) \quad
		\mbox{and} \quad
		\lim_{m\to \infty}\lim_{n\in a} y_m^*(x_n)
$$
exist and are distinct. Then $\Omega(x_n)$ is Borel when $2^\omega \times (B_{X^*})^\omega$
is equipped with the product topology induced by the usual topology of~$2^\omega$ and the $w^*$-topology of~$X^*$.
\end{lem}

\begin{pro}\label{pro:usoDobleLimite}
If $D \sub B_X$ is a countable set, then $\mathcal{RK}(D) \sim \mathcal{I}^\perp$ for some
analytic family~$\mathcal{I}$ of subsets of $\omega$.
\end{pro}
\begin{proof}
If $D$ is finite, then $\mathcal{RK}(D)\sim \{0\} \sim \mathcal{P}(\omega) =\mathcal{I}^\perp$ by taking $\mathcal{I}=\{\{\emptyset\}\}$.
Suppose now that $D$ is infinite and enumerate $D=\{x_n:n<\omega\}$.
Since $\overline{{\rm span}}(D)$ is separable, we can assume without loss of generality that $X$ is separable,
so that $(B_{X^*})^\omega$ is Polish when
equipped with the product topology induced by the $w^*$-topology of~$X^*$.
By Lemma~\ref{lem:Description}, the set
$$
	\mathcal{I}:=\bigl\{a\in 2^\omega: \, (a,(y_m^*))\in \Omega(x_n) \mbox{ for some }(y_m^*)\in (B_X^*)^\omega\bigr\}
$$
is analytic. We shall prove that $\mathcal{I}^\perp \sim \mathcal{RK}(D)$. To this end, note that
Grothendieck's double limit criterion (see e.g.~\cite[1.6]{flo})
says that a set $C \sub D$ is relatively weakly compact if and only if
for every subsequence $(x_{n_k})$ contained in~$C$ and every sequence
$(y_m^*)$ in~$B_X^*$, the iterated limits
$$
		\lim_{k\to \infty}\lim_{m\to \infty} y_m^*(x_{n_k})\quad\mbox{and}\quad \lim_{m\to \infty}\lim_{k\to \infty} y_m^*(x_{n_k})
$$
coincide whenever they exist. It is now easy to check the following statements:
\begin{itemize}
\item If $C \in \mathcal{RK}(D)$, then $F(C):=\{n\in \omega: x_n \in C\} \in \mathcal{I}^\perp$.
\item If $a\in \mathcal{I}^\perp$, then $G(a):=\{x_n:n\in a\} \in \mathcal{RK}(D)$.
\end{itemize}
Clearly, the functions $F$ and $G$
are mutually inverse and give an isomorphism between the partially ordered sets
$\mathcal{RK}(D)$ and~$\mathcal{I}^\perp$. In particular, $\mathcal{RK}(D) \sim \mathcal{I}^\perp$.
\end{proof}

By combining Theorem~\ref{projectiveideal} and Proposition~\ref{pro:usoDobleLimite} we get

\begin{cor}[\axiom]\label{cor:interrogacion}
If $D \sub B_X$ is a countable set, then $\mathcal{RK}(D)$ is Tukey equivalent to either $\{0\}$, $\omega$, $\omega^\omega$,
$\mathcal{K}(\mathbb{Q})$ or $[\mathfrak{c}]^{<\omega}$.
\end{cor}

We can now prove the main result of this section.

\begin{proof}[Proof of Theorem~\ref{PDKclassification}]
By Lemma~\ref{relativeBanach} and Corollary~\ref{cor:interrogacion} (applied to
a fixed countable norm dense set $D \sub B_X$), $\mathcal{K}(B_X)$ is Tukey equivalent to
either $\{0\}$, $\omega^\omega$, $\mathcal{K}(\mathbb{Q})$
or~$[\mathfrak{c}]^{<\omega}$ (remember that the case $\mathcal{K}(B_X)\sim \omega$ was excluded in
Proposition~\ref{KXKBX}).
\end{proof}

We finish this section by remarking that, even in the absence of analytic determinacy,
the arguments in this section can still provide some information in ZFC.

\begin{thm}\label{ZFCprojectiveideal}
Let $\mathcal{I}$ be an analytic family of subsets of~$\omega$. Then either $\mathcal{I}^\perp \sim [\mathfrak{c}]^{<\omega}$ or
$\mathcal{I}^\perp \preceq \mathcal{K}(\mathbb{Q}) \times [\omega_1]^{<\omega}$. Therefore, the same holds for
$\mathcal{K}(B_X)$ when $X$ is separable.
\end{thm}

\begin{proof}
We can distinguish two cases. In the first case, if $\mathcal{I}$ is not countably separated from some analytic subfamily
of~$\mathcal{I}^\perp$, then the argument of Case~2 of the proof of Theorem~\ref{projectiveideal} can be applied
(because Theorem~\ref{projectivestronggap} holds in ZFC when the families are both analytic, see Remark~\ref{rem:AvilesTodorcevic})
and we obtain $\mathcal{I}^\perp \sim [\mathfrak{c}]^{<\omega}$.

The second case is that $\mathcal{I}$ is countably separated from every analytic subset of $\mathcal{I}^\perp$.
Since $\mathcal{I}^\perp$ is coanalytic, we can write $\mathcal{I}^\perp = \bigcup_{\alpha<\omega_1} \mathcal{J}_\alpha$ where
each $\mathcal{J}_\alpha$ is analytic. The arguments of Case~1 of the proof of
Theorem~\ref{projectiveideal} can be applied to $\mathcal{I}^\perp$ and $\mathcal{J}_\alpha$ for each $\alpha<\omega_1$,
obtaining that $\mathcal{J}_\alpha \subset \mathcal{J}'_\alpha \subset \mathcal{I}^\perp$ where
$\mathcal{J}'_\alpha \sim \mathcal{K}_{E^\alpha_0}(\omega)$ for some coanalytic metric space $E^\alpha_0$ in which $\omega$ is a dense subset.
By Fremlin's Theorem~\ref{Fr91classification}, for every $\alpha < \omega_1$ there is a Tukey reduction
$\psi_\alpha:\mathcal{J}'_\alpha \to \mathcal{K}(\mathbb{Q})$. For every $a\in \mathcal{I}^\perp$
we choose $\alpha(a)<\omega_1$ such that $a\in\mathcal{J}_{\alpha(a)}$.
Then the function
$$
	\psi:\mathcal{I}^\perp\to \mathcal{K}(\mathbb{Q}) \times [\omega_1]^{<\omega},
	\quad
	\psi(a) := (\psi_{\alpha(a)}(a),\{\alpha(a)\}),
$$
is Tukey and the proof is over.
\end{proof}

\begin{cor}\label{cor:cfZFC}
Let $\mathcal{I}$ be an analytic family of subsets of~$\omega$. Then either $cf(\mathcal{I}^\perp) = \mathfrak{c}$ or
$cf(\mathcal{I}^\perp)\leq \mathfrak{d}$. Therefore, the same holds for $\mathcal{K}(B_X)$ when $X$ is separable.
\end{cor}

\section{Classification of $\mathcal{AK}(B_X)$ under analytic determinacy}\label{section:atclassification}

This section is devoted to the proof of the following:

\begin{thm}[\axiom]\label{introAKBXclassification}
If $X$ is separable, then one of the following holds:
\begin{enumerate}
\item[(i)] $\mathcal{AK}(B_X) \sim \cK(B_X) \sim \{0\}$,
\item[(iia)] {$\mathcal{AK}(B_X) \sim \omega$ and $\cK(B_X)\sim \omega^\omega$,}
\item[(iib)] {$\mathcal{AK}(B_X) \sim \cK(B_X) \sim \omega^\omega$,}
\item[(iii)] $\mathcal{AK}(B_X) \sim \cK(B_X) \sim \mathcal{K}(\mathbb{Q})$,
\item[(iv)] $\mathcal{AK}(B_X) \sim \cK(B_X) \sim [\mathfrak{c}]^{<\omega}$.
\end{enumerate}
\end{thm}

This will be obtained by gathering together Theorem~\ref{PDKclassification} and Propositions~\ref{AKtukeytop}
and~\ref{AKKQ} below. The general scheme is that we shall take the Tukey reductions that we constructed in the proof of
Theorem~\ref{PDKclassification} for $\mathcal{K}(B_X)$ and we shall apply some refinements, mainly of Ramsey-theoretic nature,
to obtain Tukey reductions for $\mathcal{AK}(B_X)$.

\begin{pro}[\axiom]\label{AKtukeytop}
If $X$ is separable, then $\mathcal{K}(B_X) \sim [\mathfrak{c}]^{<\omega}$ if and only if
$\mathcal{AK}(B_X)\sim [\mathfrak{c}]^{<\omega}$.
\end{pro}

\begin{pro}[\axiom]\label{AKKQ}
If $X$ is separable, then $\mathcal{K}(B_X) \sim \mathcal{K}(\mathbb{Q})$ if and
only if $\mathcal{AK}(B_X)\sim \mathcal{K}(\mathbb{Q})$.
\end{pro}

\begin{proof}[Proof of Theorem~\ref{introAKBXclassification}]
By Theorem~\ref{PDKclassification} we know that
$\mathcal{K}(B_X)$ is Tukey equivalent to either $\{0\}$, $\omega^\omega$, $\mathcal{K}(\mathbb{Q})$ or $[\mathfrak{c}]^{<\omega}$.
Of course, we have $\mathcal{K}(B_X) \sim \{0\}$ if and only if
$\mathcal{AK}(B_X) \sim \{0\}$ (if and only if $X$ is reflexive). By Proposition~\ref{AKtukeytop} we have
$$
	\mathcal{K}(B_X) \sim [\mathfrak{c}]^{<\omega} \quad \Longleftrightarrow \quad \mathcal{AK}(B_X)\sim [\mathfrak{c}]^{<\omega},
$$
while Proposition~\ref{AKKQ} says that
$$
	\mathcal{K}(B_X) \sim \mathcal{K}(\mathbb{Q}) \quad \Longleftrightarrow \quad \mathcal{AK}(B_X)\sim \mathcal{K}(\mathbb{Q}).
$$
Finally, $\cK(B_X)\sim \omega^\omega$ if and only if $\omega \preceq \mathcal{AK}(B_X) \preceq \omega^\omega$ (Corollary~\ref{as:4-cor}(ii)).
{An appeal to Lemma~\ref{lem:LV} finishes the proof.}
\end{proof}

We next prove Propositions~\ref{AKtukeytop} and~\ref{AKKQ}. We keep the notation of the previous section.

\subsection{Case $\mathcal{AK}(B_X)\sim [\mathfrak{c}]^{<\omega}$}

In this subsection we shall prove Proposition~\ref{AKtukeytop}.
To this end, the Ramsey principle needed is Lemma~\ref{Ramseychains}, which
is a corollary of Milliken's theorem~\cite{mil} (cf. \cite[Theorem 4]{avi-tod-3}).
Note that for any $i\in\{0,1\}$ the set $\mathcal{C}_i$ of all $i$-chains of~$2^{<\omega}$
is a closed subset of the compact metrizable space $2^{2^{<\omega}}$ (equipped with its usual product topology).
To state Lemma~\ref{Ramseychains} we need a definition.

\begin{defi}\label{defi:subtree}
A function $u: 2^{<\omega}\to 2^{<\omega}$ is called a {\em subtree} if for every $t,s\in 2^{<\omega}$ and $i\in \{0,1\}$ we have
$$
	t\smallfrown i \sqsubseteq s \quad \Longrightarrow u(t)\smallfrown i \sqsubseteq u(s).
$$
\end{defi}

Observe that any subtree maps $i$-chains to $i$-chains for $i\in \{0,1\}$.

\begin{lem}\label{Ramseychains}
Fix $i\in\{0,1\}$. Let $W$ be a finite set and $c:\mathcal{C}_i\to W$ an analytic measurable function
(i.e. for every $a\in W$ the set $c^{-1}(\{a\})$ belongs to the $\sigma$-algebra on~$\cC_i$
generated by the analytic sets). Then there exist a subtree $u:2^{<\omega}\to 2^{<\omega}$ and
$a\in W$ such that $c(u(S))=a$ for every $S\in \cC_i$.
\end{lem}

When using this principle, the elements of~$W$ are usually called \emph{colors}
and the function~$c$ is called a \emph{coloring} of the $i$-chains. In this language, Lemma~\ref{Ramseychains}
states that if we color the $i$-chains with finitely many colors in a suitably measurable way,
then we can pass to a subtree where all $i$-chains have the same color.

The following lemmas are needed to prove Proposition~\ref{AKtukeytop}. Some of them will also be useful
in other sections.

\begin{lem}\label{lem:PerfectSubtree}
Let $P \sub 2^\omega$ be perfect. Then there is a subtree $u:2^{<\omega} \to 2^{<\omega}$ such that
$u(2^{<\omega}) \sub \{\sigma|_n: \, \sigma\in P, \, n<\omega\}$.
\end{lem}
\begin{proof}
We define two functions~$u,\tilde{u}:2^{<\omega} \to 2^{<\omega}$ inductively. Set
$\tilde{u}(\emptyset):=\emptyset$. Suppose
$$
	\tilde{u}(t) \in T:=\{\sigma|_n: \, \sigma\in P, \, n<\omega\}
$$
has been constructed for $t\in 2^{<\omega}$. In order to define
$\tilde{u}(t\smallfrown 0), \tilde{u}(t\smallfrown 1) \in T$,
write $\tilde{u}(t)=\sigma|_n$ for some $\sigma\in P$ and $n<\omega$.
Since~$P$ is perfect, there is $\sigma'\in P \setminus \{\sigma\}$ such that $\sigma'|_n=\sigma|_n$.
If $m:=\min\{k\geq n:\sigma'|_{k+1}\neq \sigma|_{k+1}\}$, then we define
$\tilde{u}(t\smallfrown i):=\sigma|_m \smallfrown i$ for $i\in\{0,1\}$.
In addition, we define $u(t):=\sigma|_m$.
It is clear that $u$ satisfies the required properties.
\end{proof}

\begin{lem}\label{lem:Measurability}
If $X$ is separable, then the mapping $p:X^{**} \to \R$ given by
$$
	p(x^{**}):=d(x^{**},X)=\inf_{x\in X}\|x^{**}-x\|
$$
is ${\rm Borel}(X^{**},w^*)$-measurable.
\end{lem}
\begin{proof}
Let $X_0 \sub X$ be a countable dense set. Then for every $x^{**} \in X^{**}$ we have
\begin{equation}\label{eqn:p}
	p(x^{**})=\inf_{x\in X_0}\|x^{**}-x\|.
\end{equation}
Note that for each $x\in X$ the mapping $X^{**}\to \R$ given by
$$
	x^{**}\mapsto \|x^{**}-x\|=\sup_{x^*\in B_{X^*}} x^*(x^{**}-x)
$$
is $w^*$-lower semicontinuous (being the supremum of a collection of
$w^*$-continuous functions), hence ${\rm Borel}(X^{**},w^*)$-measurable.
Since $X_0$ is countable, from~\eqref{eqn:p} it follows that $p$ is ${\rm Borel}(X^{**},w^*)$-measurable.
\end{proof}

\begin{defi}
Let $\delta>0$. A sequence $(x_n)$ in~$X$ is called a {\em $\delta$-controlled} $\ell^1$-sequence
if it is bounded and
$$
	\delta\sum_{i=0}^n|a_i| \leq \left\|\sum_{i=0}^na_i x_i\right\|
$$
for every $n<\omega$ and every $a_0,\dots,a_n \in \erre$.
\end{defi}

\begin{lem}\label{lem:measurability}
Suppose $X$ is separable. Let $\{x_n:n<\omega\}\sub B_X$ be a countable set,
$f:2^{<\omega} \to \omega$ an injective function and $\delta>0$. Fix $i\in \{0,1\}$.
Each $S\in \mathcal{C}_i$ is enumerated as $S=\{S_n:n<\omega\}$ in
such a way that $S_n \smallfrown i \sqsubseteq S_{n+1}$ for all $n<\omega$.
Then:
\begin{enumerate}
\item[(i)] The set $W$ of all $S\in \cC_i$ such that $(x_{f(S_n)})$ is weakly Cauchy is coanalytic.
\item[(ii)] The set of all $S\in \cC_i$ such that $(x_{f(S_n)})$ is weakly Cauchy and
$$
	d\Big(w^*-\lim_{n\to \infty} x_{f(S_n)},X\Big) \geq \delta
$$
is coanalytic.
\item[(iii)] The set of all $S\in \cC_i$ such that $(x_{f(S_n)})$ is a $\delta$-controlled $\ell^1$-sequence is Borel.
\item[(iv)] The set of all $S\in \cC_i$ such that $(x_{f(S_n)})$ is an $\ell^1$-sequence is Borel.
\end{enumerate}
\end{lem}
\begin{proof} Note first that, for each $m<\omega$, the function
$$
	e_m: \cC_i \to X, \quad
	e_m(S):=x_{f(S_m)},
$$
is continuous to the norm topology of~$X$. Indeed, let $(S^k)$ be a sequence in~$\cC_i$ converging to~$S\in \cC_i$.
There is $k_0<\omega$ such that for every $k\geq k_0$ and every $s\in 2^{<\omega}$
with ${\rm length}(s)\leq {\rm length}(S_m)$ we have
$$
	s\in S^k \quad \Longleftrightarrow \quad s\in S,
$$
which implies that $S^k_n=S_n$ for all $n\leq m$. Hence $e_m(S^k)=e_m(S)$ for every $k\geq k_0$.

(i). We consider the set $B_{X^*}$ equipped with the $w^*$-topology. Since $X$ is separable, $(B_{X^*},w^*)$ is metrizable and so it is a Polish space.
By the continuity of the~$e_n$'s, the mapping
$$
	B_{X^*} \times \mathcal{C}_i \to \erre,\quad
	(x^*,S)\mapsto \big|x^*(e_{m}(S))-x^*(e_{k}(S))\big|,
$$
is separately continuous and so it is Borel (see e.g. \cite{burke} and references therein).
Thus, the set
$$
	E:=\bigcup_{\eps\in \mathbb{Q}^+}
	\bigcap_{n<\omega}\bigcup_{m,k\geq n}
	\Big\{(x^*,S)\in B_{X^*}\times \cC_i: \, \big|x^*(e_{m}(S))-x^*(e_{k}(S))\big| \geq \eps\Big\}
$$
is Borel in~$B_{X^*}\times \cC_i$. If $\pi: B_{X^*}\times \cC_i \to \cC_i$
stands for the second coordinate projection, then $\pi(E)$ is analytic. Clearly,
$\cC_i \setminus \pi(E)=W$, the set of all $S\in \cC_i$ for which $(x_{f(S_n)})$ is weakly Cauchy.

(ii). The function $e: W \to X^{**}$ defined by
$$
	e(S):=w^*-\lim_{n\to \infty} e_n(S)
$$
is Borel measurable from~$W$ to $(X^{**},w^*)$ (since each $e_n$ is $w^*$-continuous) and so,
by Lemma~\ref{lem:Measurability}, the real-valued function defined on~$W$ by
$S \mapsto d(e(S),X)$ is Borel. Since $W$ is coanalytic,
the set $\{S\in W: d(e(S),X)\geq \delta\}$ is coanalytic as well.

(iii). The set of all $S\in \cC_i$ such that $(e_n(S))$ is a $\delta$-controlled $\ell^1$-sequence can be written as
$$
	\bigcap_{I \in [\omega]^{<\omega}}
	\bigcap_{\phi\in \mathbb{Q}^I}
	\left\{
	S\in \cC_i: \,
	\delta \sum_{n\in I}|\phi(n)|
	\leq
	\Big\|\sum_{n\in I}\phi(n)e_n(S)\Big\|
	\right\}.
$$
Since each $e_n$ is continuous to the norm topology of~$X$, the set above is Borel in~$\cC_i$.
Finally, note that (iv) follows at once from~(iii).
\end{proof}

\begin{lem}\label{wCauchy}
Let $(x_n)$ be a bounded sequence in~$X$ such that
$$
	\eta(x_n):=\inf\{\|x^{**}-x\|: \, x^{**}\in {\rm clust}_{X^{**}}(x_n), \, x\in X\}>0
$$
where ${\rm clust}_{X^{**}}(x_n)$ stands for the set of all $w^*$-cluster points
of~$(x_n)$ in~$X^{**}$. Then for every $0<\delta<\eta(x_n)$ and every
$L\in \cK(X)$ the set $\{n<\omega: x_{n}\in L+\delta B_X\}$ is finite.
\end{lem}
\begin{proof}
Fix any $\delta>0$. Suppose there is $L\in \cK(X)$ such that $x_{n}\in L+\delta B_X$ for infinitely many $n$'s.
Then we can find a subsequence $(x_{n_k})$ and a weakly convergent sequence~$(y_k)$ in~$L$
such that $\|x_{n_k}-y_k\|\leq \delta$ for every~$k<\omega$.
Let $y\in L$ be the weak limit of~$(y_k)$ and fix an arbitrary $x^{**} \in {\rm clust}_{X^{**}}(x_{n_k})$.
Then
$$
	x^{**}-y \in {\rm clust}_{X^{**}}(x_{n_k}-y_k) \sub \delta B_{X^{**}},
$$
hence $\eta(x_n)\leq \delta$. This finishes the proof.
\end{proof}

The following lemma was proved in \cite[Lemma~5]{kal-alt}.

\begin{lem}\label{l1seq}
Let $\delta>0$. If $(x_n)$ is a $\delta$-controlled $\ell^1$-sequence in~$X$,
then $\eta(x_n) \geq \delta$.
\end{lem}

Recall that $\mathcal{RK}(B_X)$ is the family of all relatively
weakly compact subsets of~$B_X$.

\begin{lem}\label{lem:TukeyTop}
Suppose there exist $\delta>0$ and a function $f:2^{<\omega} \to B_X$ such that:
\begin{enumerate}
\item[(i)] $f(S)\in \mathcal{RK}(B_X)$ for every $S \in \cC_0$;
\item[(ii)] $f(S)\not \subseteq L+\delta B_X$ for every $S \in \cC_1$ and every $L \in \cK(B_X)$.
\end{enumerate}
Then $[\mathfrak{c}]^{<\omega} \preceq \mathcal{AK}(B_X)$.
\end{lem}
\begin{proof}
Consider the binary relation on~$\mathcal{P}(B_X)$ defined by
$$
	A \leq_\delta B
	\quad
	:\Longleftrightarrow \quad
	A \sub B+\delta B_X.
$$
By (ii) we have $f(S)\not\leq_\delta A$ for every $1$-chain $S \sub 2^{<\omega}$
and every $A\in\mathcal{RK}(B_X)$. Lemma~\ref{lem:GeneralLemma} applied to the function~$f$, the
family~$\mathcal{RK}(B_X)$ and the binary relation~$\leq_\delta$
ensures that $[\mathfrak{c}]^{<\omega} \preceq (\mathcal{RK}(B_X),\leq_\delta)$.
Since the function
$$
	g: (\mathcal{RK}(B_X),\leq_\delta) \to (\mathcal{AK}(B_X),\leq_\delta),
	\quad
	g(A):=\overline{A}^{w},
$$
is Tukey, we conclude that $[\mathfrak{c}]^{<\omega} \preceq \mathcal{AK}(B_X)$.
\end{proof}

We can now prove the main result of this subsection.

\begin{proof}[Proof of Proposition~\ref{AKtukeytop}]
The ``if part'' (valid in ZFC) follows from Propositon~\ref{at:7}(i)
and the fact that $\cK(B_X)\preceq [\mathfrak{c}]^{<\omega}$ whenever $X$ is separable (by Remark~\ref{rem:CardinalTukey}).

Conversely, suppose now that $\cK(B_X)\sim [\mathfrak{c}]^{<\omega}$.
Let $D=\{x_n:n<\omega\}$ be a countable dense subset of~$B_X$
and let $\mathcal{I}$ be the set of all $A \sub \omega$
for which $\{x_n:n\in A\}\in \mathcal{RK}(D)$, i.e. $\{x_n:n\in A\}$ is relatively weakly compact.
By the proof of Proposition~\ref{pro:usoDobleLimite}, there is
an analytic family~$\mathcal{I}_0$ of subsets of $\omega$ such that $\mathcal{I}_0^\perp=\mathcal{I} \sim \mathcal{RK}(D)$.
Hence Theorem~\ref{projectiveideal} can be applied to~$\mathcal{I}_0$.
Since
$$
	\mathcal{I}_0^\perp \sim \mathcal{RK}(D)\sim \mathcal{K}(B_X) \sim [\mathfrak{c}]^{<\omega}
$$
(bear in mind Lemma~\ref{relativeBanach}), Case~2 in the proof of Theorem~\ref{projectiveideal} occurs and
so there is an injective function $u:2^{<\omega}\to \omega$ such that
$u(S)\in \mathcal{I}$ (resp. $u(S)\in \mathcal{I}^\perp$)
for every $0$-chain (resp. $1$-chain) $S \sub 2^{<\omega}$. Let $c: \cC_1 \to \{0,1,2\}$
be the coloring defined by
$$
	c(S):=\begin{cases}
	0 & \text{if $(x_{u(S_n)})$ is weakly Cauchy}, \\
	1 & \text{if $(x_{u(S_n)})$ is an $\ell^1$-sequence}, \\
	2 & \text{otherwise}. \\
	\end{cases}
$$
(Here we follow the notation of Lemma~\ref{lem:measurability}.)
Since $c$ is analytic measurable (by Lemma~\ref{lem:measurability}),
we can apply Lemma~\ref{Ramseychains} to find a subtree $v:2^{<\omega} \to 2^{<\omega}$
such that $c(v(\cdot))$ is constant on~$\cC_1$. Rosenthal's $\ell^1$-theorem (see e.g. \cite[Theorem~5.37]{fab-ultimo})
states that every bounded sequence in a Banach space contains either a weakly Cauchy subsequence or an $\ell^1$-subsequence, so it is
impossible that $c(v(S))=2$ for every $S\in \cC_1$. Writing $w:=u\circ v$, we are therefore reduced to consider two cases:

{\sc Case 1:} $(x_{w(S_n)})$ is weakly Cauchy for every $S\in \cC_1$. The $w^*$-limit $x_S^{**}$
of such a sequence belongs to~$X^{**}\setminus X$, because
otherwise $\{x_{w(S_n)}:n<\omega\}\in \mathcal{RK}(D)$, i.e. $u(v(S))\in \mathcal{I}$,
while $u(v(S))\in \mathcal{I}^\perp$ because $v(S)$ is a $1$-chain; this is a contradiction,
since $u(v(S))$ is infinite.

Let $G \sub 2^\omega$ be the set of all $\sigma = (\sigma_k)\in 2^\omega$ such that $\sigma_k=1$ for infinitely many $k$'s and,
for any such $\sigma$, consider
$$
	\sigma^{(1)}:=\{(\sigma_0,\dots,\sigma_{k}): \, \sigma_{k+1} = 1\} \in \cC_1.
$$
Note that $G$ is a $\cG_\delta$ subset of~$2^\omega$, hence $G$ is Polish. Write $G=\bigcup_{n<\omega}A_n$, where
$$
	A_n := \Big\{\sigma\in G : \, d\big(x_{\sigma^{(1)}}^{**},X\big)> \frac{1}{n+1}\Big\}.
$$
Each $A_n$ is Borel, by the proof of Lemma~\ref{lem:measurability}(ii) and the continuity
of the mapping $G \to \cC_1$ given by $\sigma\mapsto \sigma^{(1)}$ (which can be proved easily).
Fix $m<\omega$ such that $A_m$ is uncountable. Then $A_m$ contains a set $P$ homeomorphic to~$2^\omega$
(see e.g. \cite[Theorem~13.6]{kec-J}). By Lemma~\ref{lem:PerfectSubtree},
there is a subtree $\xi:2^{<\omega} \to 2^{<\omega}$ such that
$\xi(2^{<\omega}) \sub T:=\{\sigma|_n: \, \sigma\in P, \, n<\omega\}$.

{\em Claim:} For every $1$-chain $S \sub 2^{<\omega}$ and every $L \in \cK(B_X)$ we have
$$
	\{x_{w(\xi(S_n))}:n<\omega\} \not \subseteq L + \frac{1}{m+1}B_X.
$$
Indeed, $\xi(S)$ is also a $1$-chain, hence there exist
$\sigma\in G$ and a strictly increasing sequence $(k_n)$ in~$\omega$
such that $\xi(S_n)=\sigma|_{k_n}$ and $\sigma_{k_n}=1$ for all $n<\omega$.
Since $\xi(2^{<\omega})\sub T$, we have $\sigma \in P \sub A_m$.
Note that $\sigma^{(1)} \supseteq \xi(S)$ and so
$x^{**}_{\sigma^{(1)}}=x^{**}_{\xi(S)}$, therefore
$$
	d\Big(w^*-\lim_{n\to\infty}x_{w(\xi(S_n))},X\Big)> \frac{1}{m+1}.
$$
An appeal to Lemma~\ref{wCauchy} finishes the proof of the claim.

{\sc Case 2:} $g(S):=(x_{w(S_n)})$ is an $\ell^1$-sequence for every $S \in \cC_1$.
We can write $G=\bigcup_{n<\omega} B_n$, where
$$
	B_n := \Big\{
	\sigma\in G : \, g(\sigma^{(1)}) \mbox{ is } \frac{1}{n+1}\mbox{-controlled}
	\Big\}.
$$
By Lemma~\ref{lem:measurability}(iii) and the continuity
of the mapping $G \to \cC_1$ given by $\sigma\mapsto \sigma^{(1)}$,
each $B_n$ is Borel, and so it contains a subset homeomorphic to~$2^\omega$
whenever $B_n$ is uncountable. As in Case~1,
there exist $m<\omega$ and a subtree $\xi:2^{<\omega}\to 2^{<\omega}$ such that,
for every $1$-chain $S \sub 2^{<\omega}$, the sequence
$(x_{w(\xi(S_n))})$ is a $\frac{1}{m+1}$-controlled $\ell^1$-sequence,
which implies that
$$
	\{x_{w(\xi(S_n))}:n<\omega\} \not \subseteq L + \frac{1}{2(m+1)}B_X.
$$
for every $L \in \cK(B_X)$ (by Lemmas \ref{wCauchy} and~\ref{l1seq}).

In any of the two cases, we can apply Lemma~\ref{lem:TukeyTop} to
the function $f:2^{<\omega} \to B_X$ given by $f(t):=x_{w(\xi(t))}$
to conclude that $[\mathfrak{c}]^{<\omega} \preceq \mathcal{AK}(B_X)$.
\end{proof}

\subsection{Case $\mathcal{AK}(B_X)\sim \cK(\mathbb{Q})$}

This subsection is devoted to sketching the proof of Proposition~\ref{AKKQ}.
To this end, we shall use a Ramsey theorem recently proved in~\cite{avi-tod-IHES}.
Let us try to summarize the concepts and facts that we need, that can found in detail in~\cite{avi-tod-IHES} and in the
more extended preprint \cite{avi-tod-2}.
Throughout this subsection $2^{<\omega}$ is equipped with the topology inherited
from~$2^\omega$ via the  injection $\gamma:2^{<\omega}\to 2^\omega$ given by
 $$\gamma(s) = \begin{cases}
 s \smallfrown 1\smallfrown {\bf 0} \text{ if } s \text{ is not made exclusively of zeros}\\
 {\bf 0} \text{ if } s = \emptyset\\
 0^{n-1} \smallfrown 1 \smallfrown {\bf 0} \text{ if } s=0^n \text{ consists of } n \text{ many zeros},
 \end{cases}$$
that bijects $2^{<\omega}$ with the sequences that are eventually zero, and that makes $2^{<\omega}$ homeomorphic to~$\mathbb{Q}$.
There are eight special \emph{types} of infinite subsets of~$2^{<\omega}$, denoted as
$$
	\mathfrak{T} = \{[0],[1],[01],[^0{}_1],[^1{}_0], [^{01}{}_1],[^1{}_{01}], [_0{}^1{}_1]\}.
$$

The sets of type $[0]$, $[1]$ and $[01]$ are chains and the rest are antichains. A chain $\{s_0,s_1,\ldots\}$ is of type $[0]$ if $s_{n+1}$ is the result of adding only zeroes to the previous element $s_n$. It is of type $[01]$ if, for every $n$, $s_{n+1}$ is of the form $s_n \smallfrown 0 \smallfrown t_n$ where $t_n$ is not made exclusively of zeros, and it is of type $[1]$ if, for every $n$, $s_{n+1}$ is of the form $s_n \smallfrown 1 \smallfrown t_n$ for some arbitrary $t_n$. A set of one of the other types is of the form $\{s_0 \smallfrown r_0,s_1 \smallfrown r_1,\ldots\}$ where $\{s_0,s_1,\ldots\}$ is a chain of the type indicated by the lower row in the square-bracket expression of the type, and, for every $n$, $\{s_n,s_n \smallfrown r_n\}$ could be the first two elements of a chain of the type indicated by the upper row. The types $[^1{}_{01}]$ and $[_0{}^1{}_1]$ are distinguished by further details that we omit since they are irrelevant to our discussion. Some properties are the following:
\begin{itemize}
\item[(T1)] For every set $A \sub 2^{<\omega}$ of type~$\tau$ and every set $B\sub 2^{<\omega}$
of type $\tau'\neq \tau$, the intersection $A\cap B$ is finite.
\item[(T2)] Every infinite subset of $2^{<\omega}$ contains a further infinite subset which is of type~$\tau$,
for some $\tau\in\mathfrak{T}$.
\item[(T3)] Every set of type $\tau\in\{[0],[^1{}_0]\}$ is relatively
compact in $2^{<\omega}$, while every set of type $\tau\in\mathfrak{T}\setminus\{[0],[^1{}_0]\}$ is closed and discrete in $2^{<\omega}$. This is because, according to the description given above, sets of type $[0]$ and $[^1{}_0]$ converge to an infinite sequence of $2^\omega$ that is eventually zero, while sets of any other type converge to a sequence that has infinitely many ones.
\item[(T4)] For every $\tau\in\mathfrak{T}$, the family $\mathcal{S}_\tau$ of all sets of type~$\tau$ is a $\cG_\delta$ subset
of the compact metrizable space~$2^{2^{<\omega}}$, hence it can be viewed as a Polish space.
\end{itemize}

Associated to these types, we have the notion of {\em nice embedding}. The precise definition can be found in~\cite{avi-tod-IHES}, but we do not give it here, let us just point out that a nice embedding is an injective function $u:2^{<\omega}\to 2^{<\omega}$ that has the following properties:
\begin{enumerate}
\item[(N1)] For every $s,t\in 2^{<\omega}$, if ${\rm length}(s)<{\rm length}(t)$,
then ${\rm length}(u(s))<{\rm length}(u(t))$.
\item[(N2)] For every $s\in 2^{<\omega}$, we have:
\begin{itemize}
\item $u(s\smallfrown 0)=u(s)\smallfrown 0 \smallfrown \dots \smallfrown 0$ for some number of $0$'s;
\item $u(s\smallfrown 1) \sqsupseteq u(s)\smallfrown 1$.
\end{itemize}
\item[(N3)] For every $\tau \in \mathfrak{T}$, a set $A\subset 2^{<\omega}$ has type~$\tau$ if and only if $u(A)$ has type~$\tau$.
\item[(N4)] $u$ is a homeomorphism from $2^{<\omega}$ onto a closed subset of $2^{<\omega}$.
\end{enumerate}
The Ramsey theorem from~\cite{avi-tod-IHES} that we mentioned is the following:

\begin{thm}\label{finitepartition}
Fix $\tau\in\mathfrak{T}$. Let $W$ be a finite set and $c:\mathcal{S}_\tau\to W$ an analytic measurable function
(i.e. for every $a\in W$ the set $c^{-1}(\{a\})$ belongs to the $\sigma$-algebra on~$\mathcal{S}_\tau$
generated by the analytic sets).
Then there is a nice embedding $u:2^{<\omega}\to 2^{<\omega}$ such that $c$ is constant on $\{u(A) : A\in \mathcal{S}_\tau\}$.
\end{thm}

In addition, we shall need some extension of this result for countable partitions, that holds only for some of the types;
see~\cite[Lemma 2.8.1]{avi-tod-2}.

\begin{lem}\label{countablepartition}
Fix $\tau\in\mathfrak{T}\setminus\{[0], [^1{}_0]\}$. Let $c:\mathcal{S}_\tau\to\omega$ be an
analytic measurable function such that $c(A) = c(B)$ whenever $A \triangle B$ is finite.
Then there is a nice embedding $u:2^{<\omega}\to 2^{<\omega}$ such that $c$ is constant on $\{u(A) : A\in \mathcal{S}_\tau\}$.
\end{lem}

\begin{proof}[Proof of Proposition~\ref{AKKQ}]
Suppose first that $\mathcal{AK}(B_X)\sim \mathcal{K}(\mathbb{Q})$. Then
$$
	\mathcal{K}(\mathbb{Q})\sim \mathcal{AK}(B_X) \preceq \mathcal{K}(B_X)
$$
(apply Proposition~\ref{at:7}(i)).
By Theorem~\ref{PDKclassification}, we get that either $\mathcal{K}(B_X)\sim\mathcal{K}(\mathbb{Q})$ or $\mathcal{K}(B_X)\sim [\mathfrak{c}]^{<\omega}$.
An appeal to Proposition~\ref{AKtukeytop} yields $\mathcal{K}(B_X) \sim \mathcal{K}(\mathbb{Q})$, as desired.

Suppose now that $\mathcal{K}(B_X)\sim \mathcal{K}(\mathbb{Q})$. We divide the proof
that $\mathcal{AK}(B_X)\sim \mathcal{K}(\mathbb{Q})$ into several steps.

{\sc Step~1.} There is an injective function
$\upsilon:2^{<\omega} \to B_X$ such that:
\begin{itemize}
\item[(*)] a set $A\subseteq 2^{<\omega}$ is relatively compact in~$2^{<\omega}$
if and only if $\upsilon(A)$ is relatively weakly compact in~$X$.
\end{itemize}
To prove this, let $D=\{x_n:n<\omega\}\sub B_X$ be a countable dense set.
By Proposition~\ref{pro:usoDobleLimite}, there is an analytic family $\mathcal{I}$
of subsets of~$\omega$ such that $\mathcal{RK}(D)\sim \mathcal{I}^\perp$.
Theorem~\ref{projectiveideal} can now be applied to~$\mathcal{I}$.
By Lemma~\ref{relativeBanach}, we have
$$
	\mathcal{I}^\perp \sim \mathcal{RK}(D) \sim \mathcal{K}(B_X) \sim \mathcal{K}(\mathbb{Q})
$$
and so Case~1 in the proof of Theorem~\ref{projectiveideal} occurs.
Therefore, there is a metric space~$(E_0,\rho)$ such that:
\begin{itemize}
\item $E_0$ is coanalytic (in some Polish space);
\item $\omega$ is a dense subset of~$E_0$;
\item a set $C \subseteq \omega$ is relatively compact in~$E_0$ if and only if $\{x_n:n\in C\}$ is relatively weakly compact;
\item $\mathcal{K}(E_0)\sim \mathcal{I}^\perp$.
\end{itemize}
Bearing in mind Fremlin's Theorem~\ref{Fr91classification}, we deduce that $E_0$ is not Polish.

By Hurewicz's theorem (see e.g.~\cite[21.18]{kec-J}), there is a closed set $F\subset E_0$ which is
homeomorphic to~$\mathbb{Q}$ (and so homeomorphic to~$2^{<\omega}$).
Enumerate $F = \{f_n:n<\omega\}$ and consider a set $G=\{g_n:n<\omega\}\subseteq \omega$ such that
$\rho(f_n,g_n)\leq \frac{1}{n+1}$ for all $n<\omega$. Let $\vf:2^{<\omega}\to \omega$ be a bijection such that
the induced mapping $t \mapsto f_{\vf(t)}$ is a homeomorphism between $2^{<\omega}$ and~$F$. It is
clear that the function
$$
	\upsilon:2^{<\omega} \to B_X,
	\quad
	\upsilon(t):=x_{g_{\varphi(t)}},
$$
satisfies the required properties.

{\sc Step~2.} For every infinite set $A \sub 2^{<\omega}$ we define $c_{\upsilon}(A) \in \{0,1,2\}$ by
$$
	c_\upsilon(A):=\begin{cases}
	0 & \text{if $\upsilon(A)$ is weakly Cauchy}, \\
	1 & \text{if $\upsilon(A)$ is an $\ell^1$-sequence}, \\
	2 & \text{otherwise}, \\
	\end{cases}
$$
where $\upsilon(A)$ is ordered as a sequence following the order $\leq_{lex}$ of~$2^{<\omega}$
(as defined in the proof of Lemma~\ref{lem:injectivization}).
We can assume without loss of generality that $c_\upsilon$ is constant (equal to~$0$ or~$1$) on~$\mathcal{S}_\tau$ for every $\tau\in \mathfrak{T}$.
To see this, enumerate $\mathfrak{T}=\{\tau_0,\dots,\tau_7\}$.
Since the restriction of $c_\upsilon$ to~$\mathcal{S}_{\tau_0}$ is analytic measurable,
Theorem~\ref{finitepartition} ensures the existence of a nice embedding~$u_0:2^{<\omega}\to 2^{<\omega}$
such that $c_\upsilon$ is constant on $\{u_0(A) : A\in \mathcal{S}_{\tau_0}\}$.
On the other hand, since $u_0$ is a homeomorphism onto a closed subset of~$2^{<\omega}$
(by property~(N4)), the composition $\upsilon_0:=\upsilon\circ u_0$
also has property~(*) above. Clearly,
$c_{\upsilon_0}$ is constant on~$\mathcal{S}_{\tau_0}$.
By the same argument, now applied to the mapping $c_{\upsilon_0}$ and the type~$\tau_1$,
there is a nice embedding~$u_1:2^{<\omega}\to 2^{<\omega}$
such that $c_{\upsilon_0}$ is constant
on $\{u_1(A) : A\in \mathcal{S}_{\tau_1}\}$. Note that $\upsilon_1:=\upsilon_0 \circ u_1$
satisfies property~(*) and that $c_{\upsilon_1}$ is constant on~$\mathcal{S}_{\tau_0}$
and also on~$\mathcal{S}_{\tau_1}$. By continuing in this way, we find
an injective function $\tilde{\upsilon}:2^{<\omega}\to B_X$ satisfying property~(*)
such that $c_{\tilde{\upsilon}}$ is constant on~$\mathcal{S}_\tau$ for every $\tau\in \mathfrak{T}$.
Bearing in mind Rosenthal's $\ell^1$-theorem (see e.g. \cite[Theorem~5.37]{fab-ultimo}), the constant value cannot be~$2$.

{\sc Step~3.} Let $\tau\in\mathfrak{T}\setminus\{[0],[^1{}_0]\}$ such that $c_\upsilon\equiv 0$ on~$\mathcal{S}_\tau$.
Any $A\in \mathcal{S}_\tau$ is infinite, closed and discrete in~$2^{<\omega}$, hence it is not relatively compact
and so $\upsilon(A)$ is not relatively weakly compact. Since $\upsilon(A)$ is weakly Cauchy, it converges
to some $x^{\ast\ast}_A\in X^{\ast\ast}\setminus X$. The function $\hat{c}:\mathcal{S}_\tau\to\N$ given by
$$
	\hat{c}(A) := \min\Big\{n\in \N: \, d\big(x^{\ast\ast}_A,X\big)>\frac{1}{n}\Big\}
$$
satisfies the requirements of Lemma~\ref{countablepartition} (use Lemma~\ref{lem:Measurability}),
so there is a nice embedding $u_\tau:2^{<\omega} \to 2^{<\omega}$ such that $\hat{c}$ is constant on~$\{u_\tau(A):A\in \mathcal{S}_\tau\}$,
that is, there is $\delta_\tau>0$ such that for every $A\in \mathcal{S}_\tau$ we have $d(x_{u_\tau(A)}^{**},X)> \delta_\tau$,
so $\upsilon(u_\tau(A))\not\subseteq L+\delta_\tau B_X$ for any weakly compact set $L\sub X$.

{\sc Step~4.} Let $\tau\in\mathfrak{T}\setminus\{[0],[^1{}_0]\}$ such that $c_\upsilon\equiv 1$ on~$\mathcal{S}_\tau$.
Define a function
$$
	\hat{c}:\mathcal{S}_\tau\to\N
$$
by declaring $\hat{c}(A)$ to be the least $n\in \N$ for which there is a finite set $F\subset A$ such that
$\upsilon(A \setminus F)$ is a $\frac{1}{n}$-controlled $\ell^1$-sequence. Since
$\hat{c}$ is analytic measurable, we can apply Lemma~\ref{countablepartition}
to obtain a nice embedding $u_\tau:2^{<\omega}\to 2^{<\omega}$ such that $\hat{c}$ is constant (say, equal to~$n_\tau$)
on~$\{u_\tau(A):A\in \mathcal{S}_\tau\}$. Writing $\delta_\tau:=\frac{1}{2n_\tau}$,
an appeal to Lemma~\ref{l1seq} ensures that for every $A\in \mathcal{S}_\tau$
the set $\upsilon(u_\tau(A))$ is not contained in any set of the form $L + \delta_\tau B_X$ with $L \sub X$ weakly compact.

{\sc Step~5.} We can assume that there is $\delta>0$ such that for every $\tau\in \mathfrak{T}\setminus\{[0],[^1{}_0]\}$
and every $A\in \mathcal{S}_\tau$, we have
$\upsilon(A)\not \subseteq L + \delta B_X$ for any weakly compact set $L \sub X$.
To check this, enumerate $\mathfrak{T}\setminus\{[0],[^1{}_0]\}=\{\tau_0,\dots,\tau_5\}$.
By the former steps applied to~$\tau_0$, there exist a nice embedding $u_0:2^{<\omega}\to 2^{<\omega}$ and $\delta_0>0$
such that, for every $A\in \mathcal{S}_{\tau_0}$, we have $\upsilon(u_0(A))\not \subseteq L + \delta_0 B_X$
for any weakly compact set $L \sub X$.
Now, we can apply again Step~3 or~4 to the composition $\upsilon\circ u_0:2^{<\omega}\to B_X$
and type~$\tau_1$ to obtain a nice embedding $u_1:2^{<\omega}\to 2^{<\omega}$ and $\delta_1>0$ such that, for every $A\in \mathcal{S}_{\tau_1}$,
we have $\upsilon(u_0(u_1(A))) \not \subseteq L + \delta_1 B_X$ for any weakly compact set $L \sub X$.
By continuing in this manner, we obtain a function $\tilde{\upsilon}:2^{<\omega}\to B_X$
with the same properties as~$\upsilon$ and a constant~$\delta>0$ such that,
for every $\tau\in \mathfrak{T}\setminus\{[0],[^1{}_0]\}$ and every $A\in \mathcal{S}_\tau$, we have
$\tilde{\upsilon}(A)\not \subseteq L + \delta B_X$ for any weakly compact set $L \sub X$.

{\sc Step~6.} Finally, we shall check that $\cK(2^{<\omega})\preceq \mathcal{AK}(B_X)$.
This will finish the proof since $2^{<\omega}$ is homeomorphic to~$\mathbb{Q}$.
Define
$$
	f:\mathcal{K}(2^{<\omega})\to \mathcal{AK}(B_X),
	\quad f(A):=\overline{\upsilon(A)}^{w}.
$$
We claim that for every $L\in\mathcal{AK}(B_X)$ there is $B\in\mathcal{K}(2^{<\omega})$ containing
any $A\in \mathcal{K}(2^{<\omega})$ for which $f(A)\sub  L + \delta B_X$. Indeed,
it suffices to check that
$$
	B=\overline{\upsilon^{-1}(L+\delta B_X)}
$$
is compact in~$2^{<\omega}$, which is equivalent to saying (by the metrizability of $2^{<\omega}$) that every infinite subset of
$\upsilon^{-1}(L+\delta B_X)$ contains a further infinite subset which is relatively compact in~$2^{<\omega}$.
Take an infinite set $C \sub \upsilon^{-1}(L+\delta B_X)$. By property~(T2), there is an infinite set~$D \sub C$
of some type~$\tau$. From Step~5 and the fact that $\upsilon(D) \sub L+\delta B_X$ it follows that
$\tau\in\{[0],[^1{}_0]\}$. Property~(T3) ensures that $D$ is relatively compact in~$2^{<\omega}$
and the proof is over.
\end{proof}

\section{Banach spaces not containing $\ell_1$}\label{section:nol1}

This section is entirely devoted to proving the following result, that holds in ZFC without need of additional axioms:

\begin{thm}\label{separabledual}
If $X$ is separable and contains no copy of~$\ell_1$, then $\mathcal{AK}(B_X)\sim \mathcal{K}(B_X)$. More precisely:
\begin{enumerate}
\item[(i)] $\mathcal{AK}(B_X)\sim \mathcal{K}(B_X) \sim \{0\}$ if $X$ is reflexive.
\item[(ii)] $\mathcal{AK}(B_X)\sim \mathcal{K}(B_X) \sim \omega^\omega$ if $X$ is not reflexive, has separable dual and the PCP.
\item[(iii)] $\mathcal{AK}(B_X)\sim \mathcal{K}(B_X) \sim \mathcal{K}(\mathbb{Q})$ if $X$ has separable dual and fails the PCP.
\item[(iv)] $\mathcal{AK}(B_X)\sim \mathcal{K}(B_X) \sim [\mathfrak{c}]^{<\omega}$ if $X$ has non-separable dual.
\end{enumerate}
\end{thm}

We already know from Proposition~\ref{at:7}(i) that $\mathcal{AK}(B_X)\preceq \mathcal{K}(B_X)$ holds for any Banach space~$X$
and we also know from Proposition~\ref{posets:3} the classification of $\mathcal{K}(B_X)$ when $X^*$ is separable.
Since case (i) is trivial and $\mathcal{K}(B_X)\preceq [\mathfrak{c}]^{<\omega}$ whenever $X$ is separable (by Remark~\ref{rem:CardinalTukey}),
in order to prove Theorem~\ref{separabledual} it remains to check:
\begin{itemize}
\item $\omega^\omega\preceq \mathcal{AK}(B_X)$ if $X$ has separable dual and is not reflexive.
\item $\mathcal{K}(\mathbb{Q})\preceq \mathcal{AK}(B_X)$ if $X$  has separable dual and fails the PCP.
\item $[\mathfrak{c}]^{<\omega}\preceq \mathcal{AK}(B_X)$ if $X$ has non-separable dual and contains no copy of~$\ell_1$.
\end{itemize}
We will prove these facts in Propositions~\ref{ZFCww}, \ref{ZFCKQ} and \ref{ZFCtop} below.
Note that Theorem~\ref{separabledual} implies, in particular, that
if $X$ is separable and contains no copy of~$\ell^1$, then
$\cf(\AK(B_X))\in \{{\mathfrak d},\con\}$
and $\add_\omega(\AK(B_X))\in \{\omega_1,{\mathfrak b}\}$ unless $X$ is reflexive.

As we noted in Proposition~\ref{at:7}(i), we have $\mathcal{AK}(E) \preceq \cK(E)$ for every $E \sub X$.
Lemma~\ref{discretelemma} below shows that under some assumption we have the converse reduction $\cK(E) \preceq \mathcal{AK}(B_X)$.
This lemma will be useful in the sequel.

\begin{lem}\label{discretelemma}
Let $E\subset B_X$. Suppose there is $\delta>0$ such that for every $A\subset E$ we have
\begin{enumerate}
\item[(i)] either $A$ is contained in a weakly compact subset of $E$,
\item[(ii)] or $A\not\subset L + \delta B_X$ for any $L\in \mathcal{K}(X)$.
\end{enumerate}
Then $\mathcal{K}(E) \preceq \mathcal{AK}(B_X)$. In particular, this Tukey reduction holds
if there is $\delta>0$ such that
\begin{equation}\label{eqn:dE}
	\|x^{**}-x\|>\delta \quad\mbox{for every }x^{**}\in \overline{E}^{w^*} \setminus E
	\mbox{ and every }x\in X.
\end{equation}
\end{lem}
\begin{proof}
We shall check that the function $f:\mathcal{K}(E)\to \mathcal{AK}(B_X)$ given by $f(K):=K$
is a Tukey reduction from $\cK(E)$ to~$(\AK(B_X),\leq_\delta)$.
Fix $L\in \mathcal{AK}(B_X)$ and define $L_0 := E \cap (L+\delta B_X)$. Then
$L_0$ is contained in a weakly compact subset of~$E$.
Since $L_0$ is weakly closed in~$E$, we conclude that $L_0$ is weakly compact.
Obviously, for every $K\in \mathcal{K}(E)$ satisfying $K\sub L+\delta B_X$ we have
$K\sub L_0$. This proves that $\mathcal{K}(E) \preceq \mathcal{AK}(B_X)$.
For the last statement, note that if $A \sub E \cap (L + \delta B_X)$ for some $L\in \mathcal{K}(X)$,
then $\overline{A}^{w^*} \sub \overline{E}^{w^{*}}\cap (L + \delta B_{X^{**}})$;
therefore,~\eqref{eqn:dE} implies that $\overline{A}^{w^*} \sub E$ and so $\overline{A}^{w^*}\in \cK(E)$. The proof is over.
\end{proof}

\begin{pro}\label{ZFCww}
If $X^*$ is separable and $X$ is not reflexive, then $\omega^\omega\preceq \mathcal{AK}(B_X)$.
\end{pro}
\begin{proof}
Since $X^\ast$ is separable, $B_{X^{\ast\ast}}$ is metrizable in the $w^\ast$-topology. Let $\rho$ be a metric on~$B_{X^{**}}$
that metrizes the $w^*$-topology. Fix $0<\theta<1$. By the non-reflexivity of~$X$
and Riesz's Lemma (see e.g. \cite[p.~2]{die-J}), there is $x^{\ast\ast}\in X^{\ast\ast}$ with $\|x^{\ast\ast}\|=1$
such that $\|x^{\ast\ast}-x\|>\theta$ for all $x\in X$.
Since $B_X$ is $w^*$-dense in~$B_{X^{**}}$, we can find a sequence $(x_n)$ in~$B_X$ that $w^*$-converges to~$x^{\ast\ast}$.
For each $n<\omega$, we define $y_n := \frac{1}{2}(x^{\ast\ast}-x_n)\in B_{X^{\ast\ast}}$ and we choose
a sequence $(y_{m}^{n})$ in~$B_X$ such that
\begin{equation}\label{eqn:doubleseq}
	\rho(y_{m}^n,y_n)\leq \frac{1}{2^{n+m}}
	\quad\mbox{for every }m<\omega.
\end{equation}
We claim that the set $E := \{y_{m}^n : n,m<\omega\}\cup\{0\} \sub B_X$ satisfies the equality
\begin{equation}\label{eqn:equalityE}
	\overline{E}^{w^\ast} = E\cup \{y_n : n<\omega\}.
\end{equation}
Indeed, let $(z_k)$ be a sequence in~$E$ that $w^*$-converges to some $y\in \overline{E}^{w^*}$. If $z_k=0$ only
for finitely many $k$'s, then by passing to a subsequence we may assume that
each $z_k$ is of the form $z_k=y^{n_k}_{m_k}$ for some $n_k,m_k<\omega$. Then there is
a further subsequence of~$(z_k)$, not relabeled, such that:
\begin{itemize}
\item[(a)] either $n_k<n_{k+1}$ for all $k<\omega$,
\item[(b)] or there is $n<\omega$ such that $n_k=n$ for all $k<\omega$.
\end{itemize}
If~(a) holds, then
$$
	\rho(y_{n_k},y) \leq
	\rho(y_{n_k},z_k)+\rho(z_k,y) \stackrel{\eqref{eqn:doubleseq}}{\leq}
	\frac{1}{2^{n_k+m_k}}+\rho(z_k,y) \leq
	\frac{1}{2^{n_k}}+\rho(z_k,y)
	\quad \mbox{for all }k<\omega,
$$
hence $\rho(y_{n_k},y) \to 0$ and so $y=0 \in E$ (because $(y_n)$ is $w^*$-convergent to~$0$).
If~(b) holds, then $z_k=y^n_{m_k}$ for all $k<\omega$ and so either
$y\in E$ or $y=y_n$. This proves~\eqref{eqn:equalityE}.

Note that~$E$ satisfies the hypotheses of Lemma~\ref{discretelemma}, by equality~\eqref{eqn:equalityE}
and the fact that $\|y_n-x\|>\frac{\theta}{2}$ for every $n<\omega$ and every $x\in X$.
Hence $\mathcal{K}(E)\preceq \mathcal{AK}(B_X)$.

It only remains to check that $\omega^\omega \preceq \mathcal{K}(E)$. Note that for every $\varphi\in \omega^\omega$ the
sequence $(y^{n}_{\varphi(n)})$ is weakly null, because $(y_n)$ is $w^*$-convergent to~$0$ and
$$
	\rho(y^{n}_{\varphi(n)}, y_n) \stackrel{\eqref{eqn:doubleseq}}{\leq}
	\frac{1}{2^{n+\varphi(n)}}
	\leq \frac{1}{2^{n}}\quad
	\mbox{for all }n<\omega.
$$
We claim that the function
$$
	F:\omega^{\omega}\to \mathcal{K}(E),
	\quad
	F(\varphi) :=
	\{y^{n}_{\varphi(n)} : n<\omega\} \cup \{0\},
$$
is Tukey. Indeed, fix $K\in \mathcal{K}(E)$. For each $n<\omega$,
the set $\{m<\omega: y^{n}_{m}\in K\}$ is finite, because $y^n_m \to y_n \in X^{\ast\ast}\setminus X$ in the
$w^*$-topology as $m\to \infty$. Thus, there is $\varphi_K\in \omega^\omega$
such that $y^{n}_{m}\not\in K$ for every $n<\omega$ and every $m> \varphi_K(n)$. Hence
$\varphi\leq \varphi_K$ whenever $F(\varphi) \sub K$. This proves that $F$ is a Tukey function.
\end{proof}

\begin{pro}\label{ZFCKQ}
If $X^*$ is separable and $X$ fails the PCP, then $\mathcal{K}(\mathbb{Q})\preceq \mathcal{AK}(B_X)$.
\end{pro}
\begin{proof} We divide the proof into several steps.

{\sc Step~1.} Since $X$ contains no copy of~$\ell^1$ and fails the PCP, according to
\cite[Theorem~2.4]{lop-sol} there exist $\varepsilon>0$ and a set $\{x_s : s\in \omega^{<\omega}\}\sub X$ such that,
for every $s\in \omega^{<\omega}$, we have:
$$
	\|x_s\|> \varepsilon,
	\quad
	\Bigl\|\sum_{t \sqsubseteq  s}x_t\Bigr\| \leq 1
	\quad\mbox{and}\quad
	x_{s\smallfrown n} \to 0 \ \mbox{ weakly as }n\to \infty.
$$
For each $s\in \omega^{<\omega}$ we define $y_s := \sum_{t\sqsubseteq s}x_t \in B_X$. Note that
for every $s\in \omega^{<\omega}$ the sequence $(y_{s\smallfrown n})$ is weakly convergent to~$y_s$, while
$\|y_{s\smallfrown n}-y_s\|>\varepsilon$ for all~$n<\omega$.

Since $X^\ast$ is separable, $B_{X^{\ast\ast}}$ is metrizable in the $w^\ast$-topology. Let $\rho$ be a metric on~$B_{X^{**}}$
that metrizes the $w^*$-topology. For every $s=(s_0,\dots,s_k)\in \omega^{<\omega}$ we write
$$
	\Delta(s) := 2^{-(k+1+\sum_{i=0}^k s_i)}.
$$
Clearly, we can construct recursively an injective function $u: \omega^{<\omega} \to \omega^{<\omega}$ such that,
for each $s\in \omega^{<\omega}$, there is a strictly increasing $\varphi_s\in \omega^\omega$ such that
$$
	u(s \smallfrown n)=u(s) \smallfrown \varphi_s(n)
	\quad\mbox{and}\quad
	\rho(y_{u(s \smallfrown n)},y_{u(s)})\leq \Delta(s\smallfrown n)
	\quad \mbox{for every }n<\omega.
$$
Therefore, replacing the family $\{y_s:s\in \omega^{<\omega}\}$ by $\{y_{u(s)}:s\in \omega^{<\omega}\}$ if necessary,
we can assume without loss of generality that
\begin{equation}\label{eqn:star}
	\rho(y_{s\smallfrown n},y_s) \leq \Delta(s\smallfrown n)
	\quad \mbox{for every }s\in\omega^{<\omega} \mbox{ and }n<\omega.
\end{equation}

{\em Claim~1.} For every $\sigma\in \omega^\omega$ the sequence $(y_{\sigma|_n})$ is $w^\ast$-convergent to some $y_\sigma^{\ast\ast}\in B_{X^{\ast\ast}}$.
Indeed, since $(B_{X^{**}},\rho)$ is complete, it suffices to check that $(y_{\sigma|_n})$ is $\rho$-Cauchy, and this clearly follows from
the inequality
$$
	\rho(y_{\sigma|_{n+1}},y_{\sigma|_n}) \stackrel{\eqref{eqn:star}}{\leq} \Delta(\sigma|_{n+1})
	\leq\frac{1}{2^{n+1}}
	\quad
	\mbox{for all }n\in \N.
$$

{\em Claim~2.} The function $\omega^\omega \to B_{X^{**}}$ given by $\sigma\mapsto y_\sigma^{\ast\ast}$ is
continuous from the natural topology of~$\omega^\omega$ to the $w^\ast$-topology of~$B_{X^{\ast\ast}}$.
Indeed, given $s \in \omega^{<\omega}$ and $u=(u_0,\dots,u_k)\in \omega^{<\omega}$, inequality
\eqref{eqn:star} yields
\begin{equation}\label{eqn:DeltaExtension}
	\rho(y_{s\smallfrown u},y_s) \leq \sum_{i=0}^k \Delta(s\smallfrown (u_0,\dots,u_i))
	\leq
	\sum_{i=0}^k \frac{\Delta(s)}{2^{i+1}}
	\leq
	\Delta(s).
\end{equation}
It follows that
$\rho(y_{\sigma}^{**},y_{\sigma|_n})\leq \Delta(\sigma|_n) \leq 2^{1-n}$ for every $\sigma \in \omega^{\omega}$
and every $n<\omega$.
This implies that for every $\eps>0$ there is $n<\omega$ such that $\rho(y_{\sigma}^{**},y_{\tau}^{**})\leq \epsilon$
whenever $\sigma,\tau\in \omega^\omega$ satisfy $\sigma|_n=\tau|_n$. Thus,
the mapping $\sigma\mapsto y_\sigma^{\ast\ast}$ is $w^*$-continuous.

{\sc Step~2.} For each $n\in \N$ we define
$$
	A_n := \Bigl\{\sigma\in\omega^\omega : \,
	d\big(y^{\ast\ast}_\sigma,X\big)\geq \frac{1}{n}
	\Bigr\}
$$
and we also set
$$
	A_0 := \{\sigma\in\omega^\omega : \, y^{\ast\ast}_\sigma\in X\}=\omega^\omega \setminus \bigcup_{n\in \N}A_n.
$$
Note that $A_n$ is Borel for every $n< \omega$. Indeed, to check this for $n\geq 1$, it suffices to apply
Lemma~\ref{lem:Measurability} and the $w^*$-continuity of the mapping $\sigma \mapsto y_{\sigma}^{**}$.

Now, we are going to apply some Ramsey-theoretic principles from \cite[Section~7.2]{tod-J-3}.
We fix a non-principal ultrafilter $\mathcal{U}$ on~$\omega$. For coherence with the setting in
which our reference~\cite{tod-J-3} is written, let us consider the subtree
$\omega^{[<\omega]}\sub \omega^{<\omega}$  consisting of all strictly increasing finite sequences in~$\omega$
(which is naturally identified with the set $\mathbb{N}^{[<\infty]}$ of all finite subsets of~$\N$).
Let $\omega^{[\omega]}\sub \omega^\omega$ be the closed set of all strictly increasing infinite sequences in~$\omega$
(which is naturally identified with the set $\mathbb{N}^{[\infty]}$ of all infinite subsets of~$\N$).
For every $m<\omega$ we define $A'_m := A_m \cap \omega^{[\omega]}$. Since $A'_m$ is Borel, we
can apply \cite[Theorem~7.42 and Lemma~7.36]{tod-J-3} to conclude that $A'_m$ is $\mathcal{U}$-Ramsey.
Since $\omega^{[\omega]} = \bigcup_{m<\omega} A'_m$ and the family of $\mathcal{U}$-null
sets is a proper $\sigma$-ideal of subsets of~$\omega^{[\omega]}$ (see \cite[Lemma 7.41]{tod-J-3}),
there is $m<\omega$ such that $A'_m$ is not $\mathcal{U}$-null.
It follows from \cite[Definitions~7.37 and~7.39]{tod-J-3}
that there is a \emph{$\mathcal{U}$-tree} $\Upsilon\sub \omega^{[<\omega]}$ such that
\begin{equation}\label{eqn:upsilon}
	[\Upsilon] := \{\sigma\in \omega^{[\omega]} : \, \sigma|_n \in \Upsilon \mbox{ for every }n<\omega\}\sub A'_m.
\end{equation}
That $\Upsilon$ is a $\mathcal{U}$-tree means that it is a $\sqsubseteq$-downwards closed subtree of~$\omega^{[<\omega]}$
such that, for every $s\in \Upsilon$, the set $\{n<\omega: s\smallfrown n\in \Upsilon\}$ belongs to~$\mathcal{U}$.

{\sc Step~3.} Define $E := \{y_s : s\in\Upsilon\} \sub B_X$. Note that
$E$ equipped with the weak topology is a countable metrizable space without
isolated points. Indeed, given $s\in\Upsilon$ and a weakly open set~$V \sub X$ with~$y_s \in V$, the
fact that $(y_{s\smallfrown n})$ is weakly convergent to~$y_s$ ensures the existence of $n_0<\omega$ such that
$y_{s\smallfrown n}\in V$ for all $n\geq n_0$. Since $\Upsilon$ is a $\mathcal{U}$-tree,
the set $\{n\geq n_0: s\smallfrown n\in \Upsilon\}$ is infinite, so $E \cap V$ is infinite
as well. It follows that $E$ is homeomorphic to $\mathbb{Q}$.

We claim that
\begin{equation}\label{eqn:lancia}
	\overline{E}^{w^\ast} = E \cup \{y_\sigma^{\ast\ast} : \, \sigma\in [\Upsilon]\}.
\end{equation}
Indeed, the inclusion ``$\supseteq$'' being obvious, let us prove ``$\subseteq$''.
Fix a sequence $(z_k)$ in~$E$ that $w^*$-converges to some $y\in \overline{E}^{w^*}$,
and write $z_k=y_{s^k}$ for some $s^k\in \Upsilon$.
If $(z_k)$ admits a constant subsequence, then $y\in E$. Otherwise,
by passing to a subsequence, we can assume that $z_k\neq z_{k'}$ whenever $k\neq k'$.
It is not difficult to check that there is a subsequence
of~$(z_k)$, not relabeled, satisfying one of the following conditions:
\begin{itemize}
\item Condition~1: there exist $s\in \Upsilon$ and $\varphi\in \omega^{[\omega]}$ such that $s\smallfrown \varphi(k) \sqsubseteq s^k$
for every~$k<\omega$. Then
\begin{multline*}
	\rho(y_{s^k},y_s)\leq
	\rho(y_{s^k},y_{s\smallfrown \varphi(k)})+
	\rho(y_{s\smallfrown \varphi(k)},y_s) \leq \\
	\stackrel{\eqref{eqn:DeltaExtension}}{\leq}
	\Delta(s\smallfrown \varphi(k))+\rho(y_{s\smallfrown \varphi(k)},y_s) \leq
	\frac{1}{2^{\varphi(k)}}+\rho(y_{s\smallfrown \varphi(k)},y_s)
	\quad
	\mbox{for every }k<\omega,
\end{multline*}
hence $(y_{s^k})$ is weakly convergent to~$y_s$ and so $y=y_s$.

\item Condition~2: there exist $\sigma\in [\Upsilon]$ such that $\sigma|_k \sqsubseteq s^k$ for every $k<\omega$.
Then
\begin{multline*}
	\rho(y_{s^k},y_\sigma^{**})\leq
	\rho(y_{s^k},y_{\sigma|_k})+
	\rho(y_{\sigma|_k},y_\sigma^{**}) \leq \\
	\stackrel{\eqref{eqn:DeltaExtension}}{\leq}
	\Delta(\sigma|_k)+\rho(y_{\sigma|_k},y_\sigma^{**}) \leq
	\frac{1}{2^{k}}+\rho(y_{\sigma|_k},y_\sigma^{**})
	\quad
	\mbox{for every }k<\omega,
\end{multline*}
therefore $(y_{s^k})$ is $w^*$-convergent to~$y^{**}_\sigma$ and
so $y=y^{\ast\ast}_\sigma$.
\end{itemize}
This completes the proof of~\eqref{eqn:lancia}.

{\sc Step~4.}
If $m>0$, then $E$ satisfies the requirements of Lemma~\ref{discretelemma}
(by \eqref{eqn:upsilon} and~\eqref{eqn:lancia}), and we conclude that $\mathcal{K}(\mathbb{Q}) \sim \mathcal{K}(E)\preceq \mathcal{AK}(B_X)$.
Therefore, in order to finish the proof it remains to check that $m>0$. Our proof is by contradiction.
If $m=0$, then \eqref{eqn:upsilon} and~\eqref{eqn:lancia} yield $\overline{E}^{w^\ast}\sub X$ and hence $E$ is relatively weakly compact.
Therefore, $E$ is fragmented by the norm (see \cite{nam1}), so there is
a weakly open set $V \sub X$ such that $V\cap E\neq \emptyset$ and $\|y-y'\|\leq \epsilon$ for every $y,y'\in V\cap E$.
Arguing as in the proof that $E$ has no isolated points (Step~3), we find $s\in \Upsilon$ and $n<\omega$ such that $y_s,y_{s\smallfrown n}\in V \cap E$,
hence $\|y_s- y_{s\smallfrown n}\|\leq\varepsilon$, a contradiction. The proof is over.
\end{proof}

From now on we consider a topological space $J_3$ whose underlying set is $2^{<\omega} \cup 2^\omega$ and whose topology
is defined by:
\begin{itemize}
\item all points from $2^{<\omega}$ are isolated;
\item any $x\in 2^{\omega}$ has a neighborhood basis made of the sets $\{x\} \cup \{x|_k : k>n\}$, where $n<\omega$.
\end{itemize}
Let $K_3 := 2^{<\omega} \cup 2^\omega \cup \{\infty\}$ be its one-point compactification and
consider its subspace $L_3 := 2^{<\omega}\cup \{\infty\}$. The following lemma provides
a characterization of the compact subsets of these topological spaces.
Recall that
an {\em antichain} of~$2^{<\omega}$ is a subset made up of pairwise incomparable elements.

\begin{lem}\label{lem:compactL3}
\mbox{ }
\begin{enumerate}
\item[(i)] A set $C \sub J_3$ is relatively compact if and only if $C\cap 2^\omega$ is finite and $C$ contains no
infinite antichain of~$2^{<\omega}$.
\item[(ii)] An infinite set $C\subset L_3$ is compact if and only if
$\infty \in C$ and $C\cap \{x|_n : n<\omega\}$ is finite for every $x\in 2^{\omega}$.
\end{enumerate}
\end{lem}
\begin{proof} (i). Suppose first that $C$ is relatively compact. Since
$$
	\Bigl\{\{t\}: \, t\in \overline{C}\cap 2^{<\omega}\Bigr\}
	\cup
	\Bigl\{\{x\} \cup \{x|_k : k<\omega\}:\, x\in \overline{C}\cap 2^{\omega}\Bigr\}
$$
is an open cover of~$\overline{C}$, it admits a finite subcover, that is, there exist
finite sets $C_1 \sub \overline{C}\cap 2^{<\omega}$ and $C_2 \sub \overline{C}\cap 2^{\omega}$ such that
$$
	\overline{C} \sub C_1 \cup C_2 \cup \{x|_k : \, x\in C_2, \, k<\omega\}.
$$
Clearly, this implies that $\overline{C}\cap 2^\omega=C_2$ and that $\overline{C}$
contains no infinite antichain of~$2^{<\omega}$.

For the converse, suppose that $C\cap 2^\omega$ is finite and $C\cap 2^{<\omega}$ does not contain any infinite antichain. 
First we notice that $\overline{C}\cap 2^\omega$ is also finite. Otherwise we could find an infinite sequence
$\{x_n: n<\omega\}\sub \overline{C} \cap 2^\omega\setminus C$. 
We can suppose that $(x_n)$ is convergent to some $x\in 2^\omega$ in the usual topology of $2^\omega$ and that
$x_n\neq x$ for all $n<\omega$. Moreover, writing $k[n]:=\min\{k<\omega: x_n|_k \neq x|_k\}$ for all $n<\omega$, 
we can suppose that $k[1]< k[2] < k[3] < \dots$. 
For every $n<\omega$ we have $x_n\in \overline{C}\cap 2^\omega \setminus C$ and so there is $p[n] > k[n]$ such that $x_n|_{p[n]}\in C$. 
Clearly, the set $\{x_n|_{p[n]} : n<\omega\}$ is an infinite antichain of~$2^{<\omega}$ contained in~$C$, 
a contradiction. So $\overline{C}\cap 2^\omega$ is finite. 

Now take an open cover $\mathcal{U}$ of $\overline{C}$. Take $\mathcal{V}\sub \mathcal{U}$ finite that covers 
$\overline{C}\cap 2^\omega$. We claim that $D: = \overline{C}\setminus \bigcup\mathcal{V} \sub 2^{<\omega}$ 
is finite, and hence $\mathcal{U}$ has a finite subcover. Our proof is by contradiction.
If $D$ was infinite, then by Ramsey's theorem (see e.g. \cite[Theorem~9.1]{jec}) 
it would contain either an infinite chain or an infinite antichain. 
It cannot contain an infinite antichain because $D\sub \overline{C}\cap 2^{<\omega} = C\cap 2^{<\omega}$. 
On the other hand, if $D$ contained an infinite chain, then that chain 
(ordered in the natural way) would converge to some $x\in \overline{C} \cap 2^\omega$, but this is impossible because $\overline{C}\cap 2^\omega$ 
is contained in the open set $\bigcup\mathcal{V}$ disjoint from~$D$. This contradiction
finishes the proof that $\overline{C}$ is compact.

(ii). Let $C \sub L_3$ be a compact set. If $\infty \not \in C$, then $C$ is finite because points of $L_3\setminus\{\infty\}$ are isolated.
Take any $x\in 2^{\omega}$. Since the sequence $(x|_n)$ converges to~$x$ in~$K_3$ and $C$ is a compact subset of~$L_3$, the
set $C\cap \{x|_n : n<\omega\}$ is finite.

Conversely, let $C\sub L_3$ be a set such that $\infty\in C$ and
$C\cap \{x|_n : n<\omega\}$ is finite for every $x\in 2^\omega$.
We shall prove that $K_3\setminus C$ is open in~$K_3$ (hence $C$ is compact). Clearly,
if $t\in 2^{<\omega}\setminus C$, then its open neighborhood $\{t\}$ does not meet~$C$. On the other hand,
if $x\in 2^\omega \setminus C$, then there is $n<\omega$ such that
the open neighborhood of~$x$ given by $\{x\}\cup\{x|_k:k>n\}$ does not meet~$C$.
\end{proof}

\begin{lem}\label{KL3}
$\mathcal{K}(L_3) \sim [\mathfrak{c}]^{<\omega}$.
\end{lem}
\begin{proof}
Let $\cJ$ be the family of all sets $J\sub 2^{<\omega}$ such that $J\cap \{x|_n : n<\omega\}$ is finite for every $x\in 2^\omega$,
ordered by inclusion.

{\em Claim.} $\mathcal{K}(L_3) \sim \cJ$. Indeed, by Lemma~\ref{lem:compactL3}, we can define
$f:\mathcal{K}(L_3) \to \cJ$ and $g: \cJ \to \mathcal{K}(L_3)$ by
$$
	f(C):=C\cap 2^{<\omega}, \qquad
	g(J):=J \cup \{\infty\}.
$$
It is clear that both $f$ and~$g$ are Tukey functions, which proves the claim.

Therefore, it remains to show that $\cJ \sim [\mathfrak{c}]^{<\omega}$.
The Tukey reduction $\cJ \preceq [\mathfrak{c}]^{<\omega}$ follows
from Remark~\ref{rem:CardinalTukey}. Given $x=(x(n))\in 2^\omega$, we define
$$
	s_n^x:=x|_n \smallfrown (1-x(n))\in 2^{<\omega}
	\quad
	\mbox{for every }n<\omega
$$
and we write $A_x:=\{s_n^x:n<\omega\}$, so that $A_x\in\cJ$
(since it is an antichain). Define
$$
	f: [2^{\omega}]^{<\omega} \to \cJ,
	\quad
	f(F) := \bigcup_{x\in F} A_x.
$$
In order to check that $f$ is a Tukey function it is enough to show that for every
$J\in\cJ$ the set $H_J:=\{x\in 2^\omega:A_x\sub J\}$ is finite.
Our proof is by contradiction. If $H_J$ is infinite, then
by compactness there is a sequence $(x_j)$ in~$H_J$ converging to some $x\in 2^\omega$ in the usual topology of~$2^\omega$
such that $x\neq x_j$ for all $j<\omega$. Fix $n<\omega$. Then there is $j<\omega$ such that
$x_j|_n=x|_n$. Since $x_j\neq x$, there is $m\geq n$ such that $x_j|_m= x|_m$
and $x_j|_{m+1}\neq x|_{m+1}$, that is, $x|_{m+1}=s_m^{x_j}\in A_{x_j}\sub J$. As $n<\omega$
is arbitrary, we conclude that $J \cap \{x|_n:n<\omega\}$ is infinite, a contradiction.
\end{proof}

\begin{defi}\label{defi:K3function}
A function $\varphi: K_3 \to B_{X^{**}}$ is called a {\em $K_3$-embedding} if
\begin{enumerate}
\item[(i)] it is continuous from~$K_3$ to $(B_{X^{**}},w^*)$ and one-to-one;
\item[(ii)] $\varphi(\infty)=0$;
\item[(iii)] $\varphi(2^{<\omega}) \sub X$;
\item[(iv)] $\varphi(2^{\omega}) \sub X^{**} \setminus X$.
\end{enumerate}
A $K_3$-embedding $\varphi$ is called {\em regular} if the following condition holds:
\begin{enumerate}
\item[(v)] there is $\delta>0$ such that $\|\varphi(\sigma)-x\|> \delta$ for every $\sigma\in 2^\omega$ and every~$x\in X$.
\end{enumerate}
\end{defi}

\begin{lem}\label{lem:regularTukeyTop}
If $X$ admits a regular $K_3$-embedding, then $[\mathfrak{c}]^{<\omega} \preceq \AK(B_X)$.
\end{lem}
\begin{proof}
Let $\varphi:K_3 \to B_{X^{**}}$ be a regular $K_3$-embedding and consider $E:=\varphi(L_3) \sub B_X$.
Since $\varphi$ is a homeomophism between $K_3$ and $\varphi(K_3)$, we have
$\overline{E}^{w^*} = \varphi(K_3)$ and the restriction $\varphi|_{L_3}$
is a homeomorphism between $L_3$ and~$E$.
Let $\delta>0$ be as in Definition~\ref{defi:K3function}(v). Then
$\|x^{**}-x\|>\delta$ for every $x^{**}\in \overline{E}^{w^*} \setminus E$ and every $x\in X$.
By Lemmas~\ref{KL3} and~\ref{discretelemma}, we have $[\mathfrak{c}]^{<\omega} \sim \mathcal{K}(L_3) \sim \mathcal{K}(E) \preceq \mathcal{AK}(B_X)$.
\end{proof}

Let us illustrate the notion of $K_3$-embedding with an example.

\begin{exa}\label{exa:C01}
The space $C[0,1]$ satisfies $\mathcal{AK}(B_{C[0,1]})\sim \mathcal{K}(B_{C[0,1]}) \sim [\mathfrak{c}]^{<\omega}$.
\end{exa}
\begin{proof}
Since $C[0,1]$ and $C(2^\omega)$ are isomorphic, it suffices to deal with~$X:=C(2^\omega)$
(bear in mind Proposition~\ref{at:7}(v)).
Define $\vf:K_3 \to B_{X^{**}}$ by declaring:
\begin{itemize}
\item $\varphi(\infty):=0$;
\item $\varphi(u):=1_{A_u}$ for every $u\in 2^{<\omega}$, where $1_{A_u}$ is the characteristic function
of the clopen $A_u:=\{\sigma\in 2^\omega: u \sqsubseteq \sigma\}$;
\item $\varphi(\sigma)(\mu):=\mu(\{\sigma\})$ for every $\sigma\in 2^\omega$ and every $\mu\in X^{*}$;
here $X^*$ is identified with the space of all regular Borel measures on~$2^\omega$ via Riesz's theorem.
\end{itemize}
Clearly, $\vf$ is one-to-one and satisfies properties (ii), (iii) and (iv) of Definition~\ref{defi:K3function}.
Writing $\delta_\tau \in B_{X^{*}}$ to denote the point mass at~$\tau\in 2^\omega$,
we have
$$
	\|\varphi(\sigma)-f\| \geq 	\sup_{\tau\in 2^\omega} \, \bigl|\delta_\tau(\{\sigma\})-f(\tau)\bigr|=
	\max\bigl\{|1-f(\sigma)|,\, \sup_{\tau\neq \sigma}|f(\tau)|\bigr\} \geq \frac{1}{2}
$$
for every $\sigma\in 2^\omega$ and every $f\in C(2^\omega)$. Therefore,
$\vf$ fulfills condition~(v) of Definition~\ref{defi:K3function}. In order to prove that $\vf$ is a regular $K_3$-embedding,
it remains to check that $\vf$ is continuous from~$K_3$ to $(B_{X^{**}},w^*)$. Obviously, $\vf$ is continuous at each isolated point
$u\in 2^{<\omega}$, while the continuity at each~$\sigma \in 2^\omega$ follows from the fact that
$$
	\varphi(\sigma)(\mu)=\mu(\{\sigma\})=\lim_{n\to \infty}\mu(A_{\sigma|_n})=\lim_{n\to \infty}\vf(\sigma|_n)(\mu)
$$
for every $\mu\in X^{*}$ (note that $(A_{\sigma|_n})$ is a decreasing sequence of sets with intersection~$\{\sigma\}$).
To show that $\vf$ is continuous at~$\infty$, let $W \sub X^{**}$ be any $w^*$-open neighborhood of~$0$. Then
$C:=K_3 \setminus \varphi^{-1}(W)$ is relatively compact in~$J_3$, by Lemma~\ref{lem:compactL3}(i)
and the following facts:
\begin{itemize}
\item $C\cap 2^\omega$ is finite (bear in mind that $w^{*}-\lim_{n\to \infty}\varphi(\sigma_n)=0$
for every sequence $(\sigma_n)$ of distinct points of~$2^\omega$);
\item $C\cap 2^{<\omega}$ contains no infinite antichain; indeed, if $(u_n)$ is a sequence of
incomparable elements of~$2^{<\omega}$, then $(1_{A_{u_n}})$ is bounded and pointwise convergent to~$0$, hence
weakly null (see e.g. \cite[Corollary~3.138]{fab-ultimo}).
\end{itemize}
It follows that $\varphi^{-1}(W) \supseteq K_3 \setminus \overline{C}^{J_3}$ is a neighborhood of~$\infty$ in~$K_3$.
This proves that $\vf$ is continuous at~$\infty$ and we conclude that $\vf$ is a
regular $K_3$-embedding.

By Lemma~\ref{lem:regularTukeyTop}, we have $[\mathfrak{c}]^{<\omega} \preceq \AK(B_X)$.
On the other hand, the reductions $\AK(B_X) \preceq \cK(B_X) \preceq [\mathfrak{c}]^{<\omega}$
follow from Proposition~\ref{at:7}(i) and the separability of~$X$ (see Remark~\ref{rem:CardinalTukey}).
\end{proof}

\begin{lem}\label{lem:FromSimpleToRegular}
If $X$ is separable and admits a $K_3$-embedding, then it also admits a regular $K_3$-embedding.
In particular, $\AK(B_X)\sim \cK(B_X) \sim [\mathfrak{c}]^{<\omega}$.
\end{lem}
\begin{proof}
Let $\varphi:K_3 \to B_{X^{**}}$ be a $K_3$-embedding. For each $n\in \N$, define
$$
	A_n := \left\{\sigma\in 2^\omega : \,  d\big(\varphi(\sigma),X\big)\geq \frac{1}{n}\right\}
$$
and observe that $A_n$ is Borel (by Lemma~\ref{lem:Measurability} and the $w^*$-continuity of~$\varphi$).
Since $2^\omega = \bigcup_{n\in \N} A_n$, there is $n\in \N$ such that $A_n$ is uncountable.
It follows that $A_n$ contains a set~$P$ homeomorphic to~$2^\omega$
(see e.g. \cite[Theorem~13.6]{kec-J}).
By Lemma~\ref{lem:PerfectSubtree}, there is a subtree $u:2^{<\omega}\to 2^{<\omega}$
(as in Definition~\ref{defi:subtree}) such that
$$
	u(2^{<\omega}) \sub T:=\{\sigma|_n: \, \sigma\in P, \, n<\omega\}.
$$

{\em Claim~1.} For every $\sigma\in 2^\omega$ there exists the limit
$j(\sigma):=\lim_{n\to \infty}u(\sigma|_n)$ in~$K_3$ and $j(\sigma)\in P$.
Indeed, since $u(\sigma|_n) \sqsubseteq u(\sigma|_{n+1})$ for all $n<\omega$, there
exist $\tau\in 2^\omega$ and a strictly increasing sequence $(m_n)$ in~$\omega$
such that $\tau|_{m_n}=u(\sigma|_n)$ for all $n<\omega$. Clearly, this implies
that the sequence $(u(\sigma|_n))$ converges to~$\tau$ in~$K_3$. On the other hand,
for every $n<\omega$ there is $\sigma_n\in P$ such that $\sigma_n|_{m_n}=u(\sigma|_n)=\tau|_{m_n}$.
Therefore, the sequence $(\sigma_n)$ converges to~$\tau$ in~$2^\omega$ and, since~$P$ is
closed in~$2^\omega$, we conclude that $\tau\in P$.

Define $\tilde{\varphi}: K_3 \to B_{X^{**}}$ as follows:
\begin{itemize}
\item[(i)] $\tilde{\varphi}(\infty):=0$;
\item[(ii)] $\tilde{\varphi}(t):=\varphi(u(t))$ for every $t\in 2^{<\omega}$;
\item[(iii)] $\tilde{\varphi}(\sigma):=\varphi(j(\sigma))$ for every $\sigma\in 2^\omega$.
\end{itemize}

{\em Claim~2.} $\tilde{\varphi}$ is continuous from $K_3$ to $(B_{X^{**}},w^*)$. Indeed, obviously
$\tilde{\varphi}$ is continuous at each isolated point $t\in 2^{<\omega}$. On the other hand,
for every $\sigma\in 2^\omega$ we have
$$
	\tilde{\varphi}(\sigma)=\varphi\Bigl(\lim_{n\to \infty}u(\sigma|_n)\Bigr)
	=\lim_{n\to \infty}\varphi(u(\sigma|_n))=
	\lim_{n\to \infty}\tilde{\varphi}(\sigma|_n),
$$
which implies the continuity of~$\tilde{\varphi}$ at~$\sigma$. Finally, in order
to check the continuity at~$\infty$, let $W \sub X^{**}$ be a $w^*$-open neighborhood of~$0$. Then
$K:=K_3\setminus \varphi^{-1}(W) \sub J_3$ is compact and so
$K\cap 2^\omega$ is finite and $K$ contains no infinite antichain of~$2^{<\omega}$ (by Lemma~\ref{lem:compactL3}(i)).
Define
$$
	C:=\{t\in 2^{<\omega}:\, u(t)\in K\}\cup \{\sigma\in 2^\omega: \, j(\sigma)\in K\} \sub J_3
$$
and note that $C=K_3\setminus \tilde{\varphi}^{-1}(W)$.
Since $j: 2^\omega \to P$ is one-to-one, $C \cap 2^\omega$ is finite. Note that
$C$ contains no infinite antichain of~$2^{<\omega}$ (because
if $t,s\in C$ are incomparable, then $u(t),u(s) \in K$ are incomparable as well).
Another appeal to Lemma~\ref{lem:compactL3}(i) ensures that~$C$
is relatively compact in~$J_3$. Then $K_3 \setminus \overline{C}^{J_3}$
is an open neighborhood of~$\infty$ in~$K_3$ contained in~$\tilde{\varphi}^{-1}(W)$.
This finishes the proof of the claim.

Thus, $\tilde{\varphi}$ is a $K_3$-embedding. Since
$j(\sigma)\in P \sub A_n$ for every $\sigma\in 2^\omega$,
it follows that $\tilde{\varphi}$ is regular. From Lemma~\ref{lem:regularTukeyTop}
we get $[\mathfrak{c}]^{<\omega} \preceq \AK(B_X)$.
On the other hand, the reductions $\AK(B_X) \preceq \cK(B_X) \preceq [\mathfrak{c}]^{<\omega}$
follow from Proposition~\ref{at:7}(i) and the separability of~$X$.
\end{proof}

Recall that the dual ball $B_{Y^*}$ of a Banach space~$Y$ is is \emph{$w^*$-angelic} if, for any
set $A \subset B_{Y^*}$, every point in the $w^*$-closure of~$A$ is the $w^*$-limit of a sequence
contained in~$A$. This property holds if $Y$ is WCG (cf. \cite[Theorem~13.20]{fab-ultimo}) and also
if $Y$ is the dual of a separable Banach space not containing~$\ell^1$, thanks to the Bourgain-Fremlin-Talagrand
and Odell-Rosenthal theorems (cf. \cite[Theorems~5.49 and~5.52]{fab-ultimo}).

\begin{lem}\label{lem:UsingDodos}
Suppose that $B_{X^{**}}$ is $w^*$-angelic and that there is
a biorthogonal system $\{(e_t,e^*_t):t\in 2^{<\omega}\} \sub X \times X^*$ satisfying the following properties:
\begin{enumerate}
\item[(i)] $\|e_t\|\leq 1$ for every $t\in 2^{<\omega}$;
\item[(ii)] $\overline{\{e_t:t\in I\}}^{w^*} = \{e_t:t\in I\} \cup \{0\}$ for every infinite antichain~$I \sub 2^{<\omega}$;
\item[(iii)] for every $\sigma\in 2^\omega$ there exists $w^*-\lim_{n\to \infty} e_{\sigma|_n}=e_\sigma^{**}\in X^{**}\setminus X$;
\item[(iv)] $e_\sigma^{**}\neq e_\tau^{**}$ whenever $\sigma\neq \tau$.
\end{enumerate}
Then $E:=\{e_\sigma^{**}:\sigma\in 2^\omega\}$ is $w^{*}$-discrete and $\overline{E}^{w^*}=E\cup \{0\}$.
\end{lem}
\begin{proof} We first note that $E\sub B_{X^{**}}$ and
\begin{equation}\label{eqn:ClosureE}
	\overline{E}^{w^*}\cap \{e_t:t\in 2^{<\omega}\}=\emptyset.
\end{equation}
Indeed, note that for each $t\in 2^{<\omega}$ the set $\{x^{**}\in X^{**}: x^{**}(e_t^*)>\frac{1}{2}\}$
is a $w^*$-open neighborhood of~$e_t$ which is disjoint from~$E$, by property~(iii).

{\em Claim~1.} $\overline{E}^{w^*} \sub E\cup \{0\}$. To check this, note that
condition~(iii) ensures that $\overline{E}^{w^*} \sub \overline{\{e_t:t\in 2^{<\omega}\}}^{w^*}$.
Take any $x^{**}\in \overline{E}^{w^*}$. By the $w^*$-angelicity of~$B_{X^{**}}$, there is
a sequence $(t_n)$ in~$2^{<\omega}$ such that $(e_{t_n})$ is $w^*$-convergent to~$x^{**}$.
Since~$x^{**}\neq e_t$ for all~$t\in 2^{<\omega}$ (by~\eqref{eqn:ClosureE}), by passing to a further subsequence
we can assume that ${\rm length}(t_n)<{\rm length}(t_{n+1})$ for all $n<\omega$. Ramsey's theorem (see e.g. \cite[Theorem~9.1]{jec})
ensures that the set~$\{t_n:n<\omega\}$ contains either an infinite antichain or an infinite chain.
In the first case, condition~(ii) and~\eqref{eqn:ClosureE} imply that $x^{**}=0$. In the second case,
there exist $\sigma\in 2^\omega$, a subsequence~$(t_{n_k})$ and a strictly increasing sequence~$(m_k)$ in~$\omega$
such that $t_{n_k}=\sigma|_{m_k}$ for all~$k<\omega$, hence~(iii) yields $x^{**}=e_\sigma^{**}$. The claim is proved.

{\em Claim~2.} $E$ is $w^{*}$-discrete. Our proof is by contradiction. Suppose there is $\sigma_0\in 2^\omega$
such that $e_{\sigma_0}^{**}\in \overline{\{e_\sigma^{**}:\, \sigma\in 2^\omega\setminus \{\sigma_0\}\}}^{w^*}$.
For each $\sigma \in 2^\omega\setminus \{\sigma_0\}$ we fix $n_\sigma<\omega$ such that
$\sigma|_{n_\sigma} \neq \sigma_0|_{n_\sigma}$. Note that
$$
	e_{\sigma_0}^{**}\in \overline{\{e_\sigma^{**}:\, \sigma\in 2^\omega\setminus \{\sigma_0\}\}}^{w^*}
	\stackrel{{\rm (iii)}}{\sub} \overline{\{e_{\sigma|_n}:\, \sigma\in 2^\omega\setminus \{\sigma_0\},\, n\geq n_\sigma\}}^{w^*}.
$$
Arguing as in the proof of Claim~1 (bearing in mind that $e_{\sigma_0}^{**}\neq 0$),
there exist $\sigma \in 2^\omega$, a sequence~$(\tau_{k})$ in~$2^\omega \setminus \{\sigma_0\}$ and a strictly increasing
sequence~$(m_k)$ in~$\omega$ with $m_k\geq n_{\tau_k}$ such that
$w^*-\lim_{k\to \infty} e_{\tau_k|_{m_k}}^{**}=e_{\sigma_0}^{**}$
and $\tau_k|_{m_k}=\sigma|_{m_k}$ for all~$k<\omega$. By~(iii) we get $e_{\sigma_0}^{**}=e_\sigma^{**}$
and therefore~(iv) yields $\sigma_0=\sigma$. This contradicts that $\tau_k|_{n_{\tau_k}}\neq\sigma_0|_{n_{\tau_k}}$
for all $k<\omega$ and the claim is proved.

It only remains to prove that $0\in \overline{E}^{w^*}$. To this end, let $(\sigma_n)$ be a sequence
of distinct elements of~$2^\omega$. Since $(B_{X^{**}},w^*)$ is sequentially compact (because
it is an angelic compact), there exist $x^{**}\in B_{X^{**}}$
and a subsequence $(\sigma_{n_k})$ with $w^{*}-\lim_{k\to \infty}e_{\sigma_{n_k}}^{**}=x^{**}$. We can assume
that $x^{**}\neq e_{\sigma_{n_k}}^{**}$ for all $k<\omega$.
By Claim~1, we have $x^{**}\in E \cup \{0\}$. Since $E$ is $w^{*}$-discrete (Claim~2),
we conclude that $x^{**}=0$ and so $0\in \overline{E}^{w^{*}}$. This finishes the proof of the lemma.
\end{proof}

\begin{pro}\label{ZFCtop}
If $X$ is separable, has non-separable dual and contains no copy of~$\ell_1$, then
$\mathcal{AK}(B_X) \sim \cK(B_X) \sim [\mathfrak{c}]^{<\omega}$.
\end{pro}
\begin{proof}
As we mentioned before Lemma~\ref{lem:UsingDodos}, since $X$ is separable and
contains no copy of~$\ell^1$, the ball $B_{X^{**}}$ is $w^*$-angelic.
Theorem~3 in~\cite{dod} (applied to the identity operator on~$X$)
ensures the existence of a biorthogonal system $\{(e_t,e^*_t):t\in 2^{<\omega}\} \sub X \times X^*$
as in Lemma~\ref{lem:UsingDodos}. Define a function $\varphi: K_3 \to B_{X^{**}}$ by
\begin{itemize}
\item $\varphi(\infty):=0$;
\item $\varphi(t):=e_t$ for every $t\in 2^{<\omega}$;
\item $\varphi(\sigma):=e_\sigma^{**}$ for every $\sigma\in 2^\omega$.
\end{itemize}
We next prove that $\varphi$ is continuous from~$K_3$ to $(B_{X^{**}},w^*)$.
Clearly, $\varphi$ is continuous at each point of~$2^{<\omega}\cup 2^\omega$.
To check that $\varphi$ is continuous at~$\infty$,
take any $w^{*}$-open set $W\sub X^{**}$ containing~$0$. Then:
\begin{enumerate}
\item[(a)] $(K_3 \setminus \varphi^{-1}(W)) \cap 2^{\omega}$ is finite. Indeed, it suffices
to note that $0\in \overline{\{e_{\sigma_n}^{**}:n<\omega\}}^{w^*}$ for every sequence $(\sigma_n)$ of
distinct elements of~$2^\omega$ (see the end of the proof of Lemma~\ref{lem:UsingDodos}).
\item[(b)] $(K_3 \setminus \varphi^{-1}(W)) \cap 2^{<\omega}$ contains no infinite antichain. Indeed,
this follows at once from property~(ii) in Lemma~\ref{lem:UsingDodos}.
\end{enumerate}
Conditions (a) and~(b) imply (use Lemma~\ref{lem:compactL3}(i)) that
$C:=K_3 \setminus \varphi^{-1}(W)$ is relatively compact in~$J_3$.
Then $K_3 \setminus \overline{C}^{J_3}$ is an open neighborhood of~$\infty$ in~$K_3$ contained in~$\varphi^{-1}(W)$.
This proves that $\varphi$ is continuous at~$\infty$.

Therefore, $\varphi$ is a $K_3$-embedding. By Lemma~\ref{lem:FromSimpleToRegular},
$\mathcal{AK}(B_X) \sim \cK(B_X) \sim [\mathfrak{c}]^{<\omega}$.
\end{proof}

The proof of Proposition~\ref{ZFCtop} shows that $K_3$ embeds in a natural way
into~$(B_{X^{**}},w^*)$ whenever $X$ is separable, has non-separable dual and contains no copy of~$\ell_1$.
This statement is related to the fact that non-$G_\delta$-points in Rosenthal compacta lie in the closure of a discrete set
of size continuum \cite{tod-J}.

\section{Unconditional bases}\label{section:unconditional}

Let us recall that an infinite countable set $\mathcal{B}$ of non-zero vectors in the Banach space~$X$ is called
an \emph{unconditional basic sequence} if
there exists $C>0$ such that for every finite sets $G\subset F \subset\mathcal{B}$ and
every $\alpha:F\to \mathbb{R}$ we have
$$
	\left\| \sum_{x\in G}\alpha(x)\, x\right\| \leq  C \left\| \sum_{x\in F}\alpha(x)\, x\right\|.
$$
If in addition $\overline{{\rm span}}(\mathcal{B}) = X$, then
$\mathcal{B}$ is called an {\em unconditional basis} of~$X$.

The interest of unconditional bases in this theory is that the partially ordered set $\mathcal{RK}(\mathcal{B})$
(made up of all subsets of~$\mathcal{B}$ which are relatively weakly compact, ordered by inclusion)
is combinatorially easier to analyze than $\mathcal{K}(B_X)$ and $\mathcal{AK}(B_X)$,
but yet it can provide information on these structures as Lemma~\ref{lem:Johnson} below shows.

We collect in the following lemma some well-known characterizations of relatively weakly compact
subsets of unconditional basic sequences.

\begin{lem}\label{lem:RKunconditional}
Let $\mathcal{B}$ be an unconditional basic sequence in~$X$. Then for any
infinite set $A\subset \mathcal{B}$ the following
statements are equivalent:
\begin{enumerate}
\item[(i)] $A\in \mathcal{RK}(\mathcal{B})$;
\item[(ii)] $A$ is weakly null, that is, $A$ converges to $0$ in the weak topology;
\item[(iii)] $A$ is weakly Cauchy, that is, $A$ converges in the weak$^\ast$ topology of~$X^{\ast\ast}$;
\item[(iv)] $A$ is bounded and contains no subsequence equivalent to the basis of~$\ell^1$.
\end{enumerate}
\end{lem}

The proof of the following lemma uses arguments by Johnson, Mercourakis
and Stamati~\cite{mer-sta-2} (cf. \cite[Theorems~7.40 and~7.41]{fab-alt-JJ}),
that we adapted to fit into our setting.

\begin{lem}\label{lem:Johnson}
If $\mathcal{B}=\{e_n:n<\omega\}$ is a semi-normalized unconditional basis of~$X$,
then $\mathcal{RK}(\mathcal{B}) \preceq \mathcal{AK}(B_X)$.
\end{lem}
\begin{proof}
We begin by proving the following:

{\em Claim.} If $K \sub X$ is weakly compact and $\epsilon>0$, then the set
$$
	\Gamma(K,\epsilon):=\Big\{e_n \in \mathcal{B}: \, \sup_{x\in K}|e_n^*(x)|\geq \epsilon\Big\} \in \mathcal{RK}(\mathcal{B}).
$$
Our proof is by contradiction. If $\Gamma(K,\epsilon)\not\in \mathcal{RK}(\cB)$, then
there is a strictly increasing sequence $(n_j)$ in~$\omega$ such that $e_{n_j} \in \Gamma(K,\epsilon)$
for all $j<\omega$ and $(e_{n_j})$ is equivalent to the basis of~$\ell^1$
(apply Lemma~\ref{lem:RKunconditional}). Let $(x_j)$ be a sequence in~$K$ such that
\begin{equation}\label{limit2}
	|e_{n_j}^*(x_j)|\geq \epsilon \quad\mbox{for all }j<\omega.
\end{equation}
Since~$\mathcal{B}$ is
an unconditional basis of~$X$, we can consider the (bounded and linear) operator
$$
	P:X \to X, \quad
	P(x):=\sum_{j<\omega} e_{n_j}^*(x) e_{n_j},
$$
which is a projection onto~$Y:=\overline{{\rm span}}\{e_{n_j}:j<\omega\}$.
Since $K$ is weakly compact and $Y$~has the Schur property (because it is isomorphic to~$\ell^1$),
the set $P(K)$ is norm compact. Thus, by passing to a further subsequence, we may assume
that there is $x\in K$ such that $\|P(x_j)-P(x)\|\to 0$. Bearing in mind that $\{e_n^*:n<\omega\}$
is bounded (because $\cB$ is an unconditional basis of~$X$ and $\inf_{n<\omega}\|e_n\|>0$),
we conclude that
\begin{equation}\label{limit1}
	|e_{n_j}^*(x_j)-e_{n_j}^*(x)|=|e_{n_j}^*(P(x_j)) - e_{n_j}^*(P(x))| \to 0.
\end{equation}
On the other hand, the convergence of the series $\sum_{j<\omega} e_{n_j}^*(x) e_{n_j}$
and the fact that $\inf_{j<\omega}\|e_{n_j}\|>0$ imply
that $e_{n_j}^*(x) \to 0$, which combined with~\eqref{limit1} yields
$e_{n_j}^*(x_j) \to 0$. This contradicts~\eqref{limit2} and the {\em Claim} is proved.

Define
$$
	\delta:=(2\sup_{n<\omega}\|e_n^*\|)^{-1}>0.
$$
Set $\rho:=\sup_{n<\omega}\|e_n\|$. Finally, we prove that the function
$$
	f:\mathcal{RK}(\mathcal{B}) \to (\mathcal{AK}(\rho B_X),\leq_\delta),
	\quad f(A):=\overline{A}^w,
$$
is Tukey. Indeed, it suffices to prove if $K$ is a weakly compact subset of~$\rho B_X$
and $A\in \mathcal{RK}(\mathcal{B})$ satisfies $\overline{A}^w \sub K+\delta B_X$, then
$A \sub \Gamma(K,\frac{1}{2})$. We argue by contradiction. Suppose there is $e_n\in A$
such that $|e_{n}^*(x)| < \frac{1}{2}$ for every~$x\in K$. Write $e_n = x+y$, where $x\in K$ and $\|y\|\leq \delta$. Then
$$
	1=e_{n}^*(e_n) =e_n^*(x)+e_n^*(y) < \frac{1}{2}+\|e_n^*\|\delta \leq
	\frac{1}{2} + \frac{1}{2}=1,
$$
a contradiction which proves that $f$ is Tukey. Therefore, $\mathcal{RK}(\cB) \preceq \mathcal{AK}(\rho B_X)$.
The proof finishes by using that $\mathcal{AK}(\rho B_X) \sim \mathcal{AK}(B_X)$, which can be
deduced either from Proposition~\ref{at:7}(v) or, simply, by
taking into account that the spaces with closed unit balls $B_X$ and $\rho B_X$ are isometric.
\end{proof}

\begin{rem}
\rm The equivalence $\mathcal{RK}(\mathcal{B}) \sim \mathcal{AK}(B_X)$ does not hold in general
for an unconditional basis $\mathcal{B}$ of~$X$. Indeed, the usual basis~$\cB$
of~$c_0$ is weakly null and so it satisfies $\mathcal{RK}(\mathcal{B}) \sim \{0\} \not\sim
\cK(\mathbb{Q})\sim \mathcal{AK}(B_{c_0})$ (Example~\ref{exa:c0l1}).
\end{rem}

As a first application of Lemma~\ref{lem:Johnson}, we compute
$\mathcal{AK}(B_X)$ and $\cK(B_X)$ for the space $X=\ell^p(\ell^1)$, where $1<p<\infty$.

\begin{exa}\label{exa:l2l1}
The space $X=\ell^p(\ell^1)$, $1<p<\infty$, satisfies $\mathcal{AK}(B_X)\sim \cK(B_X) \sim \omega^\omega$.
\end{exa}
\begin{proof}
It is easy to see that a set $C \sub B_X$ is relatively weakly compact if and only
if $\pi_n(C)$ is relatively norm compact in~$\ell^1$ for every $n<\omega$, where $\pi_n:X \to \ell^1$ denotes the $n$-th coordinate projection
(combine the Schur property of~$\ell^1$ with the fact that, for $1<p<\infty$, weak convergence in the ball of an $\ell^p$-sum of Banach spaces
coincides with coordinate-wise weak convergence).
Therefore, if $\cB=\{e_{nm}:n,m<\omega\}$ is the usual unconditional basis of~$X$, then a set $C \sub \cB$ belongs to~$\mathcal{RK}(\cB)$
if and only if $C$ is contained in a set of the form
$$
	F(\varphi):=\{e_{nm}: \, n<\omega, \, m < \vf(n)\}
$$
for some $\vf\in \omega^\omega$. Now, it is easy to check that
the function $F:\omega^\omega \to \mathcal{RK}(\cB)$ is Tukey,
hence $\omega^\omega \preceq \mathcal{RK}(\cB)$. An appeal to Lemma~\ref{lem:Johnson}
yields $\omega^\omega \preceq \mathcal{AK}(B_X)$.

On the other hand, note that for every sequence $(L_n)$ of norm compact subsets of~$B_{\ell^1}$, the set
$\{x\in B_X: \, \pi_n(x)\in L_n \mbox{ for all }n<\omega\}$ is weakly compact. Therefore, the function
$$
	G: \cK(B_X) \to \bigl(\cK(B_{\ell^1})\bigr)^\omega,
	\quad
	G(L):=(\pi_n(L)),
$$
is Tukey and so $\cK(B_X) \preceq (\cK(B_{\ell^1}))^\omega \sim \omega^\omega$
(recall that $\cK(B_{\ell^1}) \sim \omega^\omega$, by Example~\ref{posets:4}).
Since $\mathcal{AK}(B_X) \preceq \cK(B_X)$ (Proposition~\ref{at:7}(i)), we have
$\mathcal{AK}(B_X)\sim \cK(B_X) \sim \omega^\omega$.
\end{proof}

Recall that if $X$ is SWCG and $Y\sub X$ is a non-reflexive subspace, then
$\mathcal{K}(B_Y)\sim\omega^\omega$ (Corollary~\ref{as:4-cor}(iv))
and $\AK(B_Y)\sim \omega$ under the additional assumption that $Y$ is complemented in~$X$
(combine Theorem~\ref{swcg:2} and the fact that the SWCG property
is inherited by complemented subspaces).
Mercourakis and Stamati constructed in \cite[Theorem 3.9(ii)]{mer-sta-2}
a subspace $Y\sub L^1[0,1]$ which is not SWCG. {By Corollary~\ref{as:4-cor}(v), such subspace
satisfies $\AK(B_Y)\sim \omega^\omega$.} The following theorem uses essentially the same construction
{of~\cite{mer-sta-2} and is included to show another application of Lemma~\ref{lem:Johnson}.}

\begin{thm}\label{mer_sta}
There is a subspace $Y$ of $L^1[0,1]$ such that $\AK(B_Y)\sim \omega^\omega$.
\end{thm}

To deal with the proof of Theorem~\ref{mer_sta} we need some lemmas.

\begin{lem}\label{lem:functions}
Let $\varphi\in \omega^\omega$. Then there is $\tilde{\varphi}\in \omega^\omega$ such that:
\begin{enumerate}
\item[(i)] $\varphi \leq \tilde{\varphi}$;
\item[(ii)] $\tilde{\varphi}$ is strictly increasing;
\item[(iii)] $n\tilde{\varphi}(m)\geq (m+1)\tilde{\varphi}(n)$ for every $m>n\geq 1$;
\item[(iv)] $\tilde{\varphi}(n) \geq (n+1)(n+2)$ for every $n<\omega$.
\end{enumerate}
\end{lem}
\begin{proof}
Just define $\tilde{\varphi}$ inductively by taking
$$
	\tilde{\varphi}(n) \geq \max
	\Big\{
	\varphi(n),\tilde{\varphi}(n-1)+1,(n+1)(n+2), \frac{n+1}{m}\tilde{\varphi}(m) \mbox{ for }1\leq m<n
	\Big\}.
$$
\end{proof}

The straightforward proof of the following lemma is omitted.
We denote by~$\lambda$ the Lebesgue measure on the Borel $\sigma$-algebra of~$[0,1]$.

\begin{lem}\label{lem:functionsG}
Let $m,p\in \N$ with $m\geq p > 2$. Set $c:=\frac{1}{m}+\frac{1}{2}-\frac{1}{p} \in (\frac{1}{m},\frac{1}{2}]$
and $\alpha:=\frac{1}{2(1-c)} \in (\frac{1}{2},1]$. Then the function $f\in L^1[0,1]$ defined by
$$
	f(t):=\begin{cases}
	\frac{m}{p} & \text{if $t\in [0,\frac{1}{m}]$} \\
	1 & \text{if $t\in [\frac{1}{m},c]$} \\
	-\alpha & \text{if $t\in (c,1]$}
	\end{cases}
$$
satisfies $\int_0^1 f \, d\lambda =0$ and $\int_0^1 |f| \, d\lambda =1$.
\end{lem}

We shall use the following notation as in the proof of Example~\ref{posets:4}. Given any probability space $(\Omega,\Sigma,\mu)$, any function
$f\in L^1(\mu)$ and $n<\omega$, we denote by $o(f,n)$ the least $k<\omega$ such that, for every $B\in\Sigma$, we have:
\[
	\mbox{if } \mu(B)\le \frac{1}{k+1} \ \mbox{ then } \left| \int_{B} f\, d\mu\right|\le \frac{1}{n+1}.
\]

\begin{lem}\label{lem:functionsGG}
Let $\varphi\in \omega^\omega$ and $n\in\N$, $n\geq 2$.
Let $\tilde{\varphi} \in \omega^\omega$ be as in Lemma~\ref{lem:functions} and
let $f\in L^1[0,1]$ be as in Lemma~\ref{lem:functionsG} by taking $m=\tilde{\varphi}(n)+1$
and $p=n+1$. Then:
\begin{enumerate}
\item[(i)] $o(f,k)\leq \tilde{\varphi}(k)$ for every $k<\omega$;
\item[(ii)] $o(f,n)= \tilde{\varphi}(n)$.
\end{enumerate}
\end{lem}
\begin{proof} We divide the proof into several steps.

{\em Step~1.} $o(f,k)\leq \tilde{\varphi}(k)$ for every $k \geq n$.
Indeed, if $B \sub [0,1]$ is a Borel set with $\lambda(B)\leq \frac{1}{\tilde{\varphi}(k)+1}$, then
\begin{equation}\label{eqn:G-1}
	\Bigl|\int_{B} f \, d\lambda\Bigr| \leq \frac{\tilde{\varphi}(n)+1}{n+1}\cdot \lambda(B) \leq
	\frac{\tilde{\varphi}(n)+1}{n+1}\cdot \frac{1}{\tilde{\varphi}(k)+1}.
\end{equation}
If $k=n$, then~\eqref{eqn:G-1} yields $|\int_{B} f \, d\lambda| \leq \frac{1}{n+1}$, while
if $k>n$ then property~(iii) in Lemma~\ref{lem:functions} and~\eqref{eqn:G-1} imply that
$$
	\Bigl|\int_{B} f \, d\lambda\Bigr| \leq
	\frac{\tilde{\varphi}(n)+1}{n+1}\cdot \frac{1}{\tilde{\varphi}(k)+1} <
	\frac{\tilde{\varphi}(n)+1}{n+1}\cdot
	\frac{n}{(k+1)\tilde{\varphi}(n)}\leq  \frac{1}{k+1}.
$$
This shows that $o(f,k)\leq \tilde{\varphi}(k)$ whenever $k\geq n$.

{\em Step~2.} $o(f,n)= \tilde{\varphi}(n)$. Indeed, fix $k< \tilde{\varphi}(n)$.
Set $c:=\frac{1}{\tilde{\varphi}(n)+1}+\frac{1}{2}-\frac{1}{n+1} \in (\frac{1}{\tilde{\varphi}(n)+1},\frac{1}{2}]$
and choose $d \in \erre$ such that
$$
	\frac{1}{\tilde{\varphi}(n)+1}<d < \min\Bigl\{\frac{1}{k+1},c\Bigr\}.
$$
By the definition of~$f$, we have
$$
	\Bigl|\int_0^d f \, d\lambda\Bigr|=\frac{\tilde{\varphi}(n)+1}{n+1}\cdot \frac{1}{\tilde{\varphi}(n)+1}
	+ \Bigl(d-\frac{1}{\tilde{\varphi}(n)+1}\Bigr) > \frac{1}{n+1},
$$
hence $k>o(f,n)$. As $k< \tilde{\varphi}(n)$ is arbitrary, we get $o(f,n)=\tilde{\varphi}(n)$.

{\em Step~3.} $o(f,k)\leq \tilde{\varphi}(k)$ for every $k < n$. Indeed,
let $B \sub [0,1]$ be a Borel set with $\lambda(B)\leq \frac{1}{\tilde{\varphi}(k)+1}$.
Define $B_0:=B\cap [0,\frac{1}{\tilde{\varphi}(n)+1}]$. Then property~(iv) in Lemma~\ref{lem:functions} yields
\begin{multline*}
	\Bigl|\int_{B}f \, d\lambda \Bigr|\leq \Bigl|\int_{B_0}f \, d\lambda \Bigr|
	+\Bigl|\int_{B\setminus B_0}f \, d\lambda \Bigr| \leq \\ \leq
	\frac{1}{n+1} + \lambda(B) \leq \frac{1}{n+1}+\frac{1}{\tilde{\varphi}(k)+1}
	<
	\frac{1}{k+2}+\frac{1}{(k+1)(k+2)}=\frac{1}{k+1}.
\end{multline*}
It follows that $o(f,k)\leq \tilde{\varphi}(k)$. The proof is over.
\end{proof}

\begin{proof}[Proof of Theorem~\ref{mer_sta}]
Let $\cR_0=\{r_n:n<\omega\} \sub L^1[0,1]$ be the set made up of all functions as in Lemma~\ref{lem:functionsG}.
We shall now work in the space $X:=L^1([0,1]^\omega)$, which is isometric to $L^1[0,1]$.
For each $n<\omega$, let $f_n:[0,1]^\omega\to\er$ be given by $f_n:=r_n\circ\pi_n$, where $\pi_n:[0,1]^\omega \to [0,1]$
is the $n$-th coordinate projection. Then $\BB:=\{f_n:n<\omega\}$ is a (normalized) unconditional
basic sequence in~$X$ (see the proof of \cite[pp. 89-90]{ros2} and the references therein).
Let $Y:=\overline{{\rm span}}(\BB) \sub X$. By Lemma~\ref{lem:Johnson}, Proposition~\ref{at:7}(i)
and Corollary~\ref{as:4-cor}(iv) we have
\[
	\mathcal{RK}(\BB)\preceq \AK(B_Y)\preceq \cK(B_Y)\preceq \omega^\omega.
\]
So, in order to complete the proof we need to show that $\omega^\omega \preceq \mathcal{RK}(\BB)$.

As we already pointed out in Example~\ref{posets:4}, a bounded set
$C\sub X$ is relatively weakly compact if and only if $\{o(f,\cdot):f\in C\}$ is bounded above in~$\omega^\omega$.
Therefore, we can consider the function
$$
	F: \omega^\omega \to \mathcal{RK}(\BB),
	\quad
	F(\varphi):=\{f_n: \, o(f_n,\cdot) \leq \tilde{\varphi}\},
$$
where $\tilde{\varphi}$ is the function given by Lemma~\ref{lem:functions}.
We claim that $F$ is Tukey when $\omega^\omega$ is equipped with the binary relation $\leq_2$ defined by
$$
	\varphi \leq_2 \psi \quad :\Longleftrightarrow \quad \varphi(n) \leq \psi(n) \ \mbox{ for every }n\geq 2.
$$
Indeed, fix $C_0\in \mathcal{RK}(\BB)$ and take
$\varphi_0\in \omega^\omega$ such that $o(f_n,\cdot)\leq \vf_0$ for all $f_n \in C_0$.
We shall check that $\vf \leq_2 \vf_0$ whenever $\vf\in \omega^\omega$
satisfies $F(\vf)\sub C_0$. Take any $n\in\N$, $n\geq 2$. Let $r_j\in \cR_0$
be the function given by Lemma~\ref{lem:functionsG} by taking $m=\tilde{\vf}(n)+1$ and $p=n+1$.
By Lemma~\ref{lem:functionsGG}, we have $o(f_j,\cdot)=o(r_j,\cdot)\le \tilde{\vf}$ and $o(f_j,n)=o(r_j,n)=\tilde{\vf}(n)$.
Therefore, $f_j\in F(\vf)\sub C_0$ and so $\vf(n)\le\tilde{\vf}(n)\le \vf_0(n)$, which proves the claim.

Now, in order to finish the proof it suffices to check that $(\omega^\omega,\leq_2) \sim \omega^\omega$.
On one hand, the identity mapping $(\omega^\omega,\leq_2) \to \omega^\omega$ is obviously a Tukey function.
On the other hand, it is clear that the function $G:\omega^\omega \to (\omega^\omega,\leq_2)$
given by
$$	
	G(\varphi)(n):=\max\{\vf(0),\dots,\vf(n)\}
$$
is Tukey as well. The proof is over.
\end{proof}

We next explain a method to construct unconditional bases~$\BB$
for which $\mathcal{RK}(\mathcal{B})$ is Tukey equivalent to different posets. In this way, we shall get examples where
$\mathcal{RK}(\mathcal{B})$ is equivalent to $\{0\}$, $\omega$, $\omega^\omega$, $\mathcal{K}(\mathbb{Q})$
and $[\mathfrak{c}]^{<\omega}$. But more interesting, we shall also get a consistent example which is not equivalent to any of these
(see Theorem~\ref{omega1sequence} and the comments following it), thus showing that
Theorem~\ref{projectiveideal} and Corollary~\ref{cor:interrogacion} do not hold in general
in the absence of analytic determinacy.

Recall that a family $\mathcal{A}$ of subsets of a given countable infinite set $D$
is called {\em adequate} if, for any $A \sub D$, we have $A\in\mathcal{A}$ if and only if $[A]^{<\omega} \sub \mathcal{A}$. In this
case, we can define a norm on $c_{00}(D)$ (the linear space of all finitely supported real-valued functions on~$D$)
by the formula:
$$
	\|f\|_{\cA} := \sup\left\{\sum_{i\in T} |f(i)| : \, T\in \mathcal{A} \right\},
	\quad
	f\in c_{00}(D).
$$
In this way, the canonical Hamel basis of $c_{00}(D)$ becomes a (normalized) unconditional basis
$\mathcal{B}_{\mathcal{A}} = \{e_d:d\in D\}$ of the completion of $(c_{00}(D),\|\cdot\|_{\cA})$,
which we denote by~$\mathfrak{X}_{\cA}$. For more information on this space
(sometimes denoted by $E_{0,1}(\cA)$), we refer the reader to~\cite{arg-mer}.
Note that, as particular cases of this construction, we have:
\begin{itemize}
\item $\mathfrak{X}_{\cA}=c_0$ if $D=\omega$ and $\cA=\{\{n\}:n<\omega\}\cup\{\emptyset\}$;
\item $\mathfrak{X}_{\cA}=c_0(\ell^1)$ if $D=\omega\times \omega$ and $\cA$ is the family made up of all
sets of the form $\{(n,m):m\in F\}$, where $n<\omega$ and $F \sub \omega$.
\end{itemize}

\begin{lem}\label{nulladequate}
Let $\cA$ be an adequate family of subsets of a countable infinite set~$D$.
Let $C \sub D$. The following statements are equivalent:
\begin{enumerate}
\item[(i)] $C\in \mathcal{A}^\perp$ (i.e. $C\cap A$ is finite for every $A\in\cA$);
\item[(ii)] $\{e_d:d\in C\}\in \mathcal{RK}(\mathcal{B}_\mathcal{A})$.
\end{enumerate}
Therefore, $\mathcal{RK}(\mathcal{B}_\mathcal{A}) \sim \mathcal{A}^\perp$ (ordered by inclusion).
\end{lem}
\begin{proof} The equivalence (i)$\Leftrightarrow$(ii) is obvious if $C$ is finite. Assume
that $C$ is infinite and enumerate $C=\{d_n:n<\omega\}$. Bearing in mind Lemma~\ref{lem:RKunconditional},
(ii) is equivalent to saying that the sequence~$(e_{d_n})$ is weakly null.
The set $K_\cA:=\{1_A:A\in \cA\} \sub 2^D$ is closed, hence compact,
when $2^D$ is equipped with its usual product topology. It is well-known (and easy to check)
that $\mathfrak{X}_{\cA}$ embeds isomorphically into~$C(K_\cA)$ via an operator $T:\mathfrak{X}_\cA \to C(K_\cA)$
such that $T(e_d)$ is the $d$-th coordinate projection for all~$d\in D$. Since
a bounded sequence in~$C(K_\cA)$ is weakly null if and only if it is pointwise null
(see e.g. \cite[Corollary~3.138]{fab-ultimo}), we conclude that (ii) is equivalent to saying that,
for each $A\in \cA$, we have $T(e_{d_n})(1_A)=1_A(d_n)=0$ for $n$ large enough, that is, $C\cap A$ is finite.
The proof is finished.
\end{proof}

\begin{prop}
Let $\mathcal{A}$ be the adequate family of all
chains (including the finite ones) of $D =2^{<\omega}$.
Then $\mathcal{RK}(\mathcal{B}_\mathcal{A}) \sim [\mathfrak{c}]^{<\omega}$.
\end{prop}
\begin{proof}
Notice that $\mathcal{A}^\perp$ is the family of all subsets of $2^{<\omega}$ which do not contain any infinite chain. We
showed in the proof of Lemma~\ref{KL3} that this is Tukey equivalent to~$[\mathfrak{c}]^{<\omega}$.
The conclusion now follows from Lemma~\ref{nulladequate}.
\end{proof}

The following proposition, combined with Theorem~\ref{Fr91classification}, provides examples of unconditional bases~$\mathcal{B}$
for which $\mathcal{RK}(\mathcal{B})$ is Tukey equivalent to either $\{0\}$, $\omega$, $\omega^\omega$ or~$\mathcal{K}(\mathbb{Q})$.

If $E$ is a coanalytic subset of some Polish space, then \cite[Theorem 11]{avi2} asserts that there is
an adequate family $\mathcal{A}_E$ of closed and discrete subsets of~$E$ with the following
cofinality property:
for every infinite closed and discrete set $A\sub E$ there is an infinite set $B\sub A$ such that $B\in \mathcal{A}_E$.

\begin{prop}\label{fromadequate}
Let $E$ be a coanalytic subset of some Polish space. Let $D\sub E$ be a countable dense set
and consider the adequate family of subsets of~$D$ defined by $\mathcal{A}[E] := \{A\in\mathcal{A}_E : A\subset D\}$.
Then $\mathcal{RK}(\mathcal{B}_{\mathcal{A}[E]}) \sim \mathcal{K}(E)$.
\end{prop}
\begin{proof}
By Lemma~\ref{nulladequate} we have $\mathcal{RK}(\mathcal{B}_{\mathcal{A}[E]}) \sim \mathcal{A}[E]^\perp$.
On the other hand, we claim that $\mathcal{A}[E]^\perp$ coincides with the family~$\mathcal{K}_E(D)$
of all subsets of~$D$ which are relatively compact in~$E$. Indeed, a set $C \sub E$
is relatively compact in the metric space~$E$ if and only if $C$ contains no infinite closed and discrete set,
which is equivalent to saying that $C$ contains no infinite element of~$\cA_E$.
Therefore, a set $C \sub D$ belongs to~$\mathcal{K}_E(D)$ if and only if $C \in \mathcal{A}[E]^\perp$, as claimed.
It follows that $\mathcal{RK}(\mathcal{B}_{\mathcal{A}[E]}) \sim \mathcal{K}_E(D)$.
An appeal to Lemma~\ref{relativecompact} finishes the proof.
\end{proof}

\begin{thm}\label{omega1sequence}
If there exists a coanalytic set $E \sub 2^\omega$ of cardinality $\omega_1$, then there is a normalized
unconditional basis $\mathcal{B}$ of a Banach space
such that $\mathcal{RK}(\mathcal{B})\sim [\omega_1]^{<\omega}$.
\end{thm}
\begin{proof}
We can suppose that $E$ is dense in~$2^\omega$.
Let us consider the set $K := 2^{<\omega} \cup 2^\omega$ equipped with the compact metrizable topology
induced by the one-to-one mapping $f:K\to 2^\omega \times \mathbb{R}$ defined by
\begin{itemize}
\item $f(\sigma) := (\sigma,0)$ for all $\sigma\in 2^\omega$;
\item $f(t) := (t\smallfrown {\bf 0},({\rm length}(t)+1)^{-1})$ for all $t\in 2^{<\omega}$.
\end{itemize}
Note that the topology inherited by $2^\omega \sub K$ is the usual one.
Then $E' := 2^{<\omega} \cup E$ is a coanalytic subset of~$K$. Let $\mathcal{A}_{E'}$ be an
adequate family of closed and discrete subsets of~$E'$ such that,
for every infinite closed and discrete set $A\sub E'$, there is an infinite set $B\sub A$ such that $B\in \mathcal{A}_{E'}$
(apply \cite[Theorem 11]{avi2}). Of course, we can suppose that $\mathcal{A}_{E'}$ contains all singletons.
For every $\sigma\in 2^\omega$ we denote $c(\sigma) := \{\sigma|_n : n<\omega\}\sub 2^{<\omega}$.
We define an adequate family of subsets of~$2^{<\omega}$ by
$$
	\mathcal{A}:= \{A\subset 2^{<\omega} :\, A\cap c(\sigma) \in \mathcal{A}_{E'} \text{ for all }\sigma\in 2^\omega \}.
$$
We shall show that the unconditional basis $\mathcal{B}_{\mathcal{A}}$ is the one that we are looking for.
By Lemma~\ref{nulladequate} we have $\mathcal{RK}(\mathcal{B}_{\mathcal{A}}) \sim \mathcal{A}^\perp$, so it is enough
to prove that $\mathcal{A}^\perp \sim [E]^{<\omega}$.

We next check the following equality
\begin{equation}\label{eqn:AVILES}
	\mathcal{A}^\perp = \Bigl\{A\sub 2^{<\omega} : \, A\sub \bigcup_{\sigma\in F} c(\sigma) \mbox{ for some finite set } F \sub E \Bigr\}.
\end{equation}
For the inclusion ``$\supseteq$'' it suffices to prove that $c(\sigma)\in \cA^\perp$ for every $\sigma\in E$. Given any $A\in \cA$,
we have $A \cap c(\sigma) \in \cA_{E'}$ and so $A \cap c(\sigma)$ is closed in~$E'$. Since the sequence
$(\sigma|_n)$ converges to~$\sigma \in E' \setminus A\cap c(\sigma)$ in the topology of~$K$, the set $A\cap c(\sigma)$ is finite.

We divide the proof of the inclusion ``$\subseteq$'' in~\eqref{eqn:AVILES} into several steps. Fix $A\in \cA^\perp$.

{\em Step~1.} Let $S$ be the set of all $\sigma\in 2^\omega$ for which $A\cap c(\sigma)$ is infinite. We claim that $S$ is finite.
Our proof is by contradiction. If $S$ is infinite, then we can find a sequence $(\sigma^k)$ in~$S$ converging to some
$\sigma\in 2^\omega$ with $\sigma\neq \sigma^k$ for all $k<\omega$. Write
$$
	n_k:=\min\{n<\omega: \, \sigma^k(n)\neq \sigma(n)\}\quad \mbox{for every }k<\omega.
$$
By passing to a further subsequence, we can suppose that $n_k<n_{k+1}$
for all $k<\omega$. Now, we can pick $t^k \in A \cap c(\sigma^k)$ with ${\rm length}(t^k)>n_k$
(since $A\cap c(\sigma^k)$ is infinite)
for every $k<\omega$. Notice that $B:=\{t^k:k<\omega\} \sub A$ is an antichain,
because for every $k<l<\omega$ we have
$$
	t^l(n_k)=\sigma^l(n_k)=\sigma(n_k) \quad\mbox{and}\quad
	t^k(n_k)=\sigma^k(n_k)\neq \sigma(n_k).
$$
Therefore, $|B\cap c(\tau)|\leq 1$ for every $\tau\in 2^\omega$. Since all singletons of~$2^{<\omega}$
belong to~$\mathcal{A}_{E'}$, we have $B\in \mathcal{A}$, which contradicts that $A\in \mathcal{A}^\perp$, finishing the proof of Step~1.

{\em Step~2.} $S\sub E$. Indeed, suppose that there is $\sigma\in S\setminus E$. Since $A\cap c(\sigma)$ is infinite
and the sequence $(\sigma|_n)$ converges to~$\sigma \in K\setminus E'$, we have that $A\cap c(\sigma)$
is a closed and discrete subset of~$E'$. Therefore, there is an infinite set $B\sub A\cap c(\sigma)$ such that $B\in \mathcal{A}_{E'}$.
Bearing in mind that $\cA$ is hereditary, we conclude that
$B\in \mathcal{A}$, and this contradicts again that $A\in\mathcal{A}^\perp$.

{\em Step~3.} The set $B:=A\setminus \bigcup_{\sigma\in S}c(\sigma)$ is finite. Again, our proof
is by contradiction. If $B$ is infinite, then Ramsey's theorem (see e.g. \cite[Theorem~9.1]{jec})
ensures that $B$ contains either an infinite chain or an infinite antichain of~$2^{<\omega}$. The first case is not possible
(because $B\cap c(\sigma)$ is finite for every $\sigma\in 2^\omega$), so there exists an infinite antichain $B_0$
contained in~$B$. Since $|B_0\cap c(\sigma)|\leq 1$ for every $\sigma\in 2^\omega$,
we get $B_0\in\mathcal{A}$ (bear in mind that all singletons of~$2^{<\omega}$ belong to $\mathcal{A}_{E'}$). This is a contradiction,
because $B_0$ is infinite and $B_0 \sub B \sub A\in \cA^\perp$.

Finally, since $E$ is dense in~$2^\omega$, we have $2^{<\omega} \sub \bigcup_{\sigma\in E}c(\sigma)$. Let $S_1 \sub E$
be a finite set such that $B \sub \bigcup_{\sigma\in S_1}c(\sigma)$. Then $S\cup S_1$ is a finite subset of~$E$
such that $A\sub \bigcup_{\sigma\in S \cup S_1}c(\sigma)$. This finishes the proof of~\eqref{eqn:AVILES}.

Finally, note that equality~\eqref{eqn:AVILES} allows us to define a function
$f: [E]^{<\omega} \to \cA^\perp$ by $f(F):=\bigcup_{\sigma\in F}c(\sigma)$
and a function $g:\cA^\perp \to [E]^{<\omega}$ such that
$A \sub \bigcup_{\sigma\in g(A)}c(\sigma)$ for every $A\in \cA^\perp$. Clearly,
both $f$ and $g$ are Tukey functions, so $\cA^\perp \sim [E]^{<\omega}$ and the proof is over.
\end{proof}

There exists a model of set theory where the axioms $MA_{\aleph_1}$ and Lusin's hypothesis~\textbf{L}
(every subset of cardinality $\omega_1$ of a Polish space is coanalytic) both
hold~\cite{mar-sol}. In such a model, the hypothesis of Theorem~\ref{omega1sequence} holds, and moreover
$[\omega_1]^{<\omega}$ is not Tukey equivalent to any of $\{0\}$, $\omega$, $\omega^\omega$, $\mathcal{K}(\mathbb{Q})$
or $[\mathfrak{c}]^{<\omega}$, because the cofinality of~$[\omega_1]^{<\omega}$ equals~$\aleph_1$, while the other
posets have cofinality either $\aleph_0$, or $\mathfrak{d}$ or $\mathfrak{c}$ (see Section~\ref{section:KBX}),
but under $MA_{\aleph_1}$ we have $\aleph_1<\mathfrak{d}$.

\section{Open problems}\label{section:Problems}

In this final section we collect some questions which we were not able to answer.

\begin{problem}
Is it true that $\AK(X) \sim \AK(B_X)$ for every non-reflexive $X$?
\end{problem}

We showed in Proposition~\ref{KXKBX} that
$\omega^\omega$ was the lowest possible nontrivial value of the Tukey class of~$\mathcal{K}(B_X)$.
In the case of $\mathcal{AK}(B_X)$ we can get also~$\omega$, but still we might ask if $\omega^\omega$
is the lowest possible value after $\omega$.

\begin{problem}\label{ns:1}
Is it true that for every Banach space $X$, either
$\mathcal{AK}(B_X)\preceq\omega$ (i.e. $X$ is SWCG) or else $\omega^\omega\preceq \mathcal{AK}(B_X)$?
\end{problem}

We next give a consistent non-separable counterexample
using cardinal invariants, but the separable case remains open for us.

\begin{exa}\label{exa:non-sep-l1}
\rm The space $X=\ell^1(\omega_1)$ is not SWCG and so $\cf(\mathcal{AK}(B_X))\geq \omega_1$.
By Lemma~\ref{lem:cofinal-l1sum} we have $\mathcal{AK}(B_X) \preceq [\omega_1]^{<\omega}$, hence
$\cf(\mathcal{AK}(B_X))= \omega_1$. On the other hand, $\cf(\omega^\omega)={\mathfrak d}$, so $\omega^\omega \not\preceq \AK(B_X)$
whenever $\omega_1<{\mathfrak d}$.
\end{exa}

Note that Lemma~\ref{lem:LV} gives an affirmative answer to the previous question for separable spaces satisfying
$\mathcal{AK}(B_X)\preceq\omega^\omega$.

The main open questions related to this work focus on the validity of Theorems~\ref{PDKclassification} and~\ref{introAKBXclassification}
in the absence of analytic determinacy:

\begin{problem}\label{problem:main}
Is it relatively consistent that  there is a non-reflexive separable Banach space $X$ such that
$\cK(B_X)$ is neither Tukey equivalent to $\omega^\omega$ nor to $\cK(\qu)$ nor to $[\con]^{<\omega}$?
\end{problem}

\begin{problem}\label{problem:main2}
Is it relatively consistent that there is a non-reflexive separable Banach space $X$ such that
$\mathcal{AK}(B_X)$ is neither Tukey equivalent to $\omega^\omega$ nor to $\cK(\qu)$ nor to $[\con]^{<\omega}$?
\end{problem}

Note that such a Banach space would necessarily contain~$\ell^1$, by Theorem~\ref{separabledual}.
We believe that the contruction of Theorem~\ref{omega1sequence} should also provide a consistent affirmative answer
to Problems~\ref{problem:main} and~\ref{problem:main2}, but we were not able to prove it.

{\subsection*{Acknowledgements}
We are very grateful to the referee for a careful reading of the manuscript and his/her valuable comments, including several corrections.
We also thank Michael Hrusak for a valuable suggestion.
A. Avil\'{e}s and J. Rodr\'{i}guez were partially supported by
the research projects MTM2011-25377 and MTM2014-54182-P funded by {\em Ministerio de Econom\'{i}a y Competitividad - FEDER}
and the research project 19275/PI/14 funded by {\em Fundaci\'{o}n S\'{e}neca - Agencia de Ciencia y Tecnolog\'{i}a
de la Regi\'{o}n de Murcia} within the framework of {\em PCTIRM 2011-2014}.
G. Plebanek was partially supported by NCN grant 2013/11/B/ST1/03596 (2014-2017).}

\end{document}